\documentclass[11pt]{preprint}
\usepackage[full]{textcomp}
\usepackage[osf]{newtxtext} 
\usepackage[cal=boondoxo]{mathalfa}
\usepackage{colortbl}

\usepackage{comment}

\usepackage{amssymb}
\usepackage{mathtools}
\usepackage{hyperref}
\usepackage{breakurl}
\usepackage{mhenvs}
\usepackage{mhequ} 
\usepackage{mhsymb}
\usepackage{booktabs}
\usepackage{tikz}
\usepackage{tcolorbox}
\usepackage{mathrsfs}
\usepackage[utf8]{inputenc}
\usepackage{longtable}
\usepackage{wrapfig}
\usepackage{subcaption}
\usepackage{mathrsfs}
\usepackage{epsfig}
\usepackage{microtype}
\usepackage{comment}
\usepackage{wasysym}
\usepackage{centernot}
\usepackage{enumitem}
\usepackage{bm}
\usepackage{stackrel}
\usepackage{graphicx}
\makeatletter
\newcommand{\globalcolor}[1]{%
  \color{#1}\global\let\default@color\current@color
}
\makeatother

\usetikzlibrary{calc}
\usetikzlibrary{decorations}
\usetikzlibrary{positioning}
\usetikzlibrary{shapes}
\usetikzlibrary{external}

\definecolor{blush}{rgb}{0.87, 0.36, 0.51}
	\definecolor{brightcerulean}{rgb}{0.11, 0.67, 0.84}
	\definecolor{greenryb}{rgb}{0.4, 0.69, 0.2}

\newif\ifdark
\darkfalse

\ifdark
\definecolor{darkred}{rgb}{0.9,0.2,0.2}
\definecolor{darkblue}{rgb}{0.7,0.3,1}
\definecolor{darkgreen}{rgb}{0.1,0.9,0.1}
\definecolor{franck}{rgb}{0,0.8,1}
\definecolor{pagebackground}{rgb}{.15,.21,.18}
\definecolor{pageforeground}{rgb}{.84,.84,.85}
\pagecolor{pagebackground}
\AtBeginDocument{\globalcolor{pageforeground}}
\definecolor{symbols}{rgb}{0,0.7,1}
\colorlet{connection}{red!80!black}
\colorlet{boxcolor}{blue!50}

\else

\definecolor{darkred}{rgb}{0.7,0.1,0.1}
\definecolor{darkblue}{rgb}{0.4,0.1,0.8}
\definecolor{darkgreen}{rgb}{0.1,0.7,0.1}
\definecolor{franck}{rgb}{0,0,1}
\definecolor{pagebackground}{rgb}{1,1,1}
\definecolor{pageforeground}{rgb}{0,0,0}
\colorlet{symbols}{blue!90!black}
\colorlet{connection}{red!30!black}
\colorlet{boxcolor}{blue!50!black}

\fi

\def\slash{\leavevmode\unskip\kern0.18em/\penalty\exhyphenpenalty\kern0.18em}
\def\dash{\leavevmode\unskip\kern0.18em--\penalty\exhyphenpenalty\kern0.18em}

\DeclareMathAlphabet{\mathbbm}{U}{bbm}{m}{n}

\DeclareFontFamily{U}{BOONDOX-calo}{\skewchar\font=45 }
\DeclareFontShape{U}{BOONDOX-calo}{m}{n}{
  <-> s*[1.05] BOONDOX-r-calo}{}
\DeclareFontShape{U}{BOONDOX-calo}{b}{n}{
  <-> s*[1.05] BOONDOX-b-calo}{}
\DeclareMathAlphabet{\mcb}{U}{BOONDOX-calo}{m}{n}
\SetMathAlphabet{\mcb}{bold}{U}{BOONDOX-calo}{b}{n}

\setlist{noitemsep,topsep=4pt,leftmargin=1.5em}

\DeclareMathAlphabet{\mathbbm}{U}{bbm}{m}{n}

\DeclareMathAlphabet{\mcb}{U}{BOONDOX-calo}{m}{n}
\SetMathAlphabet{\mcb}{bold}{U}{BOONDOX-calo}{b}{n}
\DeclareFontFamily{U}{mathx}{\hyphenchar\font45}
\DeclareFontShape{U}{mathx}{m}{n}{
      <5> <6> <7> <8> <9> <10>
      <10.95> <12> <14.4> <17.28> <20.74> <24.88>
      mathx10
      }{}
\DeclareSymbolFont{mathx}{U}{mathx}{m}{n}
\DeclareMathSymbol{\bigtimes}{1}{mathx}{"91}

\providecommand{\figures}{false}
{ \ifthenelse{\equal{\figures}{false}} {#1}{\[ {\rm Figure \ missing !} \]} }{}

\def\CH{\mathcal{H}}

\def\CG{\mathcal{G}}

\def\CC{\mathcal{C}}
\def\CQ{\mathcal{Q}}

\def\CT{\mathcal{T}}

\tikzstyle{tinydots}=[dash pattern=on \pgflinewidth off \pgflinewidth]
\tikzstyle{superdense}=[dash pattern=on 4pt off 1pt]

\newcommand{\beq}{\begin{equation}}
\newcommand{\eeq}{\end{equation}}

\usepackage{empheq}

\newcommand{\T}{\mathbf{T}}

\def\Labp{\mathfrak{p}}
\def\Labe{\mathfrak{e}}

\def\Labn{\mathfrak{n}}
\def\Labo{\mathfrak{o}}

\def\Labhom{\mathfrak{t}}

\def\Lab{\mathfrak{L}}

\def\${|\!|\!|}

\newcounter{theorem}

\newtheorem{ex}[theorem]{Example}

\newenvironment{DIFnomarkup}{}{} 

\theorembodyfont{\rmfamily}

\newfont{\indic}{bbmss12}

\def\Nabla_#1{\nabla_{\!#1}}

\makeatletter
\pgfdeclareshape{crosscircle}
{
  \inheritsavedanchors[from=circle] 
  \inheritanchorborder[from=circle]
  \inheritanchor[from=circle]{north}
  \inheritanchor[from=circle]{north west}
  \inheritanchor[from=circle]{north east}
  \inheritanchor[from=circle]{center}
  \inheritanchor[from=circle]{west}
  \inheritanchor[from=circle]{east}
  \inheritanchor[from=circle]{mid}
  \inheritanchor[from=circle]{mid west}
  \inheritanchor[from=circle]{mid east}
  \inheritanchor[from=circle]{base}
  \inheritanchor[from=circle]{base west}
  \inheritanchor[from=circle]{base east}
  \inheritanchor[from=circle]{south}
  \inheritanchor[from=circle]{south west}
  \inheritanchor[from=circle]{south east}
  \inheritbackgroundpath[from=circle]
  \foregroundpath{
    \centerpoint%
    \pgf@xc=\pgf@x%
    \pgf@yc=\pgf@y%
    \pgfutil@tempdima=\radius%
    \pgfmathsetlength{\pgf@xb}{\pgfkeysvalueof{/pgf/outer xsep}}%
    \pgfmathsetlength{\pgf@yb}{\pgfkeysvalueof{/pgf/outer ysep}}%
    \ifdim\pgf@xb<\pgf@yb%
      \advance\pgfutil@tempdima by-\pgf@yb%
    \else%
      \advance\pgfutil@tempdima by-\pgf@xb%
    \fi%
    \pgfpathmoveto{\pgfpointadd{\pgfqpoint{\pgf@xc}{\pgf@yc}}{\pgfqpoint{-0.707107\pgfutil@tempdima}{-0.707107\pgfutil@tempdima}}}
    \pgfpathlineto{\pgfpointadd{\pgfqpoint{\pgf@xc}{\pgf@yc}}{\pgfqpoint{0.707107\pgfutil@tempdima}{0.707107\pgfutil@tempdima}}}
    \pgfpathmoveto{\pgfpointadd{\pgfqpoint{\pgf@xc}{\pgf@yc}}{\pgfqpoint{-0.707107\pgfutil@tempdima}{0.707107\pgfutil@tempdima}}}
    \pgfpathlineto{\pgfpointadd{\pgfqpoint{\pgf@xc}{\pgf@yc}}{\pgfqpoint{0.707107\pgfutil@tempdima}{-0.707107\pgfutil@tempdima}}}
  }
}
\makeatother

\def\symbol#1{\textcolor{symbols}{#1}}

\def\decorate#1#2{
        \ifnum#2>0
    		\foreach \count in {1,...,#2}{
	       	let
				\p1 = (sourcenode.center),
                \p2 = (sourcenode.east),
				\n1 = {\x2-\x1},
				\n2 = {1mm},
				\n3 = {(1.3+0.6*(\count-1))*\n1},
				\n4 = {0.7*\n1}
			in 
        		node[rectangle,fill=symbols,rotate=30,inner sep=0pt,minimum width=0.2*\n2,minimum height=\n2] at ($(sourcenode.center) + (\n3,\n4)$) {}
				}
		\fi
        \ifnum#1>0
    		\foreach \count in {1,...,#1}{
	       	let
				\p1 = (sourcenode.center),
                \p2 = (sourcenode.east),
				\n1 = {\x2-\x1},
				\n2 = {1mm},
				\n3 = {(1.3+0.6*(\count-1))*\n1},
				\n4 = {0.7*\n1}
			in 
        		node[rectangle,fill=symbols,rotate=-30,inner sep=0pt,minimum width=0.2*\n2,minimum height=\n2] at ($(sourcenode.center) + (-\n3,\n4)$) {}
				}
		\fi
}

\tikzset{
    dectriangle/.style 2 args={
        triangle,
        alias=sourcenode,
        append after command={\decorate{#1}{#2}}
    },
    dectriangle/.default={0}{0},
}

\tikzset{
	cross/.style={path picture={ 
  		\draw[symbols]
			(path picture bounding box.south east) -- (path picture bounding box.north west) (path picture bounding box.south west) -- (path picture bounding box.north east);
		}},
root/.style={circle,fill=green!50!black,inner sep=0pt, minimum size=1.2mm},
        dot/.style={circle,fill=pageforeground,inner sep=0pt, minimum size=1mm},
        dotred/.style={circle,fill=pageforeground!50!pagebackground,inner sep=0pt, minimum size=2mm},
        var/.style={circle,fill=pageforeground!10!pagebackground,draw=pageforeground,inner sep=0pt, minimum size=3mm},
        kernel/.style={semithick,shorten >=2pt,shorten <=2pt},
        kernels/.style={snake=zigzag,shorten >=2pt,shorten <=2pt,segment amplitude=1pt,segment length=4pt,line before snake=2pt,line after snake=5pt,},
        rho/.style={densely dashed,semithick,shorten >=2pt,shorten <=2pt},
           testfcn/.style={dotted,semithick,shorten >=2pt,shorten <=2pt},
        renorm/.style={shape=circle,fill=pagebackground,inner sep=1pt},
        labl/.style={shape=rectangle,fill=pagebackground,inner sep=1pt},
        xic/.style={very thin,circle,draw=symbols,fill=symbols,inner sep=0pt,minimum size=1.2mm},
        g/.style={very thin,rectangle,draw=symbols,fill=symbols!10!pagebackground,inner sep=0pt,minimum width=2.5mm,minimum height=1.2mm},
        xi/.style={very thin,circle,draw=symbols,fill=symbols!10!pagebackground,inner sep=0pt,minimum size=1.2mm},
	xies/.style={very thin,rectangle,fill=green!50!black!25,draw=symbols,inner sep=0pt,minimum size=1.1mm},
	xiesf/.style={very thin,rectangle,fill=green!50!black,draw=symbols,inner sep=0pt,minimum size=1.1mm},
        xix/.style={very thin,crosscircle,fill=symbols!10!pagebackground,draw=symbols,inner sep=0pt,minimum size=1.2mm},
        X/.style={very thin,cross,rectangle,fill=pagebackground,draw=symbols,inner sep=0pt,minimum size=1.2mm},
	xib/.style={thin,circle,fill=symbols!10!pagebackground,draw=symbols,inner sep=0pt,minimum size=1.6mm},
	xie/.style={thin,circle,fill=green!50!black,draw=symbols,inner sep=0pt,minimum size=1.6mm},
	xid/.style={thin,circle,fill=symbols,draw=symbols,inner sep=0pt,minimum size=1.6mm},
	xibx/.style={thin,crosscircle,fill=symbols!10!pagebackground,draw=symbols,inner sep=0pt,minimum size=1.6mm},
	kernels2/.style={very thick,draw=connection,segment length=12pt},
	keps/.style={thin,draw=symbols,->},
	kepspr/.style={thick,draw=connection,->},
	krho/.style={thin,draw=symbols,superdense,->},
	krhopr/.style={thick,draw=connection,superdense},
	triangle/.style = { regular polygon, regular polygon sides=3},
	not/.style={thin,circle,draw=connection,fill=connection,inner sep=0pt,minimum size=0.5mm},
	diff/.style = {very thin,draw=symbols,triangle,fill=red!50!black,inner sep=0pt,minimum size=1.6mm},
	diff1/.style = {very thin,dectriangle={1}{0},fill=red!50!black,draw=symbols,inner sep=0pt,minimum size=1.6mm},
	diff2/.style = {very thin,dectriangle={1}{1},fill=red!50!black,draw=symbols,inner sep=0pt,minimum size=1.6mm},
		diffmini/.style = {very thin,rectangle,fill=black,draw=black,inner sep=0pt,minimum size=0.75mm},
	 kernelsmod/.style={very thick,draw=connection,segment length=12pt},
	 rec/.style = {very thin,rectangle,fill=black,draw=black,inner sep=0pt,minimum size=2mm},
	cerc/.style={very thin,circle,draw=black,fill=symbols,inner sep=0pt,minimum size=2mm},
	stars/.style={very thin,star,star points=6,star point ratio=0.5, draw=black,fill=red,inner sep=0pt,minimum size=0.7mm},
	>=stealth,
        }
        \tikzset{
root/.style={circle,fill=black!50,inner sep=0pt, minimum size=3mm},
        circ/.style={circle,fill=white,draw=black,very thin,inner sep=.5pt, minimum size=1.2mm},
        round1/.style={fill=white,outer sep = 0,inner sep=2pt,rounded corners=1mm,draw,text=black,thin,minimum size=1.2mm},
          circ1/.style={circle,fill=red!10,draw=red,very thin,inner sep=.5pt, minimum size=1.2mm},
        rect/.style={fill=white,outer sep = 0,inner sep=2pt,rectangle,draw,text=black,thin,minimum size=1.2mm},
        rect1/.style={fill=white,outer sep = 0,inner sep=2pt,rectangle,draw,text=black,thin,minimum size=1.2mm},
        round2/.style={fill=red!10,outer sep = 0,inner sep=2pt,rounded corners=1mm,draw,text=black,thin,minimum size=1.2mm},
       round3/.style={fill=blue!10,outer sep = 0,inner sep=2pt,rounded corners=1mm,draw,text=black,thin,minimum size=1.2mm}, 
        rect2/.style={fill=black!10,outer sep = 0,inner sep=2pt,rectangle,draw,text=black,thin,minimum size=1.2mm},
        dot/.style={circle,fill=black,inner sep=0pt, minimum size=1.2mm},
        dotred/.style={circle,fill=black!50,inner sep=0pt, minimum size=2mm},
        var/.style={circle,fill=black!10,draw=black,inner sep=0pt, minimum size=3mm},
        kernel/.style={semithick,shorten >=2pt,shorten <=2pt},
         diag/.style={thin,shorten >=4pt,shorten <=4pt},
        kernel1/.style={thick},
        kernels/.style={snake=zigzag,shorten >=2pt,shorten <=2pt,segment amplitude=1pt,segment length=4pt,line before snake=2pt,line after snake=5pt,},
		kernels1/.style={snake=zigzag,segment amplitude=0.5pt,segment length=2pt},
		rho1/.style={densely dotted,semithick},
        rho/.style={densely dashed,semithick,shorten >=2pt,shorten <=2pt},
           testfcn/.style={dotted,semithick,shorten >=2pt,shorten <=2pt},
           visible/.style={draw, circle, fill, inner sep=0.25ex},
        renorm/.style={shape=circle,fill=white,inner sep=1pt},
        labl/.style={shape=rectangle,fill=white,inner sep=1pt},
        xic/.style={very thin,circle,fill=symbols,draw=black,inner sep=0pt,minimum size=1.2mm},
        xi/.style={very thin,circle,fill=blue!10,draw=black,inner sep=0pt,minimum size=1.2mm},
	xib/.style={very thin,circle,fill=blue!10,draw=black,inner sep=0pt,minimum size=1.6mm},
	xie/.style={very thin,circle,fill=green!50!black,draw=black,inner sep=0pt,minimum size=1mm},
	xid/.style={very thin,circle,fill=symbols,draw=black,inner sep=0pt,minimum size=1.6mm},
	edgetype/.style={very thin,circle,draw=black,inner sep=0pt,minimum size=5mm},
	nodetype/.style={very thick,circle,draw=black,inner sep=0pt,minimum size=5mm},
	kernels2/.style={very thick,draw=connection,segment length=12pt},
clean/.style={thin,circle,fill=black,inner sep=0pt,minimum size=1mm},	not/.style={thin,circle,fill=symbols,draw=connection,fill=connection,inner sep=0pt,minimum size=0.8mm},
	>=stealth,
        }

\makeatletter
\def\DeclareSymbol#1#2#3{%
	\expandafter\gdef\csname MH@symb@#1\endcsname{\tikzsetnextfilename{symbol#1}%
	\tikz[baseline=#2,scale=0.15,draw=symbols,line join=round]{#3}}%
	\expandafter\gdef\csname MH@symb@#1s\endcsname{\scalebox{0.75}{\tikzsetnextfilename{symbol#1}%
	\tikz[baseline=#2,scale=0.15,draw=symbols,line join=round]{#3}}}%
	\expandafter\gdef\csname MH@symb@#1ss\endcsname{\scalebox{0.65}{\tikzsetnextfilename{symbol#1}%
	\tikz[baseline=#2,scale=0.15,draw=symbols,line join=round]{#3}}}%
	}
\def\<#1>{\ifthenelse{\boolean{mmode}}{\mathchoice{\csname MH@symb@#1\endcsname}{\csname MH@symb@#1\endcsname}{\csname MH@symb@#1s\endcsname}{\csname MH@symb@#1ss\endcsname}}{\csname MH@symb@#1\endcsname}}
\makeatother

\DeclareSymbol{Xi22}{0.5}{\draw (0,0) node[xi] {} -- (-1,1) node[not] {} -- (0,2) node[xi] {};} 
\DeclareSymbol{Xi2}{-2}{\draw (-1,-0.25) node[xi] {} -- (0,1) node[xi] {};} 
\DeclareSymbol{Xi2b}{-2}{\draw (-1,-0.25) node[xic] {} -- (0,1) node[xic] {};} 
\DeclareSymbol{Xi2g}{-2}{\draw (-1,-0.25) node[xies] {} -- (0,1) node[xi] {};} 
\DeclareSymbol{Xi2g2}{-2}{\draw (-1,-0.25) node[xi] {} -- (0,1) node[xies] {};} 
\DeclareSymbol{cXi2}{-2}{\draw (0,-0.25) node[xi] {} -- (-1,1) node[xic] {};}
\DeclareSymbol{Xi3}{0}{\draw (0,0) node[xi] {} -- (-1,1) node[xi] {} -- (0,2) node[xi] {};}
\DeclareSymbol{XiIIXi}{0}{\draw (0,0) node[xi] {} -- (-1,1); \draw[kernels2] (-1,1) node[not] {} -- (0,2) node[xi] {};}

\DeclareSymbol{Xi4}{2}{\draw (0,0) node[xi] {} -- (-1,1) node[xi] {} -- (0,2) node[xi] {} -- (-1,3) node[xi] {};}
\DeclareSymbol{Xi4_1}{2}{\draw (0,0) node[xic] {} -- (-1,1) node[xic] {} -- (0,2) node[xi] {} -- (-1,3) node[xi] {};}
\DeclareSymbol{Xi4_2}{2}{\draw (0,0) node[xic] {} -- (-1,1) node[xi] {} -- (0,2) node[xi] {} -- (-1,3) node[xic] {};}
\DeclareSymbol{Xi2X}{-2}{\draw (0,-0.25) node[xi] {} -- (-1,1) node[xix] {};}
\DeclareSymbol{XXi2}{-2}{\draw (0,-0.25) node[xix] {} -- (-1,1) node[xi] {};}
\DeclareSymbol{IIXi}{0}{\draw (0,-0.25) node[not] {} -- (-1,1) node[xi] {} -- (0,2) node[xi] {};}
\DeclareSymbol{IXi^2}{-1}{\draw (-1,1) node[xi] {} -- (0,0) node[not] {} -- (1,1) node[xi] {};}
\DeclareSymbol{IIXi^2}{-4}{\draw (0,-1.5) node[not] {} -- (0,0);
\draw[kernels2] (-1,1) node[xi] {} -- (0,0) node[not] {} -- (1,1) node[xi] {};}
\DeclareSymbol{XiX}{-2.8}{\node[xibx] {};}
\DeclareSymbol{tauX}{-2.8}{ \node[X] {};}
\DeclareSymbol{Xi}{-2.8}{\node[xib] {};}

\DeclareSymbol{IXiX}{-1}{\draw (0,-0.25) node[not] {} -- (-1,1) node[xix] {};}
\DeclareSymbol{IXi3}{2}{\draw (0,-0.25) node[not] {} -- (-1,1) node[xi] {} -- (0,2) node[xi] {} -- (-1,3) node[xi] {};}
\DeclareSymbol{IXi}{-2}{\draw (0,-0.25) node[not] {} -- (-1,1) node[xi] {};}
\DeclareSymbol{XiI}{-2}{\draw (0,-0.25) node[xi] {} -- (-1,1) node[not] {};}

\DeclareSymbol{Xi4b}{0}{\draw(0,1.5) node[xi] {} -- (0,0); \draw (-1,1) node[xi] {} -- (0,0) node[xi] {} -- (1,1) node[xi] {};}
\DeclareSymbol{Xi4b'}{0}{\draw(0,1.5) node[xi] {} -- (0,-0.2); \draw (-1,1) node[xi] {} -- (0,-0.2) node[not] {} -- (1,1) node[xi] {};}
\DeclareSymbol{Xi4c}{0}{\draw (0,1) -- (0.8,2.2) node[xi] {};\draw (0,-0.25) node[xi] {} -- (0,1) node[xi] {} -- (-0.8,2.2) node[xi] {};}
\DeclareSymbol{Xi4d}{-4.5}{\draw (0,-1.5) node[not] {} -- (0,0); \draw (-1,1) node[xi] {} -- (0,0) node[xi] {} -- (1,1) node[xi] {};}
\DeclareSymbol{Xi4e}{0}{\draw (0,2) node[xi] {} -- (-1,1) node[xi] {} -- (0,0) node[xi] {} -- (1,1) node[xi] {};}
\DeclareSymbol{Xi4e'}{0}{\draw (0,2) node[xi] {} -- (-1,1) node[xi] {} -- (0,-0.2) node[not] {} -- (1,1) node[xi] {};}

\DeclareSymbol{Xitwo}
{0}{\draw[kernels2] (0,0) node[not] {} -- (-1,1) node[not] {}
-- (-2,2) node[not]{} -- (-3,3) node[xi]  {};
\draw[kernels2] (0,0) -- (1,1) node[xi] {};
\draw[kernels2] (-1,1) -- (0,2) node[xi] {};
\draw[kernels2] (-2,2) -- (-1,3) node[xi] {};}

\DeclareSymbol{IXitwo}
{0}{\draw (-.7,1.2) node[xi] {} -- (0,-0.2) -- (.7,1.2) node[xi] {};}
\DeclareSymbol{I1Xitwo}
{0}{\draw[kernels2] (0,0) node[not] {} -- (-1,1) node[xi] {};
\draw[kernels2] (0,0) -- (1,1) node[xi] {};}

\DeclareSymbol{I1Xitwobis}
{0}{\draw[kernels2] (0,0) node[not] {} -- (-1,1) node[xies] {};
\draw[kernels2] (0,0) -- (1,1) node[xies] {};}

\DeclareSymbol{I1Xitwog}
{0}{\draw[kernels2] (0,0) node[not] {} -- (-1,1) node[xies] {};
\draw[kernels2] (0,0) -- (1,1) node[xi] {};}

\DeclareSymbol{cI1Xitwo}
{0}{\draw[kernels2] (0,0) node[not] {} -- (-1,1) node[xic] {};
\draw[kernels2] (0,0) -- (1,1) node[xi] {};}

\DeclareSymbol{I1IXi3}{0}{\draw (0,0) node[xi] {} -- (-1,1) ; 
\draw[kernels2] (-1,1) node[not] {} -- (0,2) node[xi] {};
\draw[kernels2] (-1,1) node[not] {} -- (-2,2) node[xi] {};}

\DeclareSymbol{I1Xi3c}{-1}{\draw[kernels2](0,1.5) node[xi] {} -- (0,0) node[not] {}; \draw (-1,1) node[xi] {} -- (0,0) ; \draw[kernels2] (0,0) -- (1,1) node[xi] {};}

\DeclareSymbol{I1Xi3cbis}{-1}{\draw[kernels2](0,1.5) node[xies] {} -- (0,0) node[not] {}; \draw (-1,1) node[xies] {} -- (0,0) ; \draw[kernels2] (0,0) -- (1,1) node[xies] {};}

\DeclareSymbol{I1IXi3b}{0}{\draw[kernels2] (0,0) node[not] {} -- (-1,1) ; \draw[kernels2] (0,0)   -- (1,1) node[xi] {} ;
\draw (-1,1) node[xi] {} -- (0,2) node[xi] {};
}

\DeclareSymbol{I1IXi3c}{0}{\draw[kernels2] (0,0) node[not] {} -- (-1,1) ; \draw[kernels2] (0,0)   -- (1,1) node[xi] {} ;
\draw[kernels2] (-1,1) node[not] {} -- (0,2) node[xi] {};
\draw[kernels2] (-1,1) node[not] {} -- (-2,2) node[xi] {};}

\DeclareSymbol{I1IXi3cbis}{0}{\draw[kernels2] (0,0) node[not] {} -- (-1,1) ; \draw[kernels2] (0,0)   -- (1,1) node[xies] {} ;
\draw[kernels2] (-1,1) node[not] {} -- (0,2) node[xies] {};
\draw[kernels2] (-1,1) node[not] {} -- (-2,2) node[xies] {};}

\DeclareSymbol{I1Xi}{0}{\draw[kernels2] (0,0) node[not] {} -- (-1,1)  node[xi] {} ;}

\DeclareSymbol{I1Xi4a}{2}{\draw[kernels2] (0,0) node[not] {} -- (-1,1) ; \draw[kernels2] (0,0) node[not] {} -- (1,1) node[xi] {} ;
\draw (-1,1) node[xi] {} -- (0,2) node[xi] {} -- (-1,3) node[xi] {};}

\DeclareSymbol{cI1Xi4a}{2}{\draw[kernels2] (0,0) node[not] {} -- (-1,1) ; \draw[kernels2] (0,0) node[not] {} -- (1,1) node[xic] {} ;
\draw (-1,1) node[xic] {} -- (0,2) node[xi] {} -- (-1,3) node[xi] {};}

\DeclareSymbol{I1Xi4b}{2}{\draw (0,0) node[xi] {} -- (-1,1) node[xi] {} -- (0,2) ; \draw[kernels2] (0,2) node[not] {} -- (-1,3) node[xi] {};\draw[kernels2] (0,2)  -- (1,3) node[xi] {};
}

\DeclareSymbol{cI1Xi4b}{2}{\draw (0,0) node[xic] {} -- (-1,1) node[xic] {} -- (0,2) ; \draw[kernels2] (0,2) node[not] {} -- (-1,3) node[xi] {};\draw[kernels2] (0,2)  -- (1,3) node[xi] {};
}

\DeclareSymbol{I1Xi4c}{2}{\draw (0,0) node[xi] {} -- (-1,1) node[not] {}; \draw[kernels2] (-1,1) -- (0,2) ; 
\draw[kernels2] (-1,1) -- (-2,2) node[xi] {} ;
\draw (0,2) node[xi] {} -- (-1,3) node[xi] {};}

\DeclareSymbol{cI1Xi4c}{2}{\draw (0,0) node[xic] {} -- (-1,1) node[not] {}; \draw[kernels2] (-1,1) -- (0,2) ; 
\draw[kernels2] (-1,1) -- (-2,2) node[xic] {} ;
\draw (0,2) node[xi] {} -- (-1,3) node[xi] {};}

\DeclareSymbol{I1Xi4ab}{2}{\draw[kernels2] (0,0) node[not] {} -- (-1,1) ; \draw[kernels2] (0,0) node[not] {} -- (1,1) node[xi] {};\draw (-1,1) node[xi] {} -- (0,2) ; \draw[kernels2] (0,2) node[not] {} -- (-1,3) node[xi] {};\draw[kernels2] (0,2)  -- (1,3) node[xi] {}; }

\DeclareSymbol{cI1Xi4ab}{2}{\draw[kernels2] (0,0) node[not] {} -- (-1,1) ; \draw[kernels2] (0,0) node[not] {} -- (1,1) node[xic] {};\draw (-1,1) node[xic] {} -- (0,2) ; \draw[kernels2] (0,2) node[not] {} -- (-1,3) node[xi] {};\draw[kernels2] (0,2)  -- (1,3) node[xi] {}; }

\DeclareSymbol{I1Xi4bc}{2}{\draw (0,0) node[xi] {} -- (-1,1) node[not] {}; \draw[kernels2] (-1,1) -- (0,2) ; 
\draw[kernels2] (-1,1) -- (-2,2) node[xi] {} ; \draw[kernels2] (0,2) node[not] {} -- (-1,3) node[xi] {};\draw[kernels2] (0,2)  -- (1,3) node[xi] {};
}

\DeclareSymbol{cI1Xi4bc}{2}{\draw (0,0) node[xic] {} -- (-1,1) node[not] {}; \draw[kernels2] (-1,1) -- (0,2) ; 
\draw[kernels2] (-1,1) -- (-2,2) node[xic] {} ; \draw[kernels2] (0,2) node[not] {} -- (-1,3) node[xi] {};\draw[kernels2] (0,2)  -- (1,3) node[xi] {};
}

\DeclareSymbol{I1Xi4abcc1}{2}{\draw[kernels2] (0,0) node[not] {} -- (-1,1) node[not] {}
-- (-2,2) node[not]{} -- (-3,3) node[xic]  {};
\draw[kernels2] (0,0) -- (1,1) node[xic] {};
\draw[kernels2] (-1,1) -- (0,2) node[xi] {};
\draw[kernels2] (-2,2) -- (-1,3) node[xi] {};
}

\DeclareSymbol{I1Xi4abcc1b}{2}{\draw[kernels2] (0,0) node[not] {} -- (-1,1) node[not] {}
-- (-2,2) node[not]{} -- (-3,3) node[xi]  {};
\draw[kernels2] (0,0) -- (1,1) node[xic] {};
\draw[kernels2] (-1,1) -- (0,2) node[xic] {};
\draw[kernels2] (-2,2) -- (-1,3) node[xi] {};
}

\DeclareSymbol{I1Xi4abcc2}{2}{\draw[kernels2] (0,0) node[not] {} -- (-1,1) node[not] {}
-- (-2,2) node[not]{} -- (-3,3) node[xic]  {};
\draw[kernels2] (0,0) -- (1,1) node[xi] {};
\draw[kernels2] (-1,1) -- (0,2) node[xi] {};
\draw[kernels2] (-2,2) -- (-1,3) node[xic] {};
}

\DeclareSymbol{I1Xi4ac}{2}{\draw[kernels2] (0,0) node[not] {} -- (-1,1) ; \draw[kernels2] (0,0) node[not] {} -- (1,1) node[xi] {}; 
\draw[kernels2] (-1,1) node[not] {} -- (0,2) ; 
\draw[kernels2] (-1,1) -- (-2,2) node[xi] {} ;
\draw (0,2) node[xi] {} -- (-1,3) node[xi] {};}

\DeclareSymbol{cI1Xi4ac}{2}{\draw[kernels2] (0,0) node[not] {} -- (-1,1) ; \draw[kernels2] (0,0) node[not] {} -- (1,1) node[xic] {}; 
\draw[kernels2] (-1,1) node[not] {} -- (0,2) ; 
\draw[kernels2] (-1,1) -- (-2,2) node[xic] {} ;
\draw (0,2) node[xi] {} -- (-1,3) node[xi] {};}

\DeclareSymbol{I1Xi4acc1}{2}{\draw[kernels2] (0,0) node[not] {} -- (-1,1) ; \draw[kernels2] (0,0) node[not] {} -- (1,1) node[xic] {}; 
\draw[kernels2] (-1,1) node[not] {} -- (0,2) ; 
\draw[kernels2] (-1,1) -- (-2,2) node[xi] {} ;
\draw (0,2) node[xic] {} -- (-1,3) node[xi] {};}

\DeclareSymbol{I1Xi4acc2}{2}{\draw[kernels2] (0,0) node[not] {} -- (-1,1) ; \draw[kernels2] (0,0) node[not] {} -- (1,1) node[xic] {}; 
\draw[kernels2] (-1,1) node[not] {} -- (0,2) ; 
\draw[kernels2] (-1,1) -- (-2,2) node[xi] {} ;
\draw (0,2) node[xi] {} -- (-1,3) node[xic] {};}

\DeclareSymbol{2I1Xi4}{2}{\draw[kernels2] (0,0) node[not] {} -- (-1,1) node[not] {};
\draw[kernels2] (0,0) -- (1,1) node[not] {};
\draw[kernels2] (-1,1) -- (-1.5,2.5) node[xi] {};
\draw[kernels2] (-1,1) -- (-0.5,2.5) node[xi] {};
\draw[kernels2] (1,1) -- (0.5,2.5) node[xi] {};
\draw[kernels2] (1,1) -- (1.5,2.5) node[xi] {};
}

\DeclareSymbol{2I1Xi4dis}{2}{\draw[kernels2] (0,0) node[not] {} -- (-1,1) node[not] {};
\draw[kernels2] (0,0) -- (1,1) node[not] {};
\draw[kernels2] (-1,1) -- (-1.5,2.5) node[xies] {};
\draw[kernels2] (-1,1) -- (-0.5,2.5) node[xies] {};
\draw[kernels2] (1,1) -- (0.5,2.5) node[xies] {};
\draw[kernels2] (1,1) -- (1.5,2.5) node[xies] {};
}

\DeclareSymbol{2I1Xi4c1}{2}{\draw[kernels2] (0,0) node[not] {} -- (-1,1) node[not] {};
\draw[kernels2] (0,0) -- (1,1) node[not] {};
\draw[kernels2] (-1,1) -- (-1.5,2.5) node[xic] {};
\draw[kernels2] (-1,1) -- (-0.5,2.5) node[xi] {};
\draw[kernels2] (1,1) -- (0.5,2.5) node[xic] {};
\draw[kernels2] (1,1) -- (1.5,2.5) node[xi] {};
}

\DeclareSymbol{2I1Xi4c2}{2}{\draw[kernels2] (0,0) node[not] {} -- (-1,1) node[not] {};
\draw[kernels2] (0,0) -- (1,1) node[not] {};
\draw[kernels2] (-1,1) -- (-1.5,2.5) node[xic] {};
\draw[kernels2] (-1,1) -- (-0.5,2.5) node[xic] {};
\draw[kernels2] (1,1) -- (0.5,2.5) node[xi] {};
\draw[kernels2] (1,1) -- (1.5,2.5) node[xi] {};
}

\DeclareSymbol{2I1Xi4b}{2}{\draw[kernels2] (0,0) node[not] {} -- (-1,1) ;
\draw[kernels2] (0,0) -- (1,1);
\draw (-1,1) node[xi] {} -- (-1,2.5) node[xi] {};
\draw (1,1)  node[xi] {} -- (1,2.5) node[xi] {};
}

\DeclareSymbol{2I1Xi4bb}{2}{\draw[kernels2] (0,0) node[not] {} -- (-1,1) ;
\draw[kernels2] (0,0) -- (1,1);
\draw (-1,1) node[xi] {} -- (-1,2.5) node[xiesf] {};
\draw (1,1)  node[xi] {} -- (1,2.5) node[xic] {};
}

\DeclareSymbol{2I1Xi4c}{2}{\draw[kernels2] (0,0) node[not] {} -- (-1,1);
\draw[kernels2] (0,0) -- (1,1) node[not] {};
\draw (-1,1)  node[xi] {} -- (-1,2.5) node[xi] {};
\draw[kernels2] (1,1) -- (0.4,2.5) node[xi] {};
\draw[kernels2] (1,1) -- (1.6,2.5) node[xi] {};
}

\DeclareSymbol{2I1Xi4cc1}{2}{\draw[kernels2] (0,0) node[not] {} -- (-1,1);
\draw[kernels2] (0,0) -- (1,1) node[not] {};
\draw (-1,1)  node[xic] {} -- (-1,2.5) node[xi] {};
\draw[kernels2] (1,1) -- (0.4,2.5) node[xic] {};
\draw[kernels2] (1,1) -- (1.6,2.5) node[xi] {};
}

\DeclareSymbol{2I1Xi4cc2}{2}{\draw[kernels2] (0,0) node[not] {} -- (-1,1);
\draw[kernels2] (0,0) -- (1,1) node[not] {};
\draw (-1,1)  node[xic] {} -- (-1,2.5) node[xic] {};
\draw[kernels2] (1,1) -- (0.4,2.5) node[xi] {};
\draw[kernels2] (1,1) -- (1.6,2.5) node[xi] {};
}

\DeclareSymbol{Xi4ba}{0}{\draw(-0.5,1.5) node[xi] {} -- (0,0); \draw (-1.5,1) node[xi] {} -- (0,0) node[not] {}; \draw[kernels2] (0,0) -- (1.5,1) node[xi] {};
\draw[kernels2] (0,0) -- (0.5,1.5) node[xi] {} ;}

\DeclareSymbol{Xi4badis}{0}{\draw(-0.5,1.5) node[xies] {} -- (0,0); \draw (-1.5,1) node[xies] {} -- (0,0) node[not] {}; \draw[kernels2] (0,0) -- (1.5,1) node[xies] {};
\draw[kernels2] (0,0) -- (0.5,1.5) node[xies] {} ;}

\DeclareSymbol{Xi4ba1}{0}{\draw(-0.5,1.5) node[xi] {} -- (0,0); \draw (-1.5,1) node[xi] {} -- (0,0) node[not] {}; \draw[kernels2] (0,0) -- (1.5,1) node[xic] {};
\draw[kernels2] (0,0) -- (0.5,1.5) node[xic] {} ;}

\DeclareSymbol{Xi4ba1b}{0}{\draw(-0.5,1.5) node[xic] {} -- (0,0); \draw (-1.5,1) node[xic] {} -- (0,0) node[not] {}; \draw[kernels2] (0,0) -- (1.5,1) node[xi] {};
\draw[kernels2] (0,0) -- (0.5,1.5) node[xi] {} ;}

\DeclareSymbol{Xi4ba1bdiff}{0}{\draw(-0.5,1.5) node[xic] {} -- (0,0); \draw (-1.5,1) node[xic] {} -- (0,0) node[not] {}; \draw (0,0) -- (1.5,1) node[xi] {};
\draw (0,0) -- (0.5,1.5) node[xi] {};
\draw(0,0) node[diff] {};}

\DeclareSymbol{Xi4ba1bb}{0}{\draw(-0.5,1.5) node[xic] {} -- (0,0); \draw (-1.5,1) node[xiesf] {} -- (0,0) node[not] {}; \draw[kernels2] (0,0) -- (1.5,1) node[xi] {};
\draw[kernels2] (0,0) -- (0.5,1.5) node[xi] {} ;}

\DeclareSymbol{Xi4ba2}{0}{\draw(-0.5,1.5) node[xi] {} -- (0,0); \draw (-1.5,1) node[xic] {} -- (0,0) node[not] {}; \draw[kernels2] (0,0) -- (1.5,1) node[xi] {};
\draw[kernels2] (0,0) -- (0.5,1.5) node[xic] {} ;}

\DeclareSymbol{Xi4ba2b}{0}{\draw(-0.5,1.5) node[xi] {} -- (0,0); \draw (-1.5,1) node[xic] {} -- (0,0) node[not] {}; \draw[kernels2] (0,0) -- (1.5,1) node[xi] {};
\draw[kernels2] (0,0) -- (0.5,1.5) node[xiesf] {} ;}

\DeclareSymbol{Xi4ca}{0}{\draw (0,1) -- (-1,2.2) node[xi] {};\draw (0,-0.25) node[xi] {} -- (0,1) ; \draw[kernels2] (0,1) node[not] {} -- (1,2.2) node[xi] {};
\draw[kernels2] (0,1) {} -- (0,2.7) node[xi] {};
}

\DeclareSymbol{Xi4cb}{0}{\draw (-1,1) -- (-2,2) node[xi] {};\draw[kernels2] (0,0)  -- (-1,1) node[xi] {} ; \draw[kernels2] (0,0) node[not] {} -- (1,1) node[xi] {} ; 
\draw (-1,1) node[xi] {} -- (0,2) node[xi] {};}

\DeclareSymbol{Xi4cbb}{0}{\draw (-1,1) -- (-2,2) node[xiesf] {};\draw[kernels2] (0,0)  -- (-1,1) node[xi] {} ; \draw[kernels2] (0,0) node[not] {} -- (1,1) node[xi] {} ; 
\draw (-1,1) node[xi] {} -- (0,2) node[xic] {};}

\DeclareSymbol{Xi4cbc1}{0}{\draw (-1,1) -- (-2,2) node[xic] {};\draw[kernels2] (0,0)  -- (-1,1) node[xic] {} ; \draw[kernels2] (0,0) node[not] {} -- (1,1) node[xi] {} ; 
\draw (-1,1) node[xic] {} -- (0,2) node[xi] {};}

\DeclareSymbol{Xi4cbc2}{0}{\draw (-1,1) -- (-2,2) node[xi] {};\draw[kernels2] (0,0)  -- (-1,1) node[xi] {} ; \draw[kernels2] (0,0) node[not] {} -- (1,1) node[xic] {} ; 
\draw (-1,1) node[xic] {} -- (0,2) node[xi] {};}

\DeclareSymbol{Xi4cab}{0}{\draw (-1,1) -- (-2,2) node[xi] {};\draw[kernels2] (0,0)  -- (-1,1); \draw[kernels2] (0,0) node[not] {} -- (1,1) node[xi] {} ; 
\draw[kernels2] (-1,1)  {} -- (0,2) node[xi] {};
\draw[kernels2] (-1,1) node[not] {} -- (-1,2.5) node[xi] {};
}

\DeclareSymbol{Xi4cabdis}{0}{\draw (-1,1) -- (-2,2) node[xies] {};\draw[kernels2] (0,0)  -- (-1,1); \draw[kernels2] (0,0) node[not] {} -- (1,1) node[xies] {} ; 
\draw[kernels2] (-1,1)  {} -- (0,2) node[xies] {};
\draw[kernels2] (-1,1) node[not] {} -- (-1,2.5) node[xies] {};
}

\DeclareSymbol{Xi4cabc1}{0}{\draw (-1,1) -- (-2,2) node[xi] {};\draw[kernels2] (0,0)  -- (-1,1); \draw[kernels2] (0,0) node[not] {} -- (1,1) node[xic] {} ; 
\draw[kernels2] (-1,1)  {} -- (0,2) node[xic] {};
\draw[kernels2] (-1,1) node[not] {} -- (-1,2.5) node[xi] {};
}

\DeclareSymbol{Xi4cabc2}{0}{\draw (-1,1) -- (-2,2) node[xic] {};\draw[kernels2] (0,0)  -- (-1,1); \draw[kernels2] (0,0) node[not] {} -- (1,1) node[xic] {} ; 
\draw[kernels2] (-1,1)  {} -- (0,2) node[xi] {};
\draw[kernels2] (-1,1) node[not] {} -- (-1,2.5) node[xi] {};
}

\DeclareSymbol{Xi4ea}{1.5}{\draw (-1,2.5) node[xi] {} -- (-1,1) node[xi] {} -- (0,0); 
 \draw[kernels2] (0,0)  -- (1,1) node[xi] {};
\draw[kernels2] (0,0) node[not] {} -- (0,1.5) node[xi] {}; }

\DeclareSymbol{Xi4eac1}{1.5}{\draw (-1,2.5) node[xic] {} -- (-1,1) node[xi] {} -- (0,0); 
 \draw[kernels2] (0,0)  -- (1,1) node[xic] {};
\draw[kernels2] (0,0) node[not] {} -- (0,1.5) node[xi] {}; }

\DeclareSymbol{Xi4eac1b}{1.5}{\draw (-1,2.5) node[xic] {} -- (-1,1) node[xi] {} -- (0,0); 
 \draw[kernels2] (0,0)  -- (1,1) node[xiesf] {};
\draw[kernels2] (0,0) node[not] {} -- (0,1.5) node[xi] {}; }

\DeclareSymbol{Xi4eac2}{1.5}{\draw (-1,2.5) node[xic] {} -- (-1,1) node[xic] {} -- (0,0); 
 \draw[kernels2] (0,0)  -- (1,1) node[xi] {};
\draw[kernels2] (0,0) node[not] {} -- (0,1.5) node[xi] {}; }

\DeclareSymbol{Xi4eact1}{1.5}{\draw (-1,2.5) node[xic] {} -- (-1,1) node[xi] {} -- (0,0); 
 \draw (0,0)  -- (1,1) node[xic] {};
\draw[rho] (0,0) node[not] {} -- (0,1.5) node[xi] {}; }

\DeclareSymbol{Xi4eact2}{1.5}{\draw[rho] (-1,2.5) node[xic] {} -- (-1,1) node[xi] {} -- (0,0); 
 \draw (0,0)  -- (1,1) node[xic] {};
\draw (0,0) node[not] {} -- (0,1.5) node[xi] {}; }

\DeclareSymbol{Xi4eabis}{1.5}{\draw (-1,2.5) node[xi] {} -- (-1,1) ; \draw[kernels2] (-1,1) node[xi] {} -- (0,0); 
 \draw (0,0)  -- (1,1) node[xi] {};
\draw[kernels2] (0,0) node[not] {} -- (0,1.5) node[xi] {}; }

\DeclareSymbol{Xi4eabisc1}{1.5}{\draw (-1,2.5) node[xic] {} -- (-1,1) ; \draw[kernels2] (-1,1) node[xi] {} -- (0,0); 
 \draw (0,0)  -- (1,1) node[xi] {};
\draw[kernels2] (0,0) node[not] {} -- (0,1.5) node[xic] {}; }

\DeclareSymbol{Xi4eabisc1b}{1.5}{\draw (-1,2.5) node[xic] {} -- (-1,1) ; \draw[kernels2] (-1,1) node[xi] {} -- (0,0); 
 \draw (0,0)  -- (1,1) node[xi] {};
\draw[kernels2] (0,0) node[not] {} -- (0,1.5) node[xiesf] {}; }

\DeclareSymbol{Xi4eabisc1bis}{1.5}{\draw (-1,2.5) node[xi] {} -- (-1,1) ; \draw[kernels2] (-1,1) node[xi] {} -- (0,0); 
 \draw (0,0)  -- (1,1) node[xi] {};
\draw[kernels2] (0,0) node[not] {} -- (0,1.5) node[xi] {};
\draw (-2,1) node[] {\tiny{$i$}};
\draw (-2,2.5) node[] {\tiny{$\ell$}};
\draw (2,1) node[] {\tiny{$k$}};
\draw (0,2.5) node[] {\tiny{$j$}};
 }

\DeclareSymbol{Xi4eabisc1tris}{1.5}{\draw (-1,2.5) node[xi] {} -- (-1,1) ; \draw[kernels2] (-1,1) node[xi] {} -- (0,0); 
 \draw (0,0)  -- (1,1) node[xi] {};
\draw[kernels2] (0,0) node[not] {} -- (0,1.5) node[xi] {};
\draw (-2,1) node[] {\tiny{i}};
\draw (-2,2.5) node[] {\tiny{j}};
\draw (2,1) node[] {\tiny{j}};
\draw (0,2.5) node[] {\tiny{i}};
 }

\DeclareSymbol{Xi4eabisc1quater}{1.5}{\draw (-1,2.5) node[xic] {} -- (-1,1) ; \draw[kernels2] (-1,1) node[xi] {} -- (0,0); 
 \draw (0,0)  -- (1,1) node[xic] {};
\draw[kernels2] (0,0) node[not] {} -- (0,1.5) node[xi] {};
 }

\DeclareSymbol{Xi4eabisc2}{1.5}{\draw (-1,2.5) node[xic] {} -- (-1,1) ; \draw[kernels2] (-1,1) node[xi] {} -- (0,0); 
 \draw (0,0)  -- (1,1) node[xic] {};
\draw[kernels2] (0,0) node[not] {} -- (0,1.5) node[xi] {}; }

\DeclareSymbol{Xi4eabisc2l}{1.5}{\draw (-1,2.5) node[xiesf] {} -- (-1,1) ; \draw[kernels2] (-1,1) node[xi] {} -- (0,0); 
 \draw (0,0)  -- (1,1) node[xic] {};
\draw[kernels2] (0,0) node[not] {} -- (0,1.5) node[xi] {}; }

\DeclareSymbol{Xi4eabisc2r}{1.5}{\draw (-1,2.5) node[xic] {} -- (-1,1) ; \draw[kernels2] (-1,1) node[xi] {} -- (0,0); 
 \draw (0,0)  -- (1,1) node[xiesf] {};
\draw[kernels2] (0,0) node[not] {} -- (0,1.5) node[xi] {}; }

\DeclareSymbol{Xi4eabisc3}{1.5}{\draw (-1,2.5) node[xic] {} -- (-1,1) ; \draw[kernels2] (-1,1) node[xic] {} -- (0,0); 
 \draw (0,0)  -- (1,1) node[xi] {};
\draw[kernels2] (0,0) node[not] {} -- (0,1.5) node[xi] {}; }

\DeclareSymbol{Xi4eb}{0}{
\draw[kernels2] (0,2) node[xi] {} -- (-1,1) ; \draw[kernels2] (-2,2)  node[xi] {} -- (-1,1) ; \draw (-1,1)  node[not] {} -- (0,0); 
 \draw (0,0) node[xi] {}  -- (1,1) node[xi] {};
}

\DeclareSymbol{Xi4eab}{1.5}{\draw[kernels2] (-1,2.5) node[xi] {} -- (-1,1) ; \draw[kernels2] (-2,2)  node[xi] {} -- (-1,1) ; \draw (-1,1)  node[not] {} -- (0,0); 
 \draw[kernels2] (0,0)  -- (1,1) node[xi] {};
\draw[kernels2] (0,0) node[not] {} -- (0,1.5) node[xi] {}; 
}

\DeclareSymbol{Xi4eabdis}{1.5}{\draw[kernels2] (-1,2.5) node[xies] {} -- (-1,1) ; \draw[kernels2] (-2,2)  node[xies] {} -- (-1,1) ; \draw (-1,1)  node[not] {} -- (0,0); 
 \draw[kernels2] (0,0)  -- (1,1) node[xies] {};
\draw[kernels2] (0,0) node[not] {} -- (0,1.5) node[xies] {}; 
}

\DeclareSymbol{Xi4eabc1}{1.5}{\draw[kernels2] (-1,2.5) node[xic] {} -- (-1,1) ; \draw[kernels2] (-2,2)  node[xi] {} -- (-1,1) ; \draw (-1,1)  node[not] {} -- (0,0); 
 \draw[kernels2] (0,0)  -- (1,1) node[xic] {};
\draw[kernels2] (0,0) node[not] {} -- (0,1.5) node[xi] {}; 
}

\DeclareSymbol{Xi4eabc2}{1.5}{\draw[kernels2] (-1,2.5) node[xi] {} -- (-1,1) ; \draw[kernels2] (-2,2)  node[xi] {} -- (-1,1) ; \draw (-1,1)  node[not] {} -- (0,0); 
 \draw[kernels2] (0,0)  -- (1,1) node[xic] {};
\draw[kernels2] (0,0) node[not] {} -- (0,1.5) node[xic] {}; 
}

\DeclareSymbol{Xi4eabbis}{1.5}{\draw[kernels2] (-1,2.5) node[xi] {} -- (-1,1) ; \draw[kernels2] (-2,2)  node[xi] {} -- (-1,1) ; \draw[kernels2] (-1,1)  node[not] {} -- (0,0); 
 \draw (0,0)  -- (1,1) node[xi] {};
\draw[kernels2] (0,0) node[not] {} -- (0,1.5) node[xi] {}; 
}

\DeclareSymbol{Xi4eabbisc1}{1.5}{\draw[kernels2] (-1,2.5) node[xic] {} -- (-1,1) ; \draw[kernels2] (-2,2)  node[xi] {} -- (-1,1) ; \draw[kernels2] (-1,1)  node[not] {} -- (0,0); 
 \draw (0,0)  -- (1,1) node[xic] {};
\draw[kernels2] (0,0) node[not] {} -- (0,1.5) node[xi] {}; 
}

\DeclareSymbol{Xi4eabbisc1perm}{1.5}{\draw[kernels2] (-1,2.5) node[xi] {} -- (-1,1) ; \draw[kernels2] (-2,2)  node[xic] {} -- (-1,1) ; \draw[kernels2] (-1,1)  node[not] {} -- (0,0); 
 \draw (0,0)  -- (1,1) node[xic] {};
\draw[kernels2] (0,0) node[not] {} -- (0,1.5) node[xi] {}; 
}

\DeclareSymbol{Xi4eabbisc2}{1.5}{\draw[kernels2] (-1,2.5) node[xi] {} -- (-1,1) ; \draw[kernels2] (-2,2)  node[xi] {} -- (-1,1) ; \draw[kernels2] (-1,1)  node[not] {} -- (0,0); 
 \draw (0,0)  -- (1,1) node[xic] {};
\draw[kernels2] (0,0) node[not] {} -- (0,1.5) node[xic] {}; 
}

\DeclareSymbol{Xi2cbis}{0}{\draw[kernels2] (0,1) -- (0.8,2.2) node[xi] {};\draw[kernels2] (0,-0.25) node[not] {} -- (0,1); \draw[kernels2] (0,1) node[not] {} -- (-0.8,2.2) node[xi] {};}

\DeclareSymbol{Xi2cbis1}{0}{\draw (0,1) -- (-0.8,2.2) node[xi] {};\draw[kernels2] (0,-0.25) node[not] {} -- (0,1) node[xi] {}; }

\DeclareSymbol{Xi2Xbis}{-2}{\draw[kernels2] (0,-0.25)  -- (-1,1) ; \draw (-1,1) node[xix] {};
\draw[kernels2] (0,-0.25) node[not] {} -- (1,1) node[xi] {};}

\DeclareSymbol{XXi2bis}{-2}{\draw[kernels2] (0,-0.25) -- (-1,1) node[xi] {};
\draw[kernels2] (0,-0.25) node[X] {} -- (1,1) node[xi] {};}

\DeclareSymbol{I1XiIXi}{0}{\draw[kernels2] (0,-0.25) -- (1,1) node[xi] {};
\draw (0,-0.25) node[not] {} -- (-1,1) node[xi] {};}

\DeclareSymbol{I1XiIXib}{0}{\draw  (0,-0.25) node[xi] {} -- (0,1) node[not] {};
\draw[kernels2] (0,1) -- (0,2.25) ; \draw (0,2.25) node[xi]{}; }

\DeclareSymbol{I1XiIXic}{0}{
\draw[kernels2] (0,0) -- (1,1) node[xi] {} ; 
\draw[kernels2] (0,0) node[not] {}  -- (-1,1) node[not] {} -- (0,2) node[xi] {};
}

\DeclareSymbol{thin}{1.4}{\draw[pagebackground] (-0.3,0) -- (0.3,0); \draw  (0,0) -- (0,2);}
\DeclareSymbol{thin2}{1.4}{\draw[pagebackground] (-0.3,0) -- (0.3,0); \draw[tinydots]  (0,0) -- (0,2);}

\DeclareSymbol{thick}{1.4}{\draw[pagebackground] (-0.3,0) -- (0.3,0); \draw[kernels2]  (0,0) -- (0,2);}

\DeclareSymbol{thick2}{1.4}{\draw[pagebackground] (-0.3,0) -- (0.3,0); \draw[kernels2,tinydots]  (0,0) -- (0,2);}

\DeclareSymbol{Xi4ind}{2}{\draw (0,0) node[xi,label={[label distance=-0.2em]right: \scriptsize  $ i $}]  { } -- (-1,1) node[xi,label={[label distance=-0.2em]left: \scriptsize  $ j $}] {} -- (0,2) node[xi,label={[label distance=-0.2em]right: \scriptsize  $ k $}] {} -- (-1,3) node[xi,label={[label distance=-0.2em]left: \scriptsize  $ \ell $}] {};}

\DeclareSymbol{Xi4c1}{2}{\draw (0,0) node[xic] {} -- (-1,1) node[xi] {} -- (0,2) node[xic] {} -- (-1,3) node[xi] {};} 
\DeclareSymbol{IXi2ex}{0}{\draw (0,-0.25) node[xie] {} -- (-1,1) node[xi] {} ; \draw (0,-0.25)-- (1,1) node[xi] {};}
\DeclareSymbol{IXi2ex1}{0}{\draw (0,-0.25) node[xie] {} -- (-1,1) node[xi] {} -- (0,2) node[xi] {};}

\DeclareSymbol{Xi4b1}{0}{\draw(0,1.5) node[xic] {} -- (0,0); \draw (-1,1) node[xic] {} -- (0,0) node[xi] {} -- (1,1) node[xi] {};}

\DeclareSymbol{Xi4ec1}{0}{\draw (0,2) node[xi] {} -- (-1,1) node[xic] {} -- (0,0) node[xic] {} -- (1,1) node[xi] {};}
\DeclareSymbol{Xi4ec2}{0}{\draw (0,2) node[xic] {} -- (-1,1) node[xi] {} -- (0,0) node[xic] {} -- (1,1) node[xi] {};}
\DeclareSymbol{Xi4ec3}{0}{\draw (0,2) node[xic] {} -- (-1,1) node[xic] {} -- (0,0) node[xi] {} -- (1,1) node[xi] {};}

\DeclareSymbol{I1Xi4ac1}{2}{\draw[kernels2] (0,0) node[not] {} -- (-1,1) ; \draw[kernels2] (0,0) node[not] {} -- (1,1) node[xic] {} ;
\draw (-1,1) node[xi] {} -- (0,2) node[xic] {} -- (-1,3) node[xi] {};}

\DeclareSymbol{I1Xi4ac2}{2}{\draw[kernels2] (0,0) node[not] {} -- (-1,1) ; \draw[kernels2] (0,0) node[not] {} -- (1,1) node[xic] {} ;
\draw (-1,1) node[xi] {} -- (0,2) node[xi] {} -- (-1,3) node[xic] {};}

\DeclareSymbol{I1Xi4bp}{2}{\draw (0,0) node[not] {} -- (-1,1) node[xi] {} -- (0,2) ; \draw[kernels2] (0,2) node[not] {} -- (-1,3) node[xi] {};\draw[kernels2] (0,2)  -- (1,3) node[xi] {};
}

\DeclareSymbol{I1Xi4bc1}{2}{\draw (0,0) node[xic] {} -- (-1,1) node[xi] {} -- (0,2) ; \draw[kernels2] (0,2) node[not] {} -- (-1,3) node[xi] {};\draw[kernels2] (0,2)  -- (1,3) node[xic] {};
}

\DeclareSymbol{I1Xi4bc2}{2}{\draw (0,0) node[xic] {} -- (-1,1) node[xi] {} -- (0,2) ; \draw[kernels2] (0,2) node[not] {} -- (-1,3) node[xic] {};\draw[kernels2] (0,2)  -- (1,3) node[xi] {};
}

\DeclareSymbol{I1Xi4cp}{2}{\draw (0,0) node[not] {} -- (-1,1) node[not] {}; \draw[kernels2] (-1,1) -- (0,2) ; 
\draw[kernels2] (-1,1) -- (-2,2) node[xi] {} ;
\draw (0,2) node[xi] {} -- (-1,3) node[xi] {};}

\DeclareSymbol{I1Xi4cc1}{2}{\draw (0,0) node[xic] {} -- (-1,1) node[not] {}; \draw[kernels2] (-1,1) -- (0,2) ; 
\draw[kernels2] (-1,1) -- (-2,2) node[xi] {} ;
\draw (0,2) node[xic] {} -- (-1,3) node[xi] {};}

\DeclareSymbol{I1Xi4cc2}{2}{\draw (0,0) node[xic] {} -- (-1,1) node[not] {}; \draw[kernels2] (-1,1) -- (0,2) ; 
\draw[kernels2] (-1,1) -- (-2,2) node[xi] {} ;
\draw (0,2) node[xi] {} -- (-1,3) node[xic] {};}

\DeclareSymbol{I1Xi4abc1}{2}{\draw[kernels2] (0,0) node[not] {} -- (-1,1) ; \draw[kernels2] (0,0) node[not] {} -- (1,1) node[xic] {};\draw (-1,1) node[xi] {} -- (0,2) ; \draw[kernels2] (0,2) node[not] {} -- (-1,3) node[xic] {};\draw[kernels2] (0,2)  -- (1,3) node[xi] {}; }

\DeclareSymbol{I1Xi4abc2}{2}{\draw[kernels2] (0,0) node[not] {} -- (-1,1) ; \draw[kernels2] (0,0) node[not] {} -- (1,1) node[xic] {};\draw (-1,1) node[xi] {} -- (0,2) ; \draw[kernels2] (0,2) node[not] {} -- (-1,3) node[xi] {};\draw[kernels2] (0,2)  -- (1,3) node[xic] {}; }

\DeclareSymbol{R1}{0}{\draw (-1,1) node[xi] {} -- (0,0) node[not] {};
\draw[kernels2] (0,1.5) node[xic] {} -- (0,0) -- (1,1) node[xic] {};}
\DeclareSymbol{R2}{0}{\draw (-1,1) node[xic] {} -- (0,0) node[not] {};
\draw[kernels2] (0,1.5)  {} -- (0,0) -- (1,1)  {};
\draw (0,1.5) node[xi] {};
\draw (1,1) node[xic] {};
}
\DeclareSymbol{R3}{1}{\draw[kernels2] (-1,1.5)  {} -- (0,0) node[not] {} -- (1,1.5);
\draw (-1,1.5) node[xi] {};
\draw[kernels2] (0,3) {} -- (1,1.5) -- (2,3)  {};
\draw  (0,3) node[xic] {} ;
\draw (2,3) node[xic] {};}
\DeclareSymbol{R4}{1}{\draw[kernels2] (-1,1.5) node[xic] {} -- (0,0) node[not] {} -- (1,1.5);
\draw[kernels2] (0,3) {} -- (1,1.5) -- (2,3) node[xic] {};
\draw (0,3) node[xi] {};}

\DeclareSymbol{I1Xi4bcp}{2}{\draw (0,0) node[not] {} -- (-1,1) node[not] {}; \draw[kernels2] (-1,1) -- (0,2) ; 
\draw[kernels2] (-1,1) -- (-2,2) node[xi] {} ; \draw[kernels2] (0,2) node[not] {} -- (-1,3) node[xi] {};\draw[kernels2] (0,2)  -- (1,3) node[xi] {};
}

\DeclareSymbol{I1Xi4bcc1}{2}{\draw (0,0) node[xic] {} -- (-1,1) node[not] {}; \draw[kernels2] (-1,1) -- (0,2) ; 
\draw[kernels2] (-1,1) -- (-2,2) node[xi] {} ; \draw[kernels2] (0,2) node[not] {} -- (-1,3) node[xi] {};\draw[kernels2] (0,2)  -- (1,3) node[xic] {};
}

\DeclareSymbol{I1Xi4bcc2}{2}{\draw (0,0) node[xic] {} -- (-1,1) node[not] {}; \draw[kernels2] (-1,1) -- (0,2) ; 
\draw[kernels2] (-1,1) -- (-2,2) node[xi] {} ; \draw[kernels2] (0,2) node[not] {} -- (-1,3) node[xic] {};\draw[kernels2] (0,2)  -- (1,3) node[xi] {};
} 

\DeclareSymbol{2I1Xi4bc1}{2}{\draw[kernels2] (0,0) node[not] {} -- (-1,1) ;
\draw[kernels2] (0,0) -- (1,1);
\draw (-1,1) node[xic] {} -- (-1,2.5) node[xi] {};
\draw (1,1)  node[xic] {} -- (1,2.5) node[xi] {};
}

\DeclareSymbol{2I1Xi4bc2}{2}{\draw[kernels2] (0,0) node[not] {} -- (-1,1) ;
\draw[kernels2] (0,0) -- (1,1);
\draw (-1,1) node[xi] {} -- (-1,2.5) node[xic] {};
\draw (1,1)  node[xic] {} -- (1,2.5) node[xi] {};
}

\DeclareSymbol{diff2I1Xi4bc2}{2}{\draw (0,0) node[diff] {} -- (-1,1) ;
\draw (0,0) -- (1,1);
\draw (-1,1) node[xi] {} -- (-1,2.5) node[xic] {};
\draw (1,1)  node[xic] {} -- (1,2.5) node[xi] {};
}

\DeclareSymbol{2I1Xi4bc3}{2}{\draw[kernels2] (0,0) node[not] {} -- (-1,1) ;
\draw[kernels2] (0,0) -- (1,1);
\draw (-1,1) node[xic] {} -- (-1,2.5) node[xic] {};
\draw (1,1)  node[xi] {} -- (1,2.5) node[xi] {};
}

\DeclareSymbol{Xi41}{0}{\draw (0,1) -- (0.8,2.2) node[xic] {};\draw (0,-0.25) node[xi] {} -- (0,1) node[xi] {} -- (-0.8,2.2) node[xic] {};} 

\DeclareSymbol{Xi42}{0}{\draw (0,1) -- (0.8,2.2) node[xi] {};\draw (0,-0.25) node[xic] {} -- (0,1) node[xi] {} -- (-0.8,2.2) node[xic] {};}

\DeclareSymbol{Xi4ca1}{0}{\draw (0,1) -- (-1,2.2) node[xic] {};\draw (0,-0.25) node[xi] {} -- (0,1) ; \draw[kernels2] (0,1) node[not] {} -- (1,2.2) node[xic] {};
\draw[kernels2] (0,1) {} -- (0,2.7) node[xi] {};
}

\DeclareSymbol{Xi4ca2}{0}{\draw (0,1) -- (-1,2.2) node[xi] {};\draw (0,-0.25) node[xi] {} -- (0,1) ; \draw[kernels2] (0,1) node[not] {} -- (1,2.2) node[xic] {};
\draw[kernels2] (0,1) {} -- (0,2.7) node[xic] {};
}

\DeclareSymbol{Xi4cap}{0}{\draw (0,1) -- (-1,2.2) node[xi] {};\draw (0,-0.25) node[not] {} -- (0,1) ; \draw[kernels2] (0,1) node[not] {} -- (1,2.2) node[xi] {};
\draw[kernels2] (0,1) {} -- (0,2.7) node[xi] {};
}

\DeclareSymbol{Xi3a}{0}{
 \draw (-1,1)  node[xi] {} -- (0,0); 
 \draw (0,0) node[xi] {}  -- (1,1) node[xi] {};
}

\DeclareSymbol{Xi4ebc1}{0}{
\draw[kernels2] (0,2) node[xi] {} -- (-1,1) ; \draw[kernels2] (-2,2)  node[xic] {} -- (-1,1) ; \draw (-1,1)  node[not] {} -- (0,0); 
 \draw (0,0) node[xic] {}  -- (1,1) node[xi] {};
}

\DeclareSymbol{Xi4ebc2}{0}{
\draw[kernels2] (0,2) node[xi] {} -- (-1,1) ; \draw[kernels2] (-2,2)  node[xi] {} -- (-1,1) ; \draw (-1,1)  node[not] {} -- (0,0); 
 \draw (0,0) node[xic] {}  -- (1,1) node[xic] {};
}

\DeclareSymbol{Xi2cbispex}{0}{\draw[kernels2] (0,1) -- (0.8,2.2) node[xi] {};\draw (0,-0.25) node[xie] {} -- (0,1); \draw[kernels2] (0,1) node[not] {} -- (-0.8,2.2) node[xi] {};}

\DeclareSymbol{Xi2cbis1p}{0}{\draw (0,1) -- (-0.8,2.2) node[xi] {};\draw (0,-0.25) node[not] {} -- (0,1) node[xi] {}; }

\DeclareSymbol{Xi2Xp}{-2}{\draw (0,-0.25) node[not] {} -- (-1,1) node[xix] {};} 

\DeclareSymbol{I1XiIXib}{0}{\draw  (0,-0.25) node[xi] {} -- (0,1) node[not] {};
\draw[kernels2] (0,1) -- (0,2.25) ; \draw (0,2.25) node[xi]{}; }

\DeclareSymbol{IXi2b}{0}{\draw  (0,-0.25) node[xi] {} -- (0,1) node[not] {};
\draw (0,1) -- (0,2.25) ; \draw (0,2.25) node[xi]{}; }

\DeclareSymbol{IXi2bex}{0}{\draw  (0,-0.25) node[xi] {} -- (0,1) node[xie] {};
\draw (0,1) -- (0,2.25) ; \draw (0,2.25) node[xi]{}; }

 \def\1{\mathbf{\symbol{1}}}

\def\one{\mathbf{1}}

\DeclareSymbol{diff}{0}{
\draw (0,0.5) node[diff] {};
}

\DeclareSymbol{diff1}{0}{
\draw (0,0.5) node[diff1] {};
}

\DeclareSymbol{diff2}{0}{
\draw (0,0.5) node[diff2] {};
}

\DeclareSymbol{geo}{0}{
\draw (0,0) node[diff] {};
\draw (0.3,0) node[diff] {};
}

\DeclareSymbol{generic}{0}{
\draw (0,0.6) node[xi] {};
}

\DeclareSymbol{g}{0}{
\draw (0,0.6) node[g] {};
}

\DeclareSymbol{Ito}{0}{
\draw (0,0.6) node[xies] {};
}

\DeclareSymbol{Itob}{0}{
\draw (0,0.6) node[xiesf] {};
}

\DeclareSymbol{greycirc}{0}{
\draw (0,0.3) node[xi] {};
}

\DeclareSymbol{not}{0}{
\draw (0,0.6) node[not] {};
\draw[tinydots] (0,0.6) circle (0.8);
}

\DeclareSymbol{genericb}{0}{
\draw (0,0.6) node[xic] {};
}

\DeclareSymbol{bluecirc}{0}{
\draw (0,0.3) node[xic] {};
}

\DeclareSymbol{genericxix}{0}{
\draw (0,0.6) node[xix] {};
}

\DeclareSymbol{genericX}{0}{
\draw (0,0.6) node[X] {};
}

\DeclareSymbol{diffIto}{1}{
\draw  (0,2.5) -- (0,0) ;
\draw (0,-0.1) node[diff] {};
\draw (0,2.5) node[xies] {};
}
\DeclareSymbol{Itodiff}{2}{
\draw(0,2.9) -- (0,-0.2);
\draw (0,2.9) node[diff] {};
\draw (0,-0.1) node[xies] {};
}

\DeclareSymbol{diffgeneric}{1}{
\draw  (0,2.5) -- (0,0) ;
\draw (0,-0.1) node[diff] {};
\draw (0,2.5) node[xi] {};
}

\DeclareSymbol{genericdiff}{2}{
\draw(0,2.9) -- (0,-0.2);
\draw (0,2.9) node[diff] {};
\draw (0,-0.1) node[xi] {};
}

\DeclareSymbol{diffdot}{2}{
\draw  (0,3) -- (0,-0.1) ;
\draw (0,3) node[not] {};
\draw (0,-0.1) node[diff] {};
}

\DeclareSymbol{diffdotmini}{0}{
\draw  (0,0) -- (0,1.2) ;
\draw (0,1.2) node[not] {};
\draw (0,0) node[diffmini] {};
}

\DeclareSymbol{dotdiff}{2}{
\draw[kernelsmod]  (0,3) -- (0,-0.1) ;
\draw (0,3) node[diff] {};
\draw (0,-0.1) node[not] {};
}

\DeclareSymbol{dotdiff1}{2}{
\draw[kernelsmod]  (0,3) -- (0,-0.1) ;
\draw (0,3) node[diff1] {};
\draw (0,-0.1) node[not] {};
}

\DeclareSymbol{dotdiff1mini}{0}{
\draw[kernelsmod]  (0,1.2) -- (0,0) ;
\draw (0,1.2) node[diffmini] {};
\draw (0,0) node[not] {};
}

\DeclareSymbol{dotdiff2}{2}{
\draw (0,3) -- (0,-0.1) ;
\draw (0,3) node[diff] {};
\draw (0,-0.1) node[not] {};
}

\DeclareSymbol{dotdiff2mini}{0}{
\draw (0,1.2) -- (0,0) ;
\draw (0,1.2) node[diffmini] {};
\draw (0,0) node[not] {};
}

\DeclareSymbol{dotdiffstraight}{0}{
\draw  (0,3) -- (0,-0.1) ;
\draw (0,3) node[diff] {};
\draw (0,-0.1) node[not] {};
}

\DeclareSymbol{arbre1}{0}{
\draw  (0,0) -- (1.5,1.5) ;
\draw (1.5,1.5) node[not] {};
\draw (0,0) node[not] {};
}

\DeclareSymbol{arbre2}{0}{
\draw  (0,0) -- (1.5,1.5) ;
\draw[kernelsmod] (0,0) -- (-1.5,1.5);
\draw (1.5,1.5) node[not] {};
\draw (0,0) node[not] {};
\draw (-1.5,1.5) node[xi] {};
}

\DeclareSymbol{arbre3}{0}{
\draw  (0,0) -- (1.5,1.5) ;
\draw[kernelsmod] (1.5,1.5) -- (0,3);
\draw (0,0) node[not] {};
\draw (1.5,1.5) node[not] {};
\draw (0,3) node[xi] {};
}

\DeclareSymbol{treeeval}{0}{
\draw (0,0) -- (1,1);
\draw (0,0) node[xi] {};
\draw (1.25,1.25) node[xi] {};
\draw (-0.6,0.6) node[]{\tiny{$i$}};
\draw (0.65,1.85) node[]{\tiny{$j$}};
}

\DeclareSymbol{testeval}{0}{
\draw (0,0) -- (1,1);
\draw (0,0) -- (-1,1);
\draw (0,0) node[xi] {};
\draw (1.25,1.25) node[xi] {};
\draw (-1.25,1.25) node[xi] {};
\draw (-0.6,-0.6) node[]{\tiny{$i$}};
\draw (0.65,1.85) node[]{\tiny{$j$}};
\draw (-1.95,1.85) node[]{\tiny{$k$}};
}

\DeclareSymbol{treeeval2}{0}{
\draw[kernelsmod] (-0.25,-1) -- (1,0.5) ;
\draw[kernelsmod] (1,0.5) -- (-0.25,2);
\draw (1,0.5) node[diff2] {};
\draw (-0.25,-1) node[not] {};
\draw (-0.25,2) node[xi] {};
\draw (-0.6,1.2) node[]{\tiny{1}};
}

\DeclareSymbol{arbreact}{1}{
\draw (0,0) node[not] {};
\draw[kernelsmod] (0,0) -- (1,1);
\draw[kernelsmod] (0,0) -- (-1,1);
\draw (-1,1) node[xic] {};
\draw  (0,2) -- (1,1) ;
\draw (0,2) node[xic] {};
\draw (1,1) node[xi] {};
}

\DeclareSymbol{arbreact1}{0}{
\draw (0,-1.5) -- (0,0);
\draw[kernelsmod] (0,0) -- (1,1);
\draw[kernelsmod] (0,0) -- (-1,1);
\draw  (0,2) -- (1,1) ;
\draw (0,-1.5) node[diff] {};
\draw (0,0) node[not] {};
\draw (-1,1) node[xic] {};
\draw (0,2) node[xic] {};
\draw (1,1) node[xi] {};
}

\DeclareSymbol{arbreact2}{0}{
\draw (0,-0.75) -- (-1,0.5); 
\draw (0,-0.75) -- (1,0.5);
\draw (0,1.5) -- (1,0.5);
\draw (0,1.5) node[xic] {};
\draw (1,0.5) node[xi] {};
\draw (-1,0.5) node[xic] {};
\draw (0,-0.75) node[diff] {};
}

\DeclareSymbol{arbreact3}{0}{
\draw[kernelsmod] (0,-0.75) -- (-1,0.5); 
\draw[kernelsmod] (0,-0.75) -- (1,0.5);
\draw (0,1.75) -- (1,0.5);
\draw (2,1.75) -- (1,0.5);
\draw (0,1.75) node[xic] {};
\draw (1,0.5) node[diff] {};
\draw (-1,0.5) node[xic] {};
\draw (2,1.75) node[xi] {};
\draw (0,-0.75) node[not] {};
}

\DeclareSymbol{pre_im_I1Xitwo}{0}{
\draw[kernels2] (0,-0.3) node[not] {} -- (-0.6,0.7) ;
\draw[kernels2] (0,-0.3) -- (0.6,0.7);
\draw (0,0.9) node[g] {};
}

\DeclareSymbol{pre_im_cI1Xi4ab}{2}{
\draw[kernels2] (0,-1) node[not] {} -- (-0.6,0) ;
\draw[kernels2] (0,-1) -- (0.6,0);
\draw (0,0.2) node[g] {};
\draw (0,0.6) -- (0,1.5);
\draw[kernels2] (0,1.5) node[not] {} -- (-0.6,2.5) ;
\draw[kernels2] (0,1.5) -- (0.6,2.5);
\draw (0,2.7) node[g] {};
}

\DeclareSymbol{pre_im_I1Xi4acc2}{0}{
\draw[kernels2] (-1,-0.5) node[not] {} -- (-1.6,0.5) ;
\draw[kernels2] (-1,-0.5) -- (-0.4,0.5);
\draw[kernels2] (-1,-0.5) -- (0.2,-1.5) node[not] {} ;
\draw (-1,1.1) -- (-1,2);
\draw[kernels2] (0.2,-1.5) -- (0.2,2);
\draw (-1,0.7) node[g] {};
\draw (-0.3,2.2) node[g] {};
}

\DeclareSymbol{pre_im_I1Xi4abcc2}{2}{
\draw[kernels2] (0,-1) node[not] {} -- (-1,0) node[not] {};
\draw[kernels2] (-1,1.2) node[not] {} -- (-1,0);
\draw[kernels2] (-1,1.2) -- (-1.5,2.5);
\draw[kernels2] (-1,1.2) -- (-0.5,2.5);
\draw[kernels2] (-1,0) -- (0.7,2.5);
\draw[kernels2] (0,-1) -- (1.5,2.5);
\draw (-1,2.7) node[g] {};
\draw (1,2.7) node[g] {};
}

\DeclareSymbol{pre_im_2I1Xi4c1}{2}{
\draw[kernels2] (0,-0.5) node[not] {} -- (-1,0.5) node[not] {};
\draw[kernels2] (0,-0.5) -- (1,0.5) node[not] {};
\draw[kernels2] (-1,0.5) node[not] {}-- (-1.7,2);
\draw[kernels2]  (-1,2) -- (1,0.5);
\draw[kernels2] (-1,0.5) -- (1,2);
\draw[kernels2] (1,0.5) -- (1.7,2);
\draw (-1.2,2.2) node[g] {};
\draw (1.2,2.2) node[g] {};
}

\DeclareSymbol{pre_im_Xi4eabisc2}{2}{
\draw[kernels2] (1.2,-0.5) node[not] {} -- (-0.7,0.8) ;
\draw[kernels2] (1.2,-0.5) -- (0.4,0.8);
\draw (0,1.4)  -- (0,2.2);
\draw (1.2,2.2) -- (1.2,-0.6);
\draw (0,1) node[g] {};
\draw (0.6,2.4) node[g] {};
}

\DeclareSymbol{pre_im_Xi4eabisc22}{2}{
\draw (1.2,-0.5) node[not] {} -- (-0.7,0.8) ;
\draw[kernels2] (1.2,-0.5) -- (0.4,0.8);
\draw (0,1.4)  -- (0,2.2);
\draw[kernels2] (1.2,2.2) -- (1.2,-0.6);
\draw (0,1) node[g] {};
\draw (0.6,2.4) node[g] {};
}

\DeclareSymbol{pre_im_Xi4eabisc222}{2}{
\draw[kernels2] (0.4,-0.5) node[not] {} -- (-0.6,1) ;
\draw[kernels2] (1.2,0) -- (0.3,1);
\draw (0,1.1)  -- (0,2.5);
\draw[kernels2] (1.2,2.5) -- (1.2,0) node[not] {} -- (0.4,-0.6);
\draw (0,1.2) node[g] {};
\draw (0.6,2.5) node[g] {};
}

\DeclareSymbol{pre_im_Xi4eabc2}{2}{
\draw (0,-0.5) node[not] {} -- (-1,0.5) node[not] {};
\draw[kernels2] (-1,0.5) -- (-1.5,2);
\draw[kernels2] (0,-0.5)  -- (0.7,2);
\draw[kernels2] (-1,0.5) -- (-0.5,2);
\draw[kernels2] (0,-0.5) -- (1.5,2);
\draw (-1,2.2) node[g] {};
\draw (1,2.2) node[g] {};
}

\DeclareSymbol{pre_im_Xi4eabbisc2}{2}{
\draw[kernels2] (0,-0.5) node[not] {} -- (-1,0.5) node[not] {};
\draw[kernels2] (-1,0.5) -- (-1.5,2);
\draw[kernels2] (0,-0.5)  -- (0.7,2);
\draw[kernels2] (-1,0.5) -- (-0.5,2);
\draw (0,-0.5) -- (1.5,2);
\draw (-1,2.2) node[g] {};
\draw (1,2.2) node[g] {};
}

\DeclareSymbol{pre_im_I1Xi4abcc1}{2}{
\draw[kernels2] (0,-1) node[not] {} -- (-1,0) node[not] {};
\draw[kernels2] (0,1.1) node[not] {} -- (-1,0);
\draw[kernels2] (-1,0) -- (-1.5,2.5);
\draw[kernels2] (0,1.1) node[not] {} -- (-0.5,2.5);
\draw[kernels2] (0,1.1) -- (0.5,2.5);
\draw[kernels2] (0,-1) -- (1.5,2.5);
\draw (-1,2.7) node[g] {};
\draw (1,2.7) node[g] {};
}

\DeclareSymbol{pre_im_Xi4eabc1}{2}{
\draw (0,-0.5) node[not] {} -- (-1,0.5) node[not] {};
\draw[kernels2] (-1,0.5) -- (-1.5,2);
\draw[kernels2] (0,-0.5)  -- (-0.5,2);
\draw[kernels2] (-1,0.5) -- (0.8,2);
\draw[kernels2] (0,-0.5) -- (1.5,2);
\draw (-1,2.2) node[g] {};
\draw (1,2.2) node[g] {};
}

\DeclareSymbol{pre_im_Xi4ba1b}{2}{
\draw[kernels2] (0,0) node[not] {}  -- (1.8,1.5);
\draw[kernels2] (0,0) -- (0.8,1.5);
\draw (0,-0.1) -- (-1.8,1.5);
\draw (0,-0.1) -- (-0.8,1.5);
\draw (-1,1.7) node[g] {};
\draw (1,1.7) node[g] {};
}

\DeclareSymbol{pre_im_Xi4ba2}{2}{
\draw (0,-0.1) node[not] {}  -- (1.8,1.5);
\draw[kernels2] (0,0) -- (0.8,1.5);
\draw (0,-0.1) -- (-1.8,1.5);
\draw[kernels2] (0,0) -- (-0.8,1.5);
\draw (-1,1.7) node[g] {};
\draw (1,1.7) node[g] {};
}

\DeclareSymbol{pre_im_Xi4cabc2}{2}{
\draw[kernels2] (0,-0.5) node[not] {} -- (-1,0.5) node[not] {};
\draw[kernels2] (-1,0.5) -- (-1.5,2);
\draw (-1,0.5)  -- (0.7,2);
\draw[kernels2] (-1,0.5) -- (-0.5,2);
\draw[kernels2] (0,-0.5) -- (1.5,2);
\draw (-1,2.2) node[g] {};
\draw (1,2.2) node[g] {};
}

\DeclareSymbol{pre_im_Xi4cabc1}{2}{
\draw[kernels2] (0,-0.5) node[not] {} -- (-1,0.5) node[not] {};
\draw (-1,0.5) -- (-1.5,2);
\draw[kernels2] (-1,0.5)  -- (0.7,2);
\draw[kernels2] (-1,0.5) -- (-0.5,2);
\draw[kernels2] (0,-0.5) -- (1.5,2);
\draw (-1,2.2) node[g] {};
\draw (1,2.2) node[g] {};
}

\DeclareSymbol{pre_im_Xi4eabbisc1}{2}{
\draw[kernels2] (0,-0.5) node[not] {} -- (-1,0.5) node[not] {};
\draw[kernels2] (-1,0.5) -- (-1.5,2);
\draw[kernels2] (0,-0.5)  -- (-0.5,2);
\draw[kernels2] (-1,0.5) -- (0.8,2);
\draw (0,-0.5) -- (1.5,2);
\draw (-1,2.2) node[g] {};
\draw (1,2.2) node[g] {};
}

\DeclareSymbol{pre_im_1}{0}{
\draw[kernels2] (0,-0.5) node[not] {} -- (-0.6,0.5) ;
\draw[kernels2] (0,-0.5) -- (0.6,0.5);
\draw (0,1.1)  -- (-0.55,2);
\draw (0,1.1)  -- (0.55,2);
\draw (0,0.7) node[g] {};
\draw (0,2.2) node[g] {};
}

\DeclareSymbol{disconnect}{0}{
\draw[kernels2] (0,-0.5) node[not] {} -- (-0.6,0.5) ;
\draw[kernels2] (0,-0.5) -- (0.6,0.5);
\draw (-0.55,1.1)  -- (-0.55,2.3);
\draw (0.55,2.3) -- (0.55,1.5) -- (1.2,1.5) -- (1.2,3.5) -- (0.55,3.5) -- (0.55,2.7);
\draw (0,0.7) node[g] {};
\draw (0,2.5) node[g] {};
}

\DeclareSymbol{pre_im_2}{2}{\draw[kernels2] (0,0) node[not] {} -- (-1,1) node[not] {};
\draw[kernels2] (0,0) -- (1,1) node[not] {};
\draw[kernels2] (-1,1) -- (-1.5,2.5);
\draw[kernels2] (-1,1) -- (-0.5,2.5);
\draw[kernels2] (1,1) -- (0.5,2.5);
\draw[kernels2] (1,1) -- (1.5,2.5);
\draw (-1,2.7) node[g] {};
\draw (1,2.7) node[g] {};
}

\DeclareSymbol{CX_rec}{0}{
\draw [black] (-0.3,1) to (-0.3,-0.3);
\draw [black] (0.3,1) to (0.3,-0.3);
\draw [black] (-0.3,1) to (-0.3,2.3);
\draw [black] (0.3,1) to (0.3,2.3);
\draw (0,1) node[rec] {};
}

\DeclareSymbol{CX_cerc}{0}{
\draw [black] (0,1) to (0,-0.3);
\draw (0,1) node[cerc] {};
}

\DeclareMathAlphabet{\mathpzc}{OT1}{pzc}{m}{it}

\def\eqref#1{(\ref{#1})}

\makeatletter 
\newcommand*{\bigcdot}{}
\DeclareRobustCommand*{\bigcdot}{%
  \mathbin{\mathpalette\bigcdot@{}}%
}
\newcommand*{\bigcdot@scalefactor}{.5}
\newcommand*{\bigcdot@widthfactor}{1.15}
\newcommand*{\bigcdot@}[2]{%
  \sbox0{$#1\vcenter{}$}
  \sbox2{$#1\cdot\m@th$}%
  \hbox to \bigcdot@widthfactor\wd2{%
    \hfil
    \raise\ht0\hbox{%
      \scalebox{\bigcdot@scalefactor}{%
        \lower\ht0\hbox{$#1\bullet\m@th$}%
      }%
    }%
    \hfil
  }%
}
\makeatother

\tcbset
{colframe=boxcolor,colback=symbols!7!pagebackground,coltext=pageforeground,
fonttitle=\bfseries,nobeforeafter,center title,size=fbox,boxsep=1.5pt,
top=0mm,bottom=0mm,boxsep=0mm,tcbox raise base}

\def\two{{\<generic>\kern0.05em\<genericb>}}
\def\twoI{{\<Ito>\kern0.05em\<Itob>}}

\def\mail#1{\burlalt{#1}{mailto:#1}}

\usepackage{thmtools}

\declaretheorem[style=definition]{example}

\begin{document}

\title{Approximations of dispersive PDEs in the presence of low-regularity randomness}

\author{Yvonne Alama Bronsard$^1$, Yvain Bruned$^2$, Katharina Schratz$^1$}
\institute{LJLL (UMR 7598), Sorbonne University \and IECL (UMR 7502), Université de Lorraine\\
Email:\ \begin{minipage}[t]{\linewidth}
\mail{yvonne.alama_bronsard@upmc.fr}, \\
\mail{Yvain.Bruned@universite-lorraine.fr}, \\ \mail{katharina.schratz@sorbonne-universite.fr}.
\end{minipage}} 

\maketitle 

\begin{abstract}
We introduce a new class of numerical schemes which allow for low regularity approximations to the expectation $ \mathbb{E}(|u_{k}(t, v^{\eta})|^2)$, where $u_k$ denotes the $k$-th Fourier coefficient of the solution $u$  of the dispersive  equation and $ v^{\eta}(x) $ the associated random initial data. This quantity plays an important role in physics, in particular in the study of wave turbulence where one needs to adopt a statistical approach in order to obtain deep insight into the {\it generic} long-time behaviour of solutions to dispersive equations. Our new class of schemes is based on  Wick's theorem and Feynman diagrams together with a resonance based discretisation~\cite{BS}  set in a more general context: we introduce a novel combinatorial structure called {paired decorated forests} which are two decorated trees whose decorations on the leaves come in pair.  The character of the scheme draws its inspiration from the treatment of  singular stochastic partial differential equations  via Regularity Structures. In contrast to classical approaches, we do not discretize the PDE itself, but rather its expectation. This allows us to heavily exploit the optimal resonance structure and underlying gain in regularity on the finite dimensional (discrete) level.
\\[.4em]


\end{abstract}
\setcounter{tocdepth}{2}
\setcounter{secnumdepth}{4}
\tableofcontents
\section{Introduction}
We consider  nonlinear dispersive equations
\begin{equation}\label{dis}
\begin{aligned}
& i \partial_t u(t,x) +   \mathcal{L}\left(\nabla\right) u(t,x) =\vert \nabla\vert^\alpha p\left(u(t,x), \overline u(t,x)\right),\quad (t,x) \in  \R \times \T^d\\
\end{aligned}
\end{equation}
with random initial data:
\begin{equation}\label{ic}
\begin{aligned}
u(0,x) = v^{\eta}(x) = \sum_{k\in \Z^d}v_k \eta_{k}e^{ikx},
\end{aligned}
\end{equation}
where $v_k = \overline{v_{-k}}$ and $(\eta_k)_{k\in \Z^d}$ is a family of standard complex Gaussian random variables satisfying
$$
\mathop{\mathbb{E}}(\eta_{k}\eta_{-\ell}) = \delta_{k,\ell}.
$$
Namely, we require $\eta_0$ to be a real normalized centered Gaussian, and for $(\eta_k)_{k\in \Z^d}$ to be independent and to satisfy $\eta_k = \overline{\eta}_{-k}$. We first note that the above setting is for real initial conditions. In the complex case, we simply need to take $(\eta_k)_{k\in \Z^d}$ to be independent standard centered complex Gaussians which satisfy $\mathop{\mathbb{E}}(\eta_{k}\overline{\eta_{\ell}}) = \delta_{k,\ell}$, and refer to the Remark \ref{real-ic}. Working with Gaussian random variables allows us to use Wick's formula for computing the second moment of $u_k(\tau,v^\eta)$. Nevertheless, one can go beyond the Gaussian realm. For a discussion on the subject we refer to Remark \ref{Gaus-ic} .

Furthermore, we assume in \eqref{dis} a polynomial nonlinearity $p$. Moreover, we suppose that there exists a deterministic time $T > 0$ such that the structure of \eqref{dis} implies at least almost-sure local wellposedness of the problem in an appropriate functional space on the finite time interval $]0,T]$. This is a non-trivial assumption since naturally the time of existence is random and not necessarily almost surely bounded from below by a positive time $T$. We refer to the works on probabilistic Cauchy theory given for instance by \cite{BT,BT2, SX, CT,NS, CGI} in the case of nonlinear Schr\"odinger and wave equations. With the above assumption we have that the solution $u$ admits up until time $T$ a second moment restricted on the almost surely set of its definition. In order to keep notations simple, we will omit such a set in the sequel.

The differential operators $\mathcal{L}\left(\nabla \right) $ and $\vert \nabla\vert^\alpha$ shall cast in Fourier  space into the form 
\begin{equs}\label{Lldef}
 \mathcal{L}\left(\nabla \right)(k) = k^\sigma + \sum_{\gamma : |\gamma| < \sigma} a_{\gamma} \prod_{j} k_j^{\gamma_j} ,\qquad  \vert \nabla\vert^\alpha(k) =  \sum_{\gamma : |\gamma| < \sigma} \prod_{j=1}^{d} k_j^{\gamma_j}
\end{equs}
for some $ \alpha \in \R $, $\sigma \in \N$, $ \gamma \in \Z^d $ and $ |\gamma| = \sum_i \gamma_i $,
where  for $k = (k_1,\ldots,k_d)\in \Z^d$ and $m = (m_1, \ldots, m_d)\in \Z^d$ we set \begin{equs}
k^\sigma  = k_1^\sigma + \ldots + k_d^\sigma, \qquad k \cdot m = k_1 m_1 + \ldots + k_d m_d.
\end{equs}
 Concrete examples are discussed in Section \ref{sec:app}, including the cubic nonlinear  Schr\"odinger (NLS) equation
\begin{equation}\label{nlsIntro}
i \partial_t u + \mathcal{L}\left(\nabla\right)  u = \vert u\vert^2 u, \quad \mathcal{L}\left(\nabla\right) = \Delta, \quad \alpha = 0, \quad p(u,\bar{u}) = |u|^2u
\end{equation}
and the Korteweg--de Vries (KdV) equation 
\begin{equation}\label{kdvIntro}
 i\partial_t u +\mathcal{L}\left(\nabla\right) u = i\frac12 \partial_x u^2, \quad \mathcal{L}\left(\nabla\right) = i\partial_x^3, \quad \alpha = 1, \quad p(u,\bar{u}) = i\frac{1}{2}u^2.
\end{equation}
 Note that one could deal with non-polynomial nonlinearities, and work on more general domains than the periodic one, by combining the present work with the general framework established in \cite{BBS}. Indeed, by introducing nested commutator structures the work \cite{BBS} overcomes the necessity of periodic boundary conditions and  polynomial nonlinearities, see also \cite{FS, AB}.  One can use the iterated integrals produced in \cite{BBS} and then proceed in Fourier space as in the present work. Doing all the computations in physical space will require a description of the second moment of the random initial data directly in physical space; however, such a randomisation is in general mostly described in Fourier space.

The aim of this paper is to introduce a new class of  schemes which, by denoting by $\tau$ the time step, allows for a low regularity approximation to
\begin{equation}\label{exp} 
\mathbb{E}(|u_{k}(t, v^{\eta})|^2).
\end{equation}
Here, $\mathbb{E}$ denotes the expectation and $u_k$ the $k$-th Fourier coefficient of the solution $u$ of \eqref{dis} with corresponding initial data $v^{\eta}$ defined at equation \eqref{ic}.
The quantity \eqref{exp} plays an important role in physics, in particular in the study of wave turbulence. A fundamental question in the latter is to derive a rigorous justification of the wave kinetic equation (WKE), an equation which describes the effective dynamics of an interacting wave system in the thermodynamic limit. This yields deep insight into the {\it generic} long-time behaviour of solutions to dispersive equations.
In order to derive this equation one starts with a random initial value and work on a large box of size $ L $. Then under suitable scaling laws 
which depend on the strength $\mu$ of the non-linearity, and the size $ L $ of the box, one can show that in the limit of large $L$ and small $\mu$ the effective dynamic of $ \mathbb{E}(|u_{k}(t, v^{\eta})|^2) $ is given by the WKE over sufficiently long time scales.

 A rigorous derivation of the WKE for NLS is performed in \cite{deng_hani_2021,deng2021derivation,ACG} via  a diagrammatic expansion. A similar rigorous derivation result for a high dimensional $(d \ge 14)$ discrete KdV-type equation was achieved in  \cite{staffilani2021wave} using also Feynman diagrams, in order to derive its associated wave kinetic equation at the kinetic time under a suitable scaling law. All these works are based on the iteration of   Duhamel's formula for \eqref{dis}:
\begin{equs}\label{duhamel}
u(\tau,v^{\eta}) = e^{ i\tau  \mathcal{L}\left(\nabla\right)} v^{\eta}  - i\vert \nabla \vert^\alpha e^{ i\tau  \mathcal{L}\left(\nabla \right)}  \int_0^{\tau} e^{ -i\xi  \mathcal{L}\left(\nabla \right)} p(u(\xi,v^{\eta} ),\bar u(\xi,v^{\eta} ))  d\xi \quad
\end{equs}
which in Fourier space cast into the form
\begin{equs} \label{duhamel_fourier}
u_k(\tau,v^{\eta}) & = e^{ i \tau \mathcal{L}\left(\nabla \right)(k)} \eta_k v_k   - i\vert \nabla \vert^\alpha(k) e^{ i\tau  \mathcal{L}\left(\nabla \right)(k)} \\ & \int_0^{\tau} e^{ -i\xi \mathcal{L}\left(\nabla \right)(k)} p_k(u(\xi,v^{\eta}),\bar u(\xi,v^{\eta} )) d\xi ,
\end{equs}
where $ p_k $ is given for $ p(u,\bar{u}) = |u|^2 u $ by
\begin{equs}
p_k(u, \bar{u}) = \sum_{k = -k_1 + k_2 +k_3 } \bar{u}_{k_1} u_{k_2} u_{k_3}. 
\end{equs}
Iterations of \eqref{duhamel_fourier} produces a tree series, see also 
\cite{Christ,Gub11,oh1} where the series is made explicit for some specific equations. This expansion can also be described via random tensors introduced in \cite{deng2020random}.   
From this tree series, one can compute 
$ \mathbb{E}(|u_{k}(\tau, v^{\eta})|^2) $ using Wick's theorem and Feynman diagrams  \cite{deng_hani_2021,deng2021derivation}. 

The theoretical works mentioned above on the study of wave turbulence, together with the idea of a  {resonance based discretisation}  \cite{BS}, set the inspiration for this work. From an algebraic point of view we, however, have to work in a much more general context considering trees with more decorations than introduced in \cite{BS}, with a new combinatorial structure called \emph{paired decorated forests}. This will allow us to approximate a large class of dispersive PDEs with one general scheme, without the need of treating each and every equation separately.  

In contrast to classical approaches,  we do not discretize the PDE itself, but rather its expectation. This   allows us to  heavily exploit the optimal resonance structure and underlying gain in regularity on the discrete level. More precisely, the iteration of Duhamel's formulation \eqref{duhamel} of \eqref{dis} can be expressed using decorated trees, see \cite{BS}.  Namely, it follows from \cite[Section 1.2]{BS} that the $k-$th Fourier coefficient of the approximated solution at order $  r $ is given by 
\begin{equs} \label{decoratedV1}
U_{k}^{r}(\tau, v^{\eta}) & = \sum_{T \in \CV^r_k} \frac{\Upsilon^{p}(T)(v^{\eta})}{S(T)} \left( \Pi T \right)(\tau)
\end{equs}
where $ \CV^r_k $ is a set of decorated trees which incorporate the frequency $k$, $ S(T) $ is the symmetry factor associated to the tree $ T $, $ \Upsilon^{p}(T) $ is the coefficient appearing in the iteration of Duhamel's formulation and $ (\Pi T)(\tau) $ represents a Fourier iterated integral. The exponent $ p $ corresponds to the nonlinearity appearing in the right hand side of \eqref{dis}. The exponent $r$ in $ \CV^r_k $ means that we consider only 
trees of size at most $ r +1 $, these are the trees corresponding to an iterated integral of depth at most $r+1$. These quantities are described in detail in Section \ref{sec::4}. Generally speaking, the sum \eqref{decoratedV1} is the truncation at order $ r $ of the infinite series describing formally the solution of \eqref{dis}. The low regularity scheme for \eqref{dis} is then obtained by replacing each oscillatory integral $ (\Pi T)(\tau) $ appearing in the finite sum \eqref{decoratedV1} by a low regularity approximation that embeds the resonance structure into the  numerical discretization. We will denote the latter low-regularity approximation operator by $ \Pi^{n,r} $. Here, $ r $ corresponds to the order of the discretization and $ n $ is the a priori regularity assumed on the initial data $ v^{\eta} $. Namely, we assume $v^{\eta}\in H^n$, where $H^n$ is a Sobolev space.  The general scheme then takes the form:
\begin{equs} \label{general_scheme_BS}
U_{k}^{r}(\tau, v^{\eta}) & = \sum_{T \in \CV^r_k} \frac{\Upsilon^{p}(T)(v^{\eta})}{S(T)} \left( \Pi^{n,r} T \right)(\tau).
\end{equs}
The  local error structure for each approximated iterated integral is given by
\begin{equs}\label{eq:loci}
\left(\Pi T - \Pi^{n,r} T \right)(\tau)  = \mathcal{O}\left( \tau^{r+2} \mathcal{L}^{r}_{\text{\tiny{low}}}(T,n) \right),
\end{equs}
where $\mathcal{L}^{r}_{\text{\tiny{low}}}$ involves only lower order derivatives.
The form of the scheme draws its inspiration from the treatment of  singular stochastic partial differential equations (SPDEs) via Regularity Structures in \cite{reg,BHZ,BCCH,EMS}. These decorated trees expansions are generalization of the B-series widely used for ordinary differential equations, we refer to \cite{Butcher72,CCO08,HLW,Hans}.
 In the end, one obtains  an approximation of $u $ under much lower regularity assumptions than classical methods (e.g., splitting methods, exponential integrators \cite{R1,Faou12,HLW,HochOst10,HLRS10,Law67,R3,LR04,McLacQ02,SanBook}) require, which in general introduce the local error 
\begin{equs}\label{eq:lociClass}
\mathcal{O}\left( \tau^{r+2} \mathcal{L}^{r}(T,n) \right)
\end{equs}
involving the full high order differential operator $\mathcal{L}^{r}$. Indeed, we have that  $\mathcal{D}( \mathcal{L}_{\text{\tiny{low}}}) \supset \mathcal{D}( \mathcal{L})$, meaning that the local error structure \eqref{eq:loci} allows us to deal with a  {rougher} class of solutions than the classical error bound \eqref{eq:lociClass}. The underlying idea behind these low-regularity approximations was initiated by the work of \cite{OS}, and then generalized by \cite{BS, BBS} to higher order methods, allowing for approximations to large classes of equations. 
Although the error bound \eqref{eq:loci} derived here are formal, one can obtain rigorous low-regularity error bounds using a classical Sobolev space setting (see \cite{FS, AB}), as well as 
sharp $L^2$ error estimates by using discrete Strichartz estimates and discrete Bourgain spaces (see \cite{IZ, ORS19, ORS20,RSKDV, JORS}). 
Let us mention that the local error analysis can be nicely understood via a Birkhoff factorisation of the character $ \Pi^{n,r} $ (see \cite{BS,BE}) that involves a deformed Butcher-Connes-Kreimer coproduct (see \cite{Butcher72,CK,CKI,BS,BM}). In the present work, we give a recursive description with various Taylor remainders in Definition~\ref{def:Llow}. This formulation is similar to \cite{BBS} where the Birkhoff factorisation is not available. The Birkhoff factorisation can also be used in our work for a more precise description of the Taylor remainders. We also believe that it could have an impact on the general description of symmetries for low regularity schemes. For the moment, only symmetric and symplectic schemes have been work out by hand (see \cite{AB-sym,MS}). Lastly, we note that the scheme \eqref{general_scheme_BS} has been generalized to non-polynomial nonlinearities  and to parabolic equations in \cite{BBS} with the use of nested commutators first introduced in \cite{FS}.
 
A natural route to obtain a low regularity approximation to $ \mathbb{E}(|u_{k}(\tau, v^{\eta})|^2) $ would   be to first compute the low regularity approximation $U^{n,r}_{k}(\tau, v^{\eta})$ given in \eqref{general_scheme_BS} to $u_{k}(\tau, v^{\eta}) $, and then to evaluate its expectation $ \mathbb{E}(|U^{n,r}_{k}(\tau, v^{\eta})|^2) $  while removing the terms of higher order. By doing so, we would, however, ignore the enhanced resonance structure  and gain in regularity introduced by applying the expectation $ \mathbb{E} $. In order to exploit the latter on the discrete level, we will directly compute a low regularity approximation to the second moment of $ U^{r}_k $, where is $U^{r}_k$ defined in \eqref{general_scheme_BS} and is the truncation at order $ r $ of the tree series describing $ u_k $. Namely, we are interested in computing the following series:
\begin{equs}\label{series-trunc}
V_{k}^{r}(\tau, v) & = \CQ_{\leq r+1} \mathbb{E}(|U_{k}^{r}(\tau, v^{\eta})|^2)
\\ & =  \sum_{F =T_1 \cdot T_2 \in \CG^r_k}  m_{F}  \frac{\bar{\Upsilon}^{p}(T_1)(v) \, \Upsilon^{p}(T_2)(v)}{S(T_1) S(T_2)} \left( \bar{\Pi} T_1 \right)(\tau)  \, \left( \Pi T_2 \right)(\tau)
\end{equs}
where $ \CQ_{\leq r+1} $ keeps only the terms of order less than or equal to $ \tau^{r+1} $, and $ m_F $ is a combinatorial coefficient associated to $ F $. Namely, $m_F$ counts the number of pairings that produce $ F $.
It remains to describe the set $ \CG^r_k $.
 For this purpose we  introduce a new combinatorial structure called paired decorated forest which are two decorated trees whose decorations on the leaves come in pair. This pairing among the leaves come from the use of Wick's theorem for products of Gaussian random variables that appear in the computation of $ V_k^r $. Similar structures have been used in the theoretical analysis of  dispersive PDEs \cite{deng_hani_2021,deng2021derivation,ACG,staffilani2021wave} as well as in stochastic differential equations  and exotic B-series \cite{LV, Bronasco} . 
 The set $ \CG^r_k $ consists of all paired forests $ F = T_1 \cdot T_2 $ of size at most $ r +1 $, meaning that the term $ \left( \bar{\Pi} T_1 \right)(\tau)  \, \left( \Pi T_2 \right)(\tau)$ contains at most $ r+1 $ integrals in time. $ T_1$  and $ T_2 $ are the two paired trees and  $ \cdot  $ denotes the forest product. 
 Further, in the above series we have used the following short hand notation:
\begin{equs}
\bar{\Upsilon}^{p}(T_1)(v) = \overline{\Upsilon^{p}(T_1)(v)}, \quad \left( \bar{\Pi} T_1 \right)(\tau) = \overline{ \left( \Pi T_1 \right)(\tau) },
\end{equs}
which stands for the complex conjugate, and will be abundantly used throughout the reminder of this article.
Notice that the above sum \eqref{series-trunc} is not symmetric, meaning that both $ \left( \bar{\Pi} T_1 \right)(\tau)  \, \left( \Pi T_2 \right)(\tau) $ and $ \left( \bar{\Pi} T_2 \right)(\tau)  \, \left( \Pi T_1 \right)(\tau)$ appear when $ T_1 \neq T_2 $.
 We then proceed with the discretization in time of \eqref{series-trunc} as is done in \cite{BS}: we replace $ (\Pi T_i)(\tau) $ by their resonance based approximations $ (\Pi^{n,r} T_i) $ which yields the following tree series:
\begin{equs} \label{schema_general}
V_{k}^{n,r}(\tau, v) & = \sum_{F= T_1 \cdot T_2 \in \CG^r_k} m_F \frac{\bar{\Upsilon}^{p}(T_1)(v) \, \Upsilon^{p}(T_2)(v)}{S(T_1) S(T_2)} \\ &  \mathcal{Q}_{\leq r +1} \left( \bar{\Pi}^{n,r} T_1  \Pi^{n,r} T_2 \right)(\tau) .
\end{equs}
The local error structure is given by
\begin{equs} \label{eq:loci2}
\left(\bar{\Pi} T_1 \Pi T_2 - \CQ_{\leq r+1}\Pi^{n,r} T_1 \bar{\Pi}^{n,r} T_2 \right)(\tau)  = \mathcal{O}\left( \tau^{r+2} \mathcal{L}^{r}_{\text{\tiny{low}}}(T_1 \cdot T_2,n) \right).
\end{equs}
The definition of $ \mathcal{L}^{r}_{\text{\tiny{low}}}(T_1 \cdot T_2,n)  $ is, however, much more involved than in the work of \cite{BS}. Indeed, given that we consider decorated forests $F = T_1\cdot T_2$, whose leaves (which encodes the frequency) come in pairs, we have that the same frequencies can appear in both the decorated tree $T_1$ and that of $T_2$. Hence, unlike in the work of \cite {BS}, the set of frequencies appearing on the leaves of $T_1$ and $T_2$ are not disjoint, and in consequence a new framework needs to be considered.
Moreover, since our calculations are made with these pairings $T_1\cdot T_2$, cancellations appear when computing the variance \eqref{exp} and the schemes obtained in this work are of simpler form than those obtained in the work of \cite{BS}.

Our main result is the general class of schemes given by \eqref{schema_general} together with the correct combinatorial structures on paired decorated forests, which yields a low-regularity approximation of \eqref{exp}. Our construction is based upon Wick's formula (see Proposition~\ref{Wick}) and is presented in Proposition~\ref{prop_moment_2}. The local error structure \eqref{eq:loci} is then given in Theorem~\ref{thm:genloc}. Theorem~\ref{thm:genloc} is heavily based on  Theorem~\ref{approxima_tree_paired} which introduces a new recursive definition of  $ \mathcal{L}^{r}_{\text{\tiny{low}}}(T_1 \cdot T_2,n)  $ (see Definition~\ref{def:Llow}) which differs from the one presented in \cite{BS}.

The new scheme \eqref{schema_general} is an attempt to implement the low regularity integrators introduced in \cite{BS} when randomness is introduced at the level of the initial data. We expect that the scheme described in \cite{BS,BBS} will be able to enhance the error structure of classical methods for (dispersive) stochastic PDEs (see, e.g. \cite{Charly,Arnaud,Lord} and the references therein) for low regularity initial data where randomness is added via a stochastic forcing.

\begin{remark} \label{real-ic}
	Our assumption on the initial data in \eqref{ic} enforces the initial data to be real. Our scheme is fairly general and works also beyond dispersive PDEs,  such as for parabolic PDEs. In \cite{deng2021derivation,deng_hani_2021}, the authors use a different assumption on the noise $(\eta_k)_{k\in \Z^d}$ which is a family of independent identically distributed complex random variables (centered standard Gaussian variables or uniform distributions on the unit circle) satisfying: For all $k,\ell \in\Z^d$
	\begin{equs}
		\mathbb{E}(|\eta_k|^2) = 1, \quad \mathbb{E}(\eta_k \eta_\ell) = 0.
	\end{equs}
In the Gaussian case see also \cite{CollotGermain}.
The second identity in the above simplifies a lot the interactions that need to be considered for building up the Feynman diagrams. 
Indeed, only the terms $ \mathbb{E}(\eta_k \overline{\eta_\ell}) $  from Wick's theorem are potentially non-zero. These terms are the easiest to treat requiring no extra regularity, (see Section~\ref{ex} and \eref{ex:trees} where the pairing $ T_0 \cdot T_{1,1} $ is non longer possible, it is the pairing that asks more regularity in the discretisation, see \eqref{ex:T11}). 
	\end{remark}

\begin{remark}\label{Gaus-ic}
	In \eqref{series-trunc}, the summation is performed on paired decorated forests of the form $ T_1 \cdot T_2 $. The fact that we have only two trees and not more correspond to the computation of the second moment. For the $p$-th moment, one will get $p$ paired decorated trees where the leaves will come in pairs. In both cases second moment or higher moments, one uses Wick's theorem applied to product of Gaussian random variables for computing the pairings.
	One can potentially go beyond the Gaussian realm by using cumulants for computing the moments of a given random variable. This was extensively used in the context of singular SPDEs (see \cite{CH16} for a general convergence result of renormalised models and \cite{DH21,BN} for recent works in the discrete setting). In our context, one can imagine to have trees where the leaves are split into different sets of cardinality bigger than two. The scheme will be robust to this case.
	\end{remark}

Before introducing our general framework we will in the next section illustrate our main idea on an example taken from NLS. 
\subsection{Example}\label{ex}
Let us consider the equation:
\begin{equs} \label{equa_example}
i \partial_t u + \Delta  u = \vert u\vert^2 u, \quad u(0,x) = v^{\eta}(x) =  \sum_{k\in \Z^d}v_k \eta_{k} e^{ikx}
\end{equs}
set on the $d$ dimensional torus. We will discuss the approximation of $\mathbb{E}\left(|u_k(\tau, v^{\eta})|^2\right)$ to first order. We rewrite \eqref{equa_example} using Duhamel's formula:
\begin{equs}
		u(\tau, v^{\eta}) = e^{i \tau \Delta } v^{\eta} - i e^{i\tau\Delta}\int_0^\tau e^{-i\xi \Delta} \left(  \overline{	u(\xi, v^{\eta})} 	u(\xi, v^{\eta} )^2 \right) d\xi.
\end{equs}
One obtains the following equation for the Fourier coefficient  $ u_k(\tau, v^{\eta}) $:
\begin{equs} \label{duhamel_Fourier}
	\begin{aligned}
	u_k(\tau, v^{\eta}) & = e^{- i \tau k^2 } \eta_k  v_k -\sum_{\substack{k_1,k_2,k_3 \in \Z^d\\-k_1+k_2+k_3 = k} } i e^{ -i \tau k^2} \\ & \int_0^\tau e^{i\xi k^2} \overline{	u_{k_1}(\xi, v^{\eta})} 	u_{k_2}(\xi, v^{\eta} )	u_{k_3}(\xi, v^{\eta})d\xi,
	\end{aligned}
\end{equs}
where the linear operator $ e^{i \tau \Delta} $ (resp.  $ e^{-i \tau \Delta} $) is sent to $  e^{ - i \tau k^2}  $ (resp. $  e^{ i \tau k^2}  $) in Fourier space. The splitting of $ k $ into $ k_1, k_2, k_3 $ comes from the fact that moving to Fourier space, products become convolution products on the frequencies.
 We fix $r=0$ and iterate \eqref{duhamel_Fourier} by replacing $ u_{k_j}(\tau, v^{\eta}) $ by
 \begin{equs}
 	u_{k_j}(\tau, v^{\eta}) =  e^{- i \tau k_j^2 } \eta_{k_j}  v_{k_j} +  \CO(\tau),
 \end{equs}
with $ j \in \lbrace 1,2,3 \rbrace  $.
 We obtain the following first order approximation of the $k$-th Fourier coefficient $ u_k(\tau, v^{\eta}) $:
\begin{equs}
u_k(\tau,v^{\eta}) & = e^{ -i \tau  k^2} \eta_k v_k   -\sum_{\substack{k_1,k_2,k_3 \in \Z^d\\-k_1+k_2+k_3 = k} } i e^{ -i \tau k^2} \\  & \int_0^{\tau} e^{ i\xi k^2} (e^{i\xi k_1^2} \overline{\eta_{k_1}} \overline{v_{k_1}}) (e^{-i\xi k_2^2} \eta_{k_2} v_{k_2}) (e^{-i\xi k_3^2} \eta_{k_3} v_{k_3}) d\xi  + \CO(\tau^2).
\end{equs}
We encode the above Duhamel iterates using the tree series \eqref{decoratedV1}. Namely, on the continuous level, from \eqref{decoratedV1} we have
\begin{equs}\label{ex:nls1}\begin{aligned}
U^{0}_k(\tau,v^{\eta}) & = u_k^0(\tau,v^{\eta}) +  u_k^1(\tau,v^{\eta}),
\\ u_k^0(\tau,v^{\eta}) & =  \frac{\Upsilon^{p}(T_0)(v^{\eta})}{S(T_0)} \left( \Pi T_0 \right)(\tau), \\ u_k^1(\tau,v^{\eta}) &= \sum_{\substack{k_1,k_2,k_3 \in \Z^d\\-k_1+k_2+k_3 = k} } \frac{\Upsilon^{p}(T_1)(v^{\eta})}{S(T_1)} \left( \Pi T_1 \right)(\tau),
\end{aligned}
\end{equs}
%

where 
\begin{equs}\label{T1k}
 T_0 & = \begin{tikzpicture}[scale=0.2,baseline=-5]
\coordinate (root) at (0,1);
\coordinate (tri) at (0,-1);
\draw[kernels2] (tri) -- (root);
\node[var] (rootnode) at (root) {\tiny{$ k $}};
\node[not] (trinode) at (tri) {};
\end{tikzpicture} ,   \quad T_1 =  \begin{tikzpicture}[scale=0.2,baseline=-5]
\coordinate (root) at (0,-1);
\coordinate (t3) at (0,1);
\coordinate (t4) at (0,3);
\coordinate (t41) at (-2,5);
\coordinate (t42) at (2,5);
\coordinate (t43) at (0,7);
\draw[kernels2] (t3) -- (root);
\draw[symbols] (t3) -- (t4);
\draw[kernels2,tinydots] (t4) -- (t41);
\draw[kernels2] (t4) -- (t42);
\draw[kernels2] (t4) -- (t43);
\node[not] (rootnode) at (root) {};
\node[not] (rootnode) at (t4) {};
\node[not] (rootnode) at (t3) {};
\node[var] (rootnode) at (t41) {\tiny{$ k_1 $}};
\node[var] (rootnode) at (t42) {\tiny{$ k_{\tiny{3}} $}};
\node[var] (rootnode) at (t43) {\tiny{$  k_{\tiny{2}} $}};
\end{tikzpicture}, \quad S(T_0) = 1, \quad S(T_1) = 2, \\
  \Upsilon^{p}(T_0)(v^{\eta}) & = \eta_k v_{k}, \quad  \Upsilon^{p}(T_1)(v^{\eta})  = 2 \bar{\eta}_{k_1} \bar{v}_{k_1} \eta_{k_2} v_{k_2}\eta_{k_3} v_{k_3},
  \\ (\Pi T_0)(\tau)  & = e^{-i \tau k^2}, \quad (\Pi T_1)(\tau) = -i e^{-i\tau k^2} \int^{\tau}_{0} e^{i \xi k^2} e^{i \xi k_1^2} e^{-i \xi k_2^2} e^{-i \xi k_3^2} d \xi.
\end{equs}
For the decorated trees written above, we have used the following coding:  an edge  $ \<thick> $ (resp. $ \<thick2> $) corresponds to a factor $ e^{-i t k^2} $ (resp. $ e^{i t k^2} $), while an edge  $\<thin>$ (resp. $\<thin2>$) corresponds to an integral $ -i \int^{\tau}_0 e^{i \xi k^2} \cdots d \xi $ (resp. $ -i\int^{\tau}_0 e^{-i \xi k^2} \cdots d \xi $). The dash dotted line introduces a minus sign. The decorations on the leaves correspond to the frequencies that add up in the intern node with a minus sign if one faces a dash edge. In $ T_1 $, the two inner nodes are decorated by $ k = - k_1 + k_2 + k_3 $ where $ - k_1 $ comes from the edge $ \<thick2> $. The definition of the symmetry factor $S(T)$ and the coefficients $\Upsilon^{p}(T)$ are defined in Section \eqref{sec::4}. The symmetry factor corresponds to the number of internal symmetries of the tree taking the edge decorations into account but not the node decorations, which explains why $S(T_1) = 2$.

By construction, $U^0_k$ is a first order truncated approximation of $u_k$, and it follows from \eqref{ex:nls1} that 
\begin{equs}
\begin{aligned}
{\mathbb{E}}(|U^{0}_k(\tau,v^{\eta})|^2) &= {\mathbb{E}}\left((u^0_k(\tau,v^{\eta})+ u^1_k(\tau,v^{\eta}))(\overline{u^0_k(\tau,v^{\eta})+ u^1_k(\tau,v^{\eta})})\right)\\\nonumber
&= {\mathbb{E}}(|u^0_k(\tau,v^{\eta})|^2) + 2Re{\mathbb{E}}\left(\bar{u}^0_k(\tau,v^{\eta})u^1_k(\tau,v^{\eta})\right) +  \mathcal{O}(\tau^2).
\end{aligned}
\end{equs}
The main idea is then that we can express on the continuous level the above expansion using the tree series \eqref{series-trunc}, and on the discrete level by the tree series \eqref{schema_general}. We show how this is done by detailing the computation of the term 
$$\mathbb{E}( \bar{u}_k^0(\tau,v^{\eta}) u_k^1(\tau,v^{\eta})).$$
It follows from \eqref{T1k} that we need to consider
\begin{equs} \label{pairings}
\mathbb{E} \left(  \overline{\Upsilon^{p}(T_0)(v^{\eta})} \Upsilon^{p}(T_1)(v^{\eta}) \right) & =2 \bar{v}_{k} \bar{v}_{k_1} v_{k_2} v_{k_3}\mathbb{E} \left( \bar{\eta}_{k} \bar{\eta}_{k_1}\eta_{k_2} \eta_{k_3} \right).
\end{equs}
Further, we have,
\begin{equs}
\mathbb{E} \left( \bar{\eta}_{k} \bar{\eta}_{k_1}\eta_{k_2} \eta_{k_3} \right) & = 
\mathbb{E} \left( \bar{\eta}_{k} \bar{\eta}_{k_1} \right) \mathbb{E} \left( \eta_{k_2} \eta_{k_3} \right) + \mathbb{E} \left( \bar{\eta}_{k} \eta_{k_2} \right) \mathbb{E} \left( \bar{\eta}_{k_1} \eta_{k_3} \right) \\ & + \mathbb{E} \left( \bar{\eta}_{k} \eta_{k_3} \right) \mathbb{E} \left( \bar{\eta}_{k_1} \eta_{k_2} \right)
\\ & = \delta_{k,\bar{k}_1} \delta_{k_2,\bar{k}_3} + \delta_{k,k_2} \delta_{k_1,k_3}  + \delta_{k,k_3} \delta_{k_1,k_2} 
\end{equs}
 where we have used Wick's formula (see Proposition~\ref{Wick}) for products of Gaussian random variables. The different $ \delta_{\cdot,\cdot} $ correspond  to different pairings and fix completely the frequencies on the leaves. We will introduce a new combinatorial object for describing these pairing called  \emph{decorated paired forests}. Below, we represent the various pairings coming from \eqref{pairings}:
 \begin{equs}\label{ex:trees}
T_{0} &  = \begin{tikzpicture}[scale=0.2,baseline=-5]
\coordinate (root) at (0,1);
\coordinate (tri) at (0,-1);
\draw[kernels2] (tri) -- (root);
\node[var] (rootnode) at (root) {\tiny{$ k $}};
\node[not] (trinode) at (tri) {};
\end{tikzpicture} ,   \quad T_{1,1} =  \begin{tikzpicture}[scale=0.2,baseline=-5]
\coordinate (root) at (0,-1);
\coordinate (t3) at (0,1);
\coordinate (t4) at (0,3);
\coordinate (t41) at (-2,5);
\coordinate (t42) at (2,5);
\coordinate (t43) at (0,7);
\draw[kernels2] (t3) -- (root);
\draw[symbols] (t3) -- (t4);
\draw[kernels2,tinydots] (t4) -- (t41);
\draw[kernels2] (t4) -- (t42);
\draw[kernels2] (t4) -- (t43);
\node[not] (rootnode) at (root) {};
\node[not] (rootnode) at (t4) {};
\node[not] (rootnode) at (t3) {};
\node[var] (rootnode) at (t41) {\tiny{$ \bar{k} $}};
\node[var] (rootnode) at (t42) {\tiny{$ \bar{k}_1 $}};
\node[var] (rootnode) at (t43) {\tiny{$  k_{1} $}};
\end{tikzpicture}   \quad T_{1,2} =  \begin{tikzpicture}[scale=0.2,baseline=-5]
\coordinate (root) at (0,-1);
\coordinate (t3) at (0,1);
\coordinate (t4) at (0,3);
\coordinate (t41) at (-2,5);
\coordinate (t42) at (2,5);
\coordinate (t43) at (0,7);
\draw[kernels2] (t3) -- (root);
\draw[symbols] (t3) -- (t4);
\draw[kernels2,tinydots] (t4) -- (t41);
\draw[kernels2] (t4) -- (t42);
\draw[kernels2] (t4) -- (t43);
\node[not] (rootnode) at (root) {};
\node[not] (rootnode) at (t4) {};
\node[not] (rootnode) at (t3) {};
\node[var] (rootnode) at (t41) {\tiny{$ k_1 $}};
\node[var] (rootnode) at (t42) {\tiny{$ k_1 $}};
\node[var] (rootnode) at (t43) {\tiny{$  k $}};
\end{tikzpicture}  \quad T_{1,3} =  \begin{tikzpicture}[scale=0.2,baseline=-5]
\coordinate (root) at (0,-1);
\coordinate (t3) at (0,1);
\coordinate (t4) at (0,3);
\coordinate (t41) at (-2,5);
\coordinate (t42) at (2,5);
\coordinate (t43) at (0,7);
\draw[kernels2] (t3) -- (root);
\draw[symbols] (t3) -- (t4);
\draw[kernels2,tinydots] (t4) -- (t41);
\draw[kernels2] (t4) -- (t42);
\draw[kernels2] (t4) -- (t43);
\node[not] (rootnode) at (root) {};
\node[not] (rootnode) at (t4) {};
\node[not] (rootnode) at (t3) {};
\node[var] (rootnode) at (t41) {\tiny{$ k_1 $}};
\node[var] (rootnode) at (t42) {\tiny{$ k $}};
\node[var] (rootnode) at (t43) {\tiny{$  k_{1} $}};
\end{tikzpicture} 
\\ F_i  &= T_0 \cdot T_{1,i}, \, i \in \lbrace  1,2,3 \rbrace
 \end{equs} 
 where  $ \bar{k}_1 = - k_1 $ and $ \cdot  $ denotes the forest product. The $ F_i $ are decorated forests. One can observe that $ F_2 $ and $ F_3 $ are similar due to the symmetry of the decorated tree $ T_1 $. Hence, we only need to consider one of the decorated forests during the computations, we choose to consider $F_2$. Further, it follows in an analogous fashion from \eqref{ex:nls1} and \eqref{T1k} that there is a single nontrivial pairing when computing ${\mathbb{E}}|u^0_k(\tau,v^{\eta})|^2$, and that the decorated forest 
 $$F_0 = T_0 \cdot T_0$$
  encodes this pairing and consequently the term ${\mathbb{E}}|u^0_k(\tau,v^{\eta})|^2$.
 Therefore, the set $ \CG^0_k $ is defined as:
\begin{equs}
\CG^0_k = \lbrace T_0 \cdot T_0, T_0 \cdot T_{1,1}, T_0 \cdot T_{1,2}, \, k_1 \in \Z^d \rbrace.
\end{equs} 
 These decorated forests encode products of iterated integrals, where each iterated integral is encoded by a decorated tree of the form \eqref{ex:trees}. Namely, we have
 \begin{equs}
   (\bar{\Pi} T_0)(\tau) \,  \, (\Pi T_{1,1})(\tau) 
 & = e^{i \tau k^2} \left(  -i e^{-i\tau k^2} \int^{\tau}_{0} e^{i \xi k^2} e^{i \xi k^2} e^{-i \xi k_2^2} e^{-i \xi k_2^2} d \xi \right)
 \\ & = - i \int^{\tau}_{0} e^{2 i \xi ( k^2 - k_2^2)} d \xi ,
 \\
  (\bar{\Pi} T_0)(\tau) \,  \, (\Pi T_{1,2})(\tau) 
 & = e^{i \tau k^2} \left(  -i e^{-i\tau k^2} \int^{\tau}_{0} e^{i \xi k^2} e^{i \xi k_1^2} e^{-i \xi k^2} e^{-i \xi k_1^2} d \xi \right)
 \\ & = - i \tau .
 \end{equs}
 One observes that for $ F_2 $, due to cancellations in the frequencies, one can perform exact integration in time, and no approximation step is required. The term which requires to be approximated and which asks for some regularity on the initial data is the integral encoded by $ F_1 $. We now discuss how the approximations are made, and how it leads us to the tree series \eqref{schema_general}.
 
 Let us first proceed by describing the idea behind the approximation of $ \Pi T_1 $. 
The main idea of the general resonance based scheme introduced in \cite{BS} is to split the frequencies in the integrand into the dominant and lower parts:
\begin{equs}\label{resDecomp}
 \mathcal{L} & = k^2 + k_1^2 - k_2^2 -k_3^2 = \mathcal{L}_{\tiny{\text{dom}}} + \mathcal{L}_{\tiny{\text{low}}} \\\nonumber
 \mathcal{L}_{\tiny{\text{dom}}} & = 2 k_1^2, \quad \mathcal{L}_{\tiny{\text{low}}} = - k_1 (k_2 + k_3) + k_2 k_3
\end{equs}  
such that (cf. \eqref{T1k})
\begin{equs}\label{T1k2}
(\Pi T_1)(\tau)  
 =  -i e^{-i\tau k^2} \int^{\tau}_{0} e^{i \xi  \mathcal{L}_{\text{\tiny{dom}}}} e^{i \xi  \mathcal{L}_{\tiny{\text{low}}}} d \xi.
\end{equs}
The idea of the splitting \eqref{resDecomp} lies in the observation that $ k_1^2 $ asks for two derivatives on the initial data while the cross-terms given in $\mathcal{L}_{\text{\tiny{low}}}$   require only one. Moreover, the term $ e^{i \xi \mathcal{L}_{\text{\tiny{dom}}}} $ can be integrated out exactly and mapped back into physical space. In order to gain regularity, we will thus only   Taylor-expand the lower order part $  e^{i \xi \mathcal{L}_{\text{\tiny{low}}}}$  in \eqref{T1k2} while integrating  the dominant part $ e^{i \xi\mathcal{L}_{\text{\tiny{dom}}}} $ exactly. This provides the following discretisation:
\begin{equs}
(\Pi T_1 )(\tau) & = - i e^{-i \tau k^2} \int^{\tau}_0 e^{2i \xi k_1^2}(1 + \mathcal{O}(\xi \mathcal{L}_{\tiny{\text{low}}}  )) d \xi \\
& =   (\Pi^{1,0} T_1 )(\tau) + \mathcal{O}(\tau^2 \mathcal{L}_{\tiny{\text{low}}}  )
\\ (\Pi^{1,0} T_1 )(\tau) & = - i e^{-i \tau k^2} \int^{\tau}_0 e^{2i \xi k_1^2} d \xi 
\end{equs} 
where the index $ 0 $ correspond to the order of the approximation inside the time integral and the index $ 1 $ embeds that we ask a priori one derivative on the initial data. If we assume more regularity, such as for instance two derivatives, one can perform the full Taylor expansion and obtains:
\begin{equs}
(\Pi T_1 )(\tau) =  (\Pi^{2,0} T_1 )(\tau) + \mathcal{O}(\tau^2 \mathcal{L}  ), \quad (\Pi^{2,0} T_1 )(\tau) =  - i \tau e^{-i \tau k^2}.
\end{equs}
We would like to proceed in the same manner for the approximation of $ (\Pi T_{1,1})(\tau) $. However, we have to choose the splitting as follows:
\begin{equs}\label{Kd}
\mathcal{L} = 2 k^2 - 2 k_2^2, \quad \mathcal{L}_{\tiny{\text{low}}} = \mathcal{L}, \quad \mathcal{L}_{\tiny{\text{dom}}} = 0.
\end{equs}
This is due to the fact that  $  \frac{1}{k^2 - k_1^2} $ cannot be mapped back into Physical space and hence, we would  recover a scheme which needs to be computed fully in Fourier space and where we could not make use of the Fast Fourier Transform (FFT). This would not allow for a practical implementation of the scheme and would lead to high computational and memory costs, in particular in higher spatial dimensions $d\geq 3$. This motivates our choice of  dominant and lower order parts   \eqref{Kd} and leads to the discretization
\begin{equs}\label{ex:T11}
( \Pi T_{1,1} )(\tau) =  (\Pi^{1,0} T_{1,1} )(\tau) + \mathcal{O}(\tau^2 \mathcal{L}  ), \quad (\Pi^{1,0} T_{1,1} )(\tau) =  - i \tau e^{-i \tau k^2}.
\end{equs}
The discretization (c.f. \eqref{schema_general}) is then given by:
\begin{equs}
 \left( \bar{\Pi}^{1,0} T_0 \right)(\tau) \,\, \left( \Pi^{1,0} T_{1,j} \right)(\tau) = - i \tau. 
\end{equs}
The  term
\begin{equs}
\mathbb{E} \left( \bar{u}_k^0(\tau,v^{\eta}) u_k^1(\tau,v^{\eta}) \right)
\end{equs}
is then well approximated by
\begin{equs}
  \frac{\bar{\Upsilon}^{p}(T_0)(v^{\eta})}{S(T_0)} & \, \left(  \sum_{\substack{k_1 \in \Z^{d}} } \frac{\Upsilon^{p}(T_{1,1})(v^{\eta})}{S(T_{1,1})}  \CQ_{\leq 1} \left( \bar{\Pi} T_0 \right)(\tau) \left( \Pi T_{1,1} \right)(\tau)
  \right. 
  \\ & \left. + 2 \sum_{\substack{k_1 \in \Z^{d}} } \frac{\Upsilon^{p}(T_{1,2})(v^{\eta})}{S(T_{1,2})}  \CQ_{\leq 1} \left( \bar{\Pi} T_0 \right)(\tau) \left( \Pi T_{1,2} \right)(\tau)  \right)
  \\  & =  - 3 i\tau |v_k|^2 \sum_{k_1 \in \Z^d} |v_{k_1}|^2  
\end{equs}
and introduces the following local error
\begin{equs}
\mathcal{O}( \tau^2 \sum_{k_1 \in \Z^{d}} \left(  k^2 - k_1^2 \right) |v_k|^2  |v_{k_1}|^2 ).
\end{equs}
We notice that a term of the form $ k^2 |v_k|^2 $ corresponds (up to a sign) to the Fourier coefficient of $ \nabla v * \nabla\tilde{v} $ where $ * $ is the space convolution, and $\tilde{v}(x) = \overline{v(-x)}$.  Therefore, in contrast to \cite{BS}, a full Taylor expansion does not ask more than one derivative in space on the initial data. In fact, $ T_0 \cdot T_{1,1} $ is the term in $\CG_k^0$ to compute which requires the most regularity on the initial data, namely a factor of $k^2$ for a first order approximation, and which corresponds in the end to only one derivative due to the previously mentioned convolution structure. In most of the examples in this paper, we end up by considering only full Taylor expansions, as is made in \eqref{ex:T11}. This is however not always the case in a general setting. We give an example of a system of PDEs where a more careful resonance analysis (c.f. \eqref{resDecomp}) is needed to obtain an approximation at low regularity. Let us for instance consider the following system:
\begin{equs}\label{systK}
i \partial_t u + \Delta u  & = |u|^2 u,  \quad u(0) = v_1,\\
 \partial_t v + \partial_x^p v & = |u|^2 u,  \quad v(0) = v_2.
\end{equs}
Then depending on the value of $ p $, we may need a resonance based approach. Iterating Duhamel's formula, we get the following oscillatory integrals:
\begin{equs}
I = \int_0^{\tau} e^{-\xi (ik)^p} e^{-i\xi k_1^2} e^{i\xi k_2^2} e^{i\xi k_3^2} d \xi, \quad k=-k_1 + k_2 + k_3.
 \end{equs}
If we assume that $ p \geq 3 $, then $k^p$ is clearly the dominant term.  Let us compare the local errors depending on the choice of approximation.  Using a resonance based discretisation we obtain that
 \begin{itemize}
 \item Resonance scheme:
 \begin{equs}
 I = \int_0^{\tau} e^{-  \xi (ik)^p} d \xi +   \mathcal{O}( \xi^2 P_1), \quad P_1 = - k^2_1 + k^2_2 + k^2_3
 \end{equs} \end{itemize}
 whereas a classical discretisation leads to
  \begin{itemize}
 \item Classical integrators:
 \begin{equs}
 I = \xi +   \mathcal{O}( \xi^2 P_2), \quad P_2 = k^p - k^2_1 + k^2_2 + k^2_3.
 \end{equs}
 \end{itemize}
 As soon as $ p \geq 3 $, we ask for at least three derivatives on the initial data for a classical integrator (due to the error term $\mathcal{O}( \xi^2 P_2)$) while only requiring two derivatives for the resonance based scheme (thanks to the improved local error term $ \mathcal{O}( \xi^2 P_1)$).

\subsection{Outline of the paper}

Let us give a short review of the content of this paper.
In Section~\ref{section_decorated}, we introduce the combinatorial structures needed for describing oscillatory integrals coming from the iteration of the Duhamel's formula \eqref{duhamel_fourier}. We recall the framework introduced in \cite{BS} by first defining a suitable vector space of decorated trees $ \hat{\CT} $ and decorated forests $ \hat \CH $. Then, we consider  approximated decorated trees and forest which carry an extra decoration $ r $ at the root. It corresponds to the order of the approximation of the oscillatory integrals. The main novelty in this section are paired decorated forests which are specific decorated forests with extra constraints. They reflect the computation of the second moment of oscillatory integrals: pairing among the leaves of two trees correspond to the Wick formula applied to product of Gaussian random variables.

In Section~\ref{sec::3}, we construct the approximation of the iterated integrals given by the character $ \Pi : \hat \CH \rightarrow  \CC  $ (see \eqref{Pi}) through the character $ \Pi^n : \CH \rightarrow \CC $ (see \eqref{recursive_pi_r}) where $ n $ is the a priori regularity assumed on the initial data $ v $. The approximation $ \Pi^n $ is given via a recursive construction involving the operator  $ \CK $ given in Definition~\ref{CK}. In the definition, we are performing a full Taylor expansion which can be used in the examples mentioned in this paper. It also provides the optimal regularity. In the general case, the resonance approach could be needed.  The local error analysis which is the error estimate on the difference  between $ \Pi  $ and its approximation $ \Pi^n $ is given in  Theorem~\ref{approxima_tree_paired}.  It is computed via a recursive definition (see Definition~\ref{def:Llow}) with Taylor remainders of $ \Pi^n$ given in Lemma~\ref{Taylor_bound}. Let us mention that  Definition~\ref{def:Llow} is more involved as the one given in \cite{BS}, as the pairing among the leaves required to be more precise in the estimation of the local error.

In Section~\ref{sec::4}, we introduce truncated series of decorated trees that solve up to order $ r $ equation~\eqref{dis}.   Then, we compute the second moment of this series in Proposition~\ref{prop_moment_2} by using Wick's formula (see Proposition~\ref{Wick}). This is where we rely on paired decorated forests for describing the new series obtained. From this series, we  built another paired decorated forests series by replacing the character $ \Pi $ by its approximation $ \Pi^n $. Then, we can write  the general scheme \eqref{general_scheme_2}. In the end, we compute its local error structure (see Theorem~\ref{thm:genloc}) based on the local error between $ \Pi $ and $ \Pi^n $ for each paired decorated forest that appears in the expansion of the scheme.

In Section~\ref{sec:app}, we illustrate our general framework on  two fundamental examples: the nonlinear Schrodinger (NLS)  and the Korteweg--de Vries (KdV) equation. 
\subsection*{Acknowledgements}

{\small
This project has received funding from the European Research Council (ERC) under the European Union's Horizon 2020 research and innovation programme (grant agreement No. 850941). Y. B. thanks the Max Planck Institute for Mathematics in the Sciences (MiS) in Leipzig for supporting his research via a long stay in Leipzig from January to June 2022 where this work was written. Y. B. is funded by the ANR via the project LoRDeT (Dynamiques de faible régularité via les arbres décorés) from the projects call T-ERC\_STG. Y. A. B. thanks ICERM, and the Simons Fondation for their support, as well as the Max Planck Institute for Mathematics in the Sciences (MiS) in Leipzig for a short stay when this project was started.
}

\section{Decorated tree based structures for oscillatory iterated integrals}
\label{section_decorated}

In this section, we introduce the combinatorial structure for describing oscillatory integrals that stem from Duhamel's formula \eqref{duhamel_fourier}. We follow the formalism of decorated trees given in \cite{BS} which has been used for a low regularity approximation of \eqref{dis}. In order to describe $ \mathbb{E}(|u_k(t,v^{\eta})|^2) $ we need to introduce a new structure: paired decorated forests which are decorated forests in the sense of $ \cite{BS} $ satisfying extra  constraints on the decorations that encode some pairings. These pairings are coming from the Wick formula used for computing  $ \mathbb{E}(|u_k(t,v^{\eta})|^2) $. Such paired structures have been used in wave turbulence theory in more specific set-ups (see \cite{deng_hani_2021,deng2021derivation,ACG,staffilani2021wave}).

\subsection{Decorated trees}
We recall briefly the structure of decorated trees introduced in \cite[Sec. 2]{BS}.
We assume   a finite set $  \Lab$ and  frequencies $ k_1,...,k_m \in \Z^{d} $. The set $ \Lab$ parametrizes a set of differential
operators with constant coefficients, whose symbols are given by the polynomials  $ (P_{\Labhom})_{\Labhom \in \Lab} $.
We define the set of decorated trees $ \hat \CT  $ 
as elements of the form  $ 
T_{\Labe}^{\Labn, \Labo} =  (T,\Labn,\Labo,\Labe) $ where 
\begin{itemize}
\item $ T $ is a non-planar rooted tree with root $ \varrho_T $, node set $N_T$ and edge set $E_T$. We denote the leaves of $ T $ by $ L_T $. $ T $ must also be a planted tree which means that there is only one edge connecting the root to the rest of the tree.
\item the map $ \Labe : E_T \rightarrow \Lab \times \lbrace 0,1\rbrace$ encodes edge decorations. The set $ \{ 0,1\} $ encodes the action of taking the conjugate, and determines the sign of the frequencies at the top of this edge. Namely, we have that $1$ corresponds to a conjugate and to multiplying by $(-1)$ the frequency on the node above and adjacent to this edge.
\item the map $ \Labn : N_T \setminus \lbrace\varrho_T \rbrace \rightarrow \N $ encodes node decorations. For every inner node $ v$, this map encodes a monomial of the form $ \xi^{\Labn(v)} $ where $ \xi $ is a time variable.
\item the map $ \Labo : N_T \setminus \lbrace\varrho_T \rbrace \rightarrow \Z^{d}$ also encodes node decorations. These decorations are frequencies that satisfy for every inner node $ u $:
\begin{equs} \label{innerdecoration}
(-1)^{\mathfrak{p}(e_u)}\Labo(u)  = \sum_{e=(u,v) \in E_T} (-1)^{\mathfrak{p}(e)} \Labo(v)
\end{equs}
where $ \Labe(e) = (\Labhom(e),\mathfrak{p}(e)) $ and  $ e_u $ is the edge outgoing $ u $ of the form $ (v,u) $ . From this definition, one can see that the node decorations at the leaves $ (\Labo(u))_{u \in L_T} $ determine the decoration of the inner nodes. One can call this identity Kirchoff's law. We assume that the node decorations at the leaves are linear combinations of
the $ k_i $ with coefficients in $ \lbrace -1,0,1 \rbrace $. In applications, the leaves decoration are either $ k_i $ or $ - k_i $, the term $ - k_i $ is coming from a pairing with the Wick formula.
\item we assume that the root of $ T $ has no decoration.
\end{itemize}

When the node decoration $ \Labn $ is zero, we will denote the decorated trees $ T_{\Labe}^{\Labn,\Labo} $ as
 $ T_{\Labe}^{\Labo} = (T,\Labo,\Labe)  $. The set of decorated trees satisfying such a condition is denoted by $ \hat \CT_0 $. We set $\hat H $ (resp. $ \hat H_0 $) the (unordered) forests composed of trees in $ \hat \CT $ (resp. $ \hat \CT_0 $) with  linear spans $\hat \CH $ and  $ \hat \CH_0 $.  The forest product is denoted by $ \cdot $, the empty forest by $ \one $. Elements in $ \hat \CT $ are abstract representation of iterated time integrals and elements in $ \hat{H} $ are product of them.
 
We now introduce how one can represent uniquely decorated trees by using symbolic notations. We denote by $\CI_{o}$, an edge decorated by $o=(\mathfrak{t}, \mathfrak{p}) \in \Lab \times \lbrace 0,1\rbrace$. We introduce the operator $  \CI_{o}(\lambda_{k}^{\ell} \cdot) : \hat \CH \rightarrow \hat \CH $  that merges all the roots of the trees composing the forest into one node decorated by $(\ell,k) \in \N \times \Z^{d} $.  The new decorated tree is then grafted via an edge decorated by $ o $ onto a new root with no decoration. 
 If the condition $ \eqref{innerdecoration} $ is not satisfied on the argument then $\CI_{o}( \lambda_{k}^{\ell} \cdot)$ gives zero.
If $ \ell = 0 $, then the term $  \lambda_{k}^{\ell} $ is denoted by $  \lambda_{k} $ as a short hand notation for $  \lambda_{k}^{0} $.
 We have choosen to put no decorations at the root in order to define the operator $ \CI_{o}(\lambda_{k}^{\ell} \cdot) $. that will assign the decoration $ k $. For us, these decorated trees represent oscillatory iterated integrals and there is no need for having a decoration at the root for encoding them. 

The forest product between $ \CI_{o_1}(  \lambda^{\ell_1}_{k_1}F_1) $ and $ \CI_{o_2}(  \lambda^{\ell_2}_{k_2}F_2) $ is given by:
\begin{equs}
 \CI_{o_1}(  \lambda^{\ell_1}_{k_1} F_1) \CI_{o_2}(  \lambda^{\ell_2}_{k_2} F_2) := \CI_{o_1}( \lambda^{\ell_1}_{k_1} F_1) \cdot \CI_{o_2}( \lambda^{\ell_2}_{k_2} F_2).  
\end{equs}
Any decorated tree $ T $ is uniquely represented as 
\begin{equs}
T = \CI_{o}(  \lambda^{\ell}_{k} F), \quad F \in \hat H.
\end{equs}
Given an iterated integral, its size is given by the number of integrations in time. Therefore, we suppose given a subset  $ \Lab_+ $ of  $ \Lab  $ that encodes edge decorations which correspond to time integrals that we have to approximate.

\begin{example} \label{ex_1} We illustrate the definitions introduced above with decorated trees coming from the NLS equation. We consider the following decorated tree:
\begin{equs}
T =  \CI_{(\mathfrak{t}_1,0)}\Big(\lambda_{k}   \CI_{(\mathfrak{t}_2,0)}\left(\lambda_{k}    \CI_{(\mathfrak{t}_1,1)}(\lambda_{k_1}   ) \CI_{(\mathfrak{t}_1,0)}(\lambda_{k_2}   ) \CI_{(\mathfrak{t}_1,0)}(\lambda_{k_3}   )\right) \Big) =  \begin{tikzpicture}[scale=0.2,baseline=-5]
\coordinate (root) at (0,-1);
\coordinate (t3) at (0,1);
\coordinate (t4) at (0,3);
\coordinate (t41) at (-2,5);
\coordinate (t42) at (2,5);
\coordinate (t43) at (0,7);
\draw[kernels2] (t3) -- (root);
\draw[symbols] (t3) -- (t4);
\draw[kernels2,tinydots] (t4) -- (t41);
\draw[kernels2] (t4) -- (t42);
\draw[kernels2] (t4) -- (t43);
\node[not] (rootnode) at (root) {};
\node[not] (rootnode) at (t4) {};
\node[not] (rootnode) at (t3) {};
\node[var] (rootnode) at (t41) {\tiny{$ k_1 $}};
\node[var] (rootnode) at (t42) {\tiny{$ k_3 $}};
\node[var] (rootnode) at (t43) {\tiny{$  k_{2} $}};
\end{tikzpicture},
\end{equs}
where $ k = - k_1 + k_2 + k_3 $, $ \Lab = \lbrace \mathfrak{t}_1, \mathfrak{t}_2 \rbrace $, $ \Lab_+ = \lbrace \mathfrak{t}_2 \rbrace $,  $P_{\mathfrak{t}_1}(\lambda) = -\lambda^2$ and $P_{\mathfrak{t}_2}(\lambda) = \lambda^2$. For the graphical notation, we put the frequencies decorations only on the leaves that determine those on the inner nodes. An edge  $ \<thick> $ (resp. $ \<thick2> $) corresponds to a decoration $ (\mathfrak{t}_1,0) $ (resp. $(\mathfrak{t}_1,1) $) and the operator associated is $e^{i t P_{\mathfrak{t_1}}(k)} =  e^{-i t k^2}   $ (resp. $ e^{-i t P_{\mathfrak{t_1}}(k)} = e^{i t k^2} $), while an edge  $\<thin>$ (resp. $\<thin2>$) corresponds to a decoration $ (\mathfrak{t}_2,0) $ (resp. $ (\mathfrak{t}_2,1) $) associated to the integral $ -i\int^{\tau}_0 e^{i \xi P_{\mathfrak{t}_2}(k)} \cdots d \xi = -i\int^{\tau}_0 e^{i \xi k^2} \cdots d \xi   $ (resp. $ -i \int^{\tau}_0 e^{-i \xi P_{\mathfrak{t}_2}(k)} \cdots d \xi=  -i \int^{\tau}_0 e^{-i \xi k^2} \cdots d \xi $). Therefore, $ T $ is an abstract version of the following integral:
\begin{equs}
- i e^{-i\tau k^{2}} \int^{\tau}_0 e^{i \xi k^2} e^{i \xi k_1^2}  e^{-i \xi k_2^2}  e^{-i \xi k_3^2}d \xi.
\end{equs}
\end{example}

The next combinatorial structure, we recall from $ \cite{BS} $  is abstract versions of  a discretization of an oscillatory integral. We denote by $ \CT $ the set of decorated trees $ T_{\Labe,r}^{\Labn,\Labo} = (T,\Labn,\Labo,\Labe,r)  $ where
\begin{itemize}
\item $ T_{\Labe}^{\Labn,\Labo} \in \hat \CT $
\item The decoration of the root is given by $ r \in \Z $, $ r \geq -1 $ such that
\begin{equs} \label{condition_trees}
 r +1 \geq  \deg(T_{\Labe}^{\Labn,\Labo})
\end{equs}
 where $ \deg $ is defined recursively by 
 \begin{equs}
\deg(\one) & = 0, \quad \deg(F_1 \cdot F_2 )  = \max(\deg(F_1),\deg(F_2)),  \\
\deg(\CI_{(\Labhom,\Labp)}(  \lambda^{\ell}_{k}F_1)   ) & = \ell +   \one_{\lbrace\Labhom \in \Lab_+\rbrace}   +  \deg(F_1)
\end{equs}
where $ F_1, F_2  $ are forests composed of trees in $ \CT $. The quantity $\deg(T_{\Labe}^{\Labn,\Labo})$ is the maximum number of edges with type in $ \Lab_+ $, corresponding to time integrations, and of node decorations $ \Labn $ lying on the same path from one leaf to the root. 
\end{itemize}
We call decorated trees in $ \CT $ {\it approximated} decorated trees. We allow in the definition of the approximated decorated trees a new decoration at the root $ r $, which corresponds to the order of the approximation.
We now define the symbol $  \CI^{r}_{o}(\lambda_{k}^{\ell} \cdot) :  \CH \rightarrow  \CH $, which plays the same role as the previously defined symbol $ \CI^{r}_{o}(\lambda_{k}^{\ell} \cdot) $, with the added adjunction of the decoration $ r $ which constrains the time-approximations to be of order $r$. We now define a projection operator $\CD^{r}$ which depends on the order $r$ of the approximation and which is used during the construction of the numerical schemes in order to only retain the terms of order at most $r$.
We define the map $ \CD^{r} : \hat \CH \rightarrow \CH $ which assigns to the root of  a decorated tree  a decoration $ r $ and performs the projection along the identity \eqref{condition_trees}. It is given by
\begin{equs}\label{DR}
\CD^{r}(\one)= \one_{\lbrace 0 \leq  r+1\rbrace} , \quad \CD^r\left( \CI_{o}( \lambda_{k}^{\ell} F) \right) =  \CI^{r}_{o}( \lambda_{k}^{\ell} F)
\end{equs}
and we extend it multiplicatively to any forest in $ \hat \CH $. 

\begin{example} We illustrate the action of the map $ \CD^r $ on the decorated tree $ T $ introduced in Example~\ref{ex_1}. First, we notice that:
\begin{equs}
\deg(T) = 1, \quad \CD^r(T)  = 0, \, r > 1.
\end{equs}
Graphically, we just add $ r $ at the root:
\begin{equs}
 \CD^{r}\left( T \right) = \begin{tikzpicture}[scale=0.2,baseline=-5]
\coordinate (root) at (0,-1);
\coordinate (t3) at (0,1);
\coordinate (t4) at (0,3);
\coordinate (t41) at (-2,5);
\coordinate (t42) at (2,5);
\coordinate (t43) at (0,7);
\draw[kernels2] (t3) -- (root);
\draw[symbols] (t3) -- (t4);
\draw[kernels2,tinydots] (t4) -- (t41);
\draw[kernels2] (t4) -- (t42);
\draw[kernels2] (t4) -- (t43);
\node[not,label= {[label distance=-0.2em]below: \scriptsize  $ r $}] (rootnode) at (root) {};
\node[not] (rootnode) at (t4) {};
\node[not] (rootnode) at (t3) {};
\node[var] (rootnode) at (t41) {\tiny{$ k_1 $}};
\node[var] (rootnode) at (t42) {\tiny{$ k_3 $}};
\node[var] (rootnode) at (t43) {\tiny{$  k_{2} $}};
\end{tikzpicture}.
\end{equs}
\end{example}

\subsection{Paired decorated forests}
Our aim is to discretise the second moment of $ U_{k}^{r}(\tau, v^{\eta}) $ which can be expanded as a tree series. By applying the Wick product, one gets products of iterated  integrals whose initial data $ v^\eta $ are paired. At the level of the trees, this correspond to a pairing on the leaves. Therefore, one has to introduce a more rigid structure for encoding these new terms.
We define $ \hat{\CG} $ as the set of paired decorated forests $F$ such that:
\begin{itemize}
\item $F$ is a forest containing solely two trees denoted by $ T $ and $ \bar{T} $.
\item One has the following assumptions on the decoration $ \Labo $:
\begin{equs}\label{pairRoot}
(-1)^{\mathfrak{p}(e_{\rho_T}) }\Labo(u_T) = (-1)^{\mathfrak{p}(e_{\rho_{\bar{T}}})  } \Labo(u_{\bar{T}})
\end{equs}
where $ u_T $ (resp. $ u_{\bar{T}} $) are the nodes in T (resp. $ \bar{T} $) connected to $ \rho_{T} $ (resp. $ \rho_{\bar T} $). The decoration $\mathfrak{p}(e_{\rho_T})$ corresponds to the second decoration component on the edge $ e_{\rho_T} $ connecting the root $ \rho_T $ to $ u_T $.
\item
 For any leaf $ u \in L_F $, there exists exactly one leaf $  v \in L_F$ such that
\begin{equs}\label{pairFreq} \begin{aligned}
(-1)^{\mathfrak{p}(e_u)} \Labo(u) & = (-1)^{\mathfrak{p}(e_v)+1} \Labo(v), \quad \text{if } \lbrace u,v \rbrace \subset L_T \text{ or } \lbrace u,v \rbrace \subset L_{\bar T}
\\ (-1)^{\mathfrak{p}(e_u)} \Labo(u) & = (-1)^{\mathfrak{p}(e_v)} \Labo(v), \quad \text{otherwise}
\end{aligned}
\end{equs}
where $ e_u $ (resp. $ e_v $) is the edge connecting $ u $ (resp. $ v $) to the rest of the tree, and $\mathfrak{p}(e)\in \{0,1\}$ is the second component of the edge decoration associated to $e$. 
\end{itemize}
Approximated paired decorated forests $ F $ are described by 
\begin{equs}
F = \CD^r(\bar F), \quad r \in \Z, \, \bar{F} \in \hat{\CG}.
\end{equs}
We denote this set as $ \CG $ and by $ \CG_k^r $ decorated paired forests having at most $ r $ edges with type in $ \Lab_+ $ and both sides of the identity \eqref{pairRoot} are equal either to $ k $ or $ -k $. 

\begin{example} We provide an example of paired decorated forests:
\begin{equs}
 F = T_1 \cdot T_2, \quad T_1 = \begin{tikzpicture}[scale=0.2,baseline=-5]
\coordinate (root) at (0,1);
\coordinate (tri) at (0,-1);
\draw[kernels2] (tri) -- (root);
\node[var] (rootnode) at (root) {\tiny{$ k $}};
\node[not] (trinode) at (tri) {};
\end{tikzpicture} ,   \quad T_2 =  \begin{tikzpicture}[scale=0.2,baseline=-5]
\coordinate (root) at (0,-1);
\coordinate (t3) at (0,1);
\coordinate (t4) at (0,3);
\coordinate (t41) at (-2,5);
\coordinate (t42) at (2,5);
\coordinate (t43) at (0,7);
\draw[kernels2] (t3) -- (root);
\draw[symbols] (t3) -- (t4);
\draw[kernels2,tinydots] (t4) -- (t41);
\draw[kernels2] (t4) -- (t42);
\draw[kernels2] (t4) -- (t43);
\node[not] (rootnode) at (root) {};
\node[not] (rootnode) at (t4) {};
\node[not] (rootnode) at (t3) {};
\node[var] (rootnode) at (t41) {\tiny{$ k_1 $}};
\node[var] (rootnode) at (t42) {\tiny{$ k $}};
\node[var] (rootnode) at (t43) {\tiny{$  k_{\tiny{1}} $}};
\end{tikzpicture}   
\end{equs}
Let $u$ denote the single leaf in the tree $T_1$, and let $\tilde{u}, \tilde{v}$, $v$ denote the leaves, from left to right, in the tree $T_2$. We have that $\Labo(e_u) = \Labo(e_{v}) = k$, and $\Labo(e_{\tilde{u}}) = k_1 = \Labo(e_{\tilde{v}})$. Moreover, $\mathfrak{p}(e_{u})= \mathfrak{p}(e_{\tilde{v}})=\mathfrak{p}(e_{v}) = 0$ and $\mathfrak{p}(e_{\tilde{u}}) = 1$ since the edge $e_{\tilde{u}}$ is a dotted brown edge, which encodes that the frequency is preceded by a negative sign.
Hence, from \eqref{pairFreq} we read:
 \begin{equs} 
(-1)^1 k_1 = (-1)^1 k_1 ,  \quad (-1)^0 k = (-1)^0k.
 \end{equs} 
 For the first condition~\eqref{pairRoot}, one has $ \mathfrak{p}(e_{\rho_{T_1}}) =  0  $ and $ \mathfrak{p}(e_{\rho_{T_2}}) =  0 $.  Therefore, one similarly obtains
 \begin{equs}
 (-1)^{0 } k = (-1)^{0} k.
 \end{equs}
\end{example}

%
\begin{remark} In the above definition, paired decorated forests are formed of only two trees with pairings among the leaves. The pairings among the leaves reflect the computation of the second moment of products of Gaussian random variables, dictated by Wick's theorem. One can easily generalise this definition depending on one needs. Indeed, first instead of computing second order moments, one can 
look at higher order moments, namely moments of order $ 2p $, and consider forests consisting of $ 2p $ trees. The pairing on the leaves between the $2p$ trees would follow the same condition~\eqref{pairRoot}. Secondly, one can consider non-Gaussian random initial conditions together with the use of cumulants to compute second moments, see Remark \ref{Gaus-ic}. In this case, the above definition can be replaced by pairings between bigger clusters of leaves and one has to change \eqref{pairFreq} accordingly, by including more than two leaves. 
\end{remark}

\section{Discretizing oscillatory iterated integrals}

\label{sec::3}

In this section, we recall the definition of the map $ \Pi : \hat{\CH} \rightarrow \CC $  that interprets decorated forests as oscillatory integrals and its low regularity discretization $ \Pi^n $ introduced both in \cite{BS}. The main simplification in comparison to the resonance method exposed is the definition of  $ \CK_{o_2}^{k,r} $ which is now performing the full Taylor expansion.
This is due to the convolution structures observed in the computation of $ \mathbb{E}(|u_k(\tau,v^{\eta})|^2) $ (see remark~\ref{rem_reg}). This will be enough for the examples treated in this paper but in general a resonance analysis could be needed. The local error analysis relies on the error introduced by the operator $ \CK_{o_2}^{k,r} $ via Lemma~\ref{Taylor_bound}. This error is propagated through a recursive definition of the error  namely $ \mathcal{L}^{r}_{\text{\tiny{low}}}(F,n) $. This step is more involved in comparison to $ \cite{BS} $. Indeed, the error introduced by the approximation of a product of iterated integrals needs a more careful treatment, see the new formulation in Definition~\ref{def:Llow}.

We adopt in this section the following notations: an element of $ \Lab_+ $ (resp, $ \Lab_+ \times \lbrace 0,1\rbrace $) is denoted by $ \Labhom_2 $ (resp. $ o_2 $) and an element of $ \Lab \setminus \Lab_+ $ (resp.  $ \Lab \setminus \Lab_+ \times \lbrace  0,1 \rbrace $) is denoted by $ \Labhom_1 $ (resp. $ o_1 $). 
We define $ \CC $ as  the space of functions of the form $ z \mapsto \sum_j Q_j(z)e^{i z P_j(k_1,...,k_m) } $ where the $ Q_j(z) $ are polynomials in $ z $ and the $ P_j $ are polynomials in $ k_1,...,k_n \in \Z^{d} $. The $ Q_j $ may also depend on  $ k_1,...,k_m $. Equipped with the pointwise product $ \CC $ is an algebra.
 Iterated integrals and their discretisation will be characters from decorated forests into $ \CC $.
For a character $ g : \hat \CH \rightarrow \CC  $, one has:
\begin{equs}
g(F \cdot \bar F) = g(F) g(\bar F), \quad F, \bar{F} \in \hat \CH.
\end{equs}
We define the character $ \Pi : \hat \CH \rightarrow \mathcal{C} $  by
\begin{equation}\label{Pi}
\begin{aligned}
\Pi \left( F \cdot \bar F \right)(\tau) & =  ( \Pi F)(\tau) ( \Pi \bar F )(\tau),
\\
\Pi \left( \CI_{o_1}( \lambda_k^{\ell} F)\right)(\tau) & = 
 e^{i \tau P_{o_1}(k)} \tau^{\ell} (\Pi F)(\tau),
\\
 \Pi \left(  \CI_{o_2}( \lambda_k^{\ell} F)\right)(\tau) & = -i  \vert \nabla\vert^{\alpha} (k)
\int_{0}^{\tau} e^{i \xi P_{o_2}(k)} \xi^{\ell}(\Pi F)(\xi) d \xi, 
\end{aligned}
\end{equation}
where $ F, \bar F \in \hat \CH $. We have used the short hand notation $ P_{o_1} $ given for $ o_1 = (\mathfrak{t}_1,\mathfrak{p}_1) $ by:
\begin{equs}
P_{o_1}(k) = (-1)^{\mathfrak{p}_1} P_{\mathfrak{t}_1}((-1)^{\mathfrak{p}_1}  k).
\end{equs}
For the discretisation, we consider a new family of characters defined now on $ \CH $ and parametrized by $ n \in \N $:
\begin{equation} \label{recursive_pi_r}
\begin{aligned}
\Pi^n \left( F \cdot \bar F \right)(\tau) & = 
\left( \Pi^n F  \right)(\tau)  \left( \Pi^n  \bar F \right)(\tau), \quad (\Pi^n  \lambda^{\ell})(\tau) = \tau^{\ell}, \\
\left(\Pi^n  \CI^{r}_{o_1}( \lambda_{k}^{\ell}  F )\right)(\tau)   & =\tau^{\ell} e^{i \tau P_{o_1}(k)} \left(\Pi^n \CD^{r-\ell}(F)\right)(\tau),  \\
\left( \Pi^n \CI^{r}_{o_2}( \lambda^{\ell}_k F) \right) (\tau) & = \CK^{k,r}_{o_2} \left(  \Pi^n \left( \lambda^{\ell} \CD^{r-\ell-1}(F) \right),n \right)(\tau).
\end{aligned}
\end{equation}
In the sequel, we will use the short hand notations:
\begin{equs}
\Pi^{n,r} = \Pi^n \CD^r, \quad (\bar{\Pi} T)(\tau) = \overline{(\Pi T)}(\tau), \quad (\bar{\Pi}^{n,r} T)(\tau) = \overline{(\Pi^{n,r} T)}(\tau).
\end{equs}
We provide the definition of the approximation operator $\CK  $ below, when one performs the full Taylor expansion of the oscillatory integrals.

\begin{definition}\label{CK}
Assume that $ G:  \xi \mapsto \xi^{q} e^{i \xi P(k_1,...,k_m)} $ where $ P $ is a polynomial in the frequencies $ k_1,...,k_m $ and let $ o_2 =  (\Labhom_2,p) \in \Lab_+ \times \lbrace 0,1 \rbrace$ and $ r \in \N $. Let $ k $ be a linear map in $ k_1,...,k_m $ using coefficients in $ \lbrace -1,0,1 \rbrace $. We set
\begin{equs} \label{exp_CK}
\CK_{o_2}^{k,r}(G,n)(\tau) =-i|\nabla|^{\alpha}(k)\sum_{\ell \leq r-q}  \frac{\tau^{\ell+q+1}}{\ell!(\ell+q+1)} \left(i  P_{o_2}(k) + i P(k_1,...,k_m) \right)^{\ell}.\qquad
\end{equs}
\end{definition}

\begin{remark} \label{rem_reg} The definition \eqref{exp_CK} for $ \CK $ is quite simple and encodes the approximation of the oscillatory integral 
$$
-i|\nabla|^{\alpha}(k)\int_0^\tau e^{i\xi P_{o_{2}}(k)}G(\xi)d\xi
$$
by a full Taylor expansion of the operator appearing in the integrand.
This approximation does not require a careful resonance based analysis as was the case in \cite[Def. 3.1]{BS}. 
This is due to the fact that we are not approximating $u_k(\tau)$ but its second moment $\mathbb{E}(|u_k(\tau,v^{\eta})|^2)$. Therefore, our local error terms involve products of two Fourier coefficients. This allows us to repartition half of the regularity on each of the Fourier coefficients, thereby obtaining a similar low-regularity local error structure as in \cite{BS} without requiring a more delicate analysis.
Indeed, to each of the frequencies $ k_1,...,k_m $ we now consider terms of the form  $ v_{k_i}^2 $, $ |v_{k_i}| $ or $ \bar{v}_{k_i} v_{k_i} $, instead of $ v_{k_i} $ or $ \bar{v}_{k_i} $. The new  quadratic terms can be interpreted as convolutions or are better summable. In the case of the convolution, the two following terms ask for the same regularity:
\begin{equs}
k^2 v_{k}^2 \equiv \partial_x v * \partial_x v, \quad k v_k \equiv \partial_x v
\end{equs}
where $ * $ is the spatial convolution. One can see the difference between the discretization presented here and the one in \cite{BS} by considering the following oscillatory integral coming from NLS:
\begin{equs}
I = \int_0^{\tau} e^{-i\xi k^2} e^{-i\xi k_1^2} e^{i\xi k_2^2} e^{i\xi k_3^2} d \xi, \quad k=-k_1 + k_2 + k_3.
\end{equs}
The idea of the resonance analysis is to split the integrand into the dominant and lower order parts (cf.\eqref{resDecomp})
\begin{equs}
 \mathcal{L} & = k^2 + k_1^2 - k_2^2 -k_3^2 = \mathcal{L}_{\tiny{\text{dom}}} + \mathcal{L}_{\tiny{\text{low}}} \\\nonumber
 \mathcal{L}_{\tiny{\text{dom}}} & = 2 k_1^2, \quad \mathcal{L}_{\tiny{\text{low}}} = - k_1 (k_2 + k_3) + k_2 k_3
\end{equs}  
 and to Taylor expand only the term of lower degree $e^{i \xi \mathcal{L}_{\tiny{\text{low}}} }$. This yields the following discretisation
\begin{equs}
I  = \int_0^{\tau} e^{- 2i \xi k_1^2} \left( 1  + \mathcal{O}( \xi P_1) \right) d \xi, \quad P_1  = k_1 (k_2 +k_3) - k_2 k_3 
\end{equs}
with
\begin{equs}
I  = \int_0^{\tau} e^{- 2i \xi k_1^2} d \xi +   \mathcal{O}( \tau^2 P_1) .
\end{equs}
In contrast, a full classical Taylor expansion gives
\begin{equs}
I  = \int_0^{\tau} e^{- 2i \xi k_1^2} d \xi +   \mathcal{O}( \tau^2 P_2), \quad P_2 = - 2 k^2_1 + k_1(k_2 + k_3) - k_2 k_3.
\end{equs}
The difference between the two approximations is that the second one asks for two derivatives instead of one. In our setting, the latter is, however, absorbed by the convolution of initial data.
\end{remark}

\begin{remark}\label{rem:sysK} For many equations, the main difference between resonance based and classical integrators lies in the  gain of one derivative (at order one). Therefore, there is no need to carry out a  resonance based discretization for $ \mathbb{E}(|u_k(\tau,v^{\eta})|^2) $ due to the improved resonance structure of the latter. Resonance based schemes are, however,  needed if the gap between dominant and lower parts of the resonance is bigger. Then, the full definition \cite[Def. 3.1]{BS} is required. This can happen when one looks at systems of dispersive PDEs as for instance system \eqref{systK}.

\end{remark}
\begin{example}
With the aid of \eqref{Pi}, one can compute recursively the following oscillatory integrals arising in the cubic NLS equation~\eqref{nlsIntro}
\begin{equs}
  (\Pi  \begin{tikzpicture}[scale=0.2,baseline=-5]
\coordinate (root) at (0,1);
\coordinate (tri) at (0,-1);
\draw[kernels2] (tri) -- (root);
\node[var] (rootnode) at (root) {\tiny{$ k_2 $}};
\node[not] (trinode) at (tri) {};
\end{tikzpicture}) (\tau) &  = (\Pi \CI_{(\mathfrak{t}_1,0)}(\lambda_{k_2}))(\tau) =  e^{-i \tau k_2^2}, \\ (\Pi  \begin{tikzpicture}[scale=0.2,baseline=-5]
\coordinate (root) at (0,1);
\coordinate (tri) at (0,-1);
\draw[kernels2,tinydots] (tri) -- (root);
\node[var] (rootnode) at (root) {\tiny{$ k_1 $}};
\node[not] (trinode) at (tri) {};
\end{tikzpicture}) (\tau) & =  (\Pi \CI_{(\mathfrak{t}_1,1)}(\lambda_{k_1}))(\tau) = e^{i \tau k_1^2}, \\
(\Pi \begin{tikzpicture}[scale=0.2,baseline=-5]
\coordinate (root) at (0,-1);
\coordinate (t1) at (-2,1);
\coordinate (t2) at (2,1);
\coordinate (t3) at (0,2);
\draw[kernels2,tinydots] (t1) -- (root);
\draw[kernels2] (t2) -- (root);
\draw[kernels2] (t3) -- (root);
\node[not] (rootnode) at (root) {};t
\node[var] (rootnode) at (t1) {\tiny{$ k_{\tiny{1}} $}};
\node[var] (rootnode) at (t3) {\tiny{$ k_{\tiny{2}} $}};
\node[var] (trinode) at (t2) {\tiny{$ k_3 $}};
\end{tikzpicture}  )(\tau) & = (\Pi  \begin{tikzpicture}[scale=0.2,baseline=-5]
\coordinate (root) at (0,1);
\coordinate (tri) at (0,-1);
\draw[kernels2,tinydots] (tri) -- (root);
\node[var] (rootnode) at (root) {\tiny{$ k_1 $}};
\node[not] (trinode) at (tri) {};
\end{tikzpicture}) (\tau) \, (\Pi  \begin{tikzpicture}[scale=0.2,baseline=-5]
\coordinate (root) at (0,1);
\coordinate (tri) at (0,-1);
\draw[kernels2] (tri) -- (root);
\node[var] (rootnode) at (root) {\tiny{$ k_2 $}};
\node[not] (trinode) at (tri) {};
\end{tikzpicture}) (\tau) \, (\Pi  \begin{tikzpicture}[scale=0.2,baseline=-5]
\coordinate (root) at (0,1);
\coordinate (tri) at (0,-1);
\draw[kernels2] (tri) -- (root);
\node[var] (rootnode) at (root) {\tiny{$ k_3 $}};
\node[not] (trinode) at (tri) {};
\end{tikzpicture}) (\tau)=  e^{i \tau (k_1^2 - k_2^2 - k_3^2)},
\\ ( \Pi \begin{tikzpicture}[scale=0.2,baseline=-5]
\coordinate (root) at (0,0);
\coordinate (tri) at (0,-2);
\coordinate (t1) at (-2,2);
\coordinate (t2) at (2,2);
\coordinate (t3) at (0,3);
\draw[kernels2,tinydots] (t1) -- (root);
\draw[kernels2] (t2) -- (root);
\draw[kernels2] (t3) -- (root);
\draw[symbols] (root) -- (tri);
\node[not] (rootnode) at (root) {};t
\node[not] (trinode) at (tri) {};
\node[var] (rootnode) at (t1) {\tiny{$ k_{\tiny{1}} $}};
\node[var] (rootnode) at (t3) {\tiny{$ k_{\tiny{2}} $}};
\node[var] (trinode) at (t2) {\tiny{$ k_3 $}};
\end{tikzpicture}) (\tau) & = -i \int^{\tau}_0 e^{is (-k_1 + k_2 + k_3)^2} e^{i s (k_1^2 - k_2^2 - k_3^2)} ds.
  \end{equs}
  For a second-order order approximation of the above integral, one has the following discretisation which in the end asks for two derivatives on the initial data,
  \begin{equs}
  ( \Pi^{2,1} \begin{tikzpicture}[scale=0.2,baseline=-5]
\coordinate (root) at (0,0);
\coordinate (tri) at (0,-2);
\coordinate (t1) at (-2,2);
\coordinate (t2) at (2,2);
\coordinate (t3) at (0,3);
\draw[kernels2,tinydots] (t1) -- (root);
\draw[kernels2] (t2) -- (root);
\draw[kernels2] (t3) -- (root);
\draw[symbols] (root) -- (tri);
\node[not] (rootnode) at (root) {};t
\node[not] (trinode) at (tri) {};
\node[var] (rootnode) at (t1) {\tiny{$ k_{\tiny{1}} $}};
\node[var] (rootnode) at (t3) {\tiny{$ k_{\tiny{2}} $}};
\node[var] (trinode) at (t2) {\tiny{$ k_3 $}};
\end{tikzpicture}) (\tau)  =- i \tau + 2\tau^2 \left(  k_1^2 - k_1(k_2+k_3) + k_2 k_3 \right).
  \end{equs}
\end{example}

We recall \cite[Lemma 3.3]{BS} for a function $ G $ as given in Definition~\ref{CK}.

\begin{lemma} \label{Taylor_bound}  We keep the notations of Definition~\ref{CK}. We suppose that $   q \leq r$ then one has
\begin{equs}
- i \vert \nabla\vert^{\alpha} (k) \int_{0}^{\tau} \xi^{q}  e^{i \xi \left(  P_{o_2}(k) + P(k_1,...,k_m) \right)}  d\xi -\CK^{k,r}_{o_2} (  { G},n)(\tau) = \CO(\tau^{r+2} R_{o_2,n}^{k,r}(G))
\end{equs} 
where $ R_{o_2,n}^{k,r}(G)$ depends on $ n $, $ \alpha$ and the frequencies $k_1,...,k_m  $. When $\CK^{k,r}_{o_2} $ performs a full Taylor expansions, it is given by
\begin{equs}\label{R_fullTaylor}
(P_{o_2}(k) + P(k_1,...,k_m))^{r+1-q} k^{\alpha}.
\end{equs}
Otherwise, one considers the resonance analysis and decomposes $ P_{o_2} + P $ into 
\begin{equs}
P_{o_2} + P = \CL_{\text{\tiny{dom}}} + \CL_{\text{\tiny{low}}}
\end{equs}
 and one has
 \begin{equs}
  R_{o_2,n}^{k,r}(G) = k^{\bar n}, \quad \bar{n} = \max(n, \alpha + \deg( \CL_{\text{\tiny{low}}}^{r-q+1})).
 \end{equs}
\end{lemma}
We define inductively the term $ (\Pi \cdot)_0  $ on decorated forests:
\begin{equs}
\Pi \left( F \cdot \bar F \right)_0  & =  ( \Pi F)_0 ( \Pi \bar F )_0,
\quad
\Pi \left( \CI_{o_1}( \lambda_k^{\ell} F)\right)_0  =  (\Pi F)_0,
\\
 \Pi \left(  \CI_{o_2}( \lambda_k^{\ell} F)\right)_0 & =  \vert \nabla\vert^{\alpha} (k)
(\Pi F)_0 .
\end{equs}
For every decorated forest, there exists $ c_{F} $ independent of the frequencies $ k_1,...,k_m $ such that:
\begin{equs}
|(\Pi F)(\tau)| \leq  c_{F} \tau^{\hat{n}_+(F)} (\Pi F)_0
\end{equs}
where $ \hat{n}_+(F) = n_+(F) + \sum_{u \in N_F} \Labn(u) $,  $ n_+(F) $ gives the number of edges with decoration $ o_2 = (\mathfrak{t}_2,\mathfrak{p}_2) $ with $ \mathfrak{t}_2 \in \Lab_+ $. In the sequel, we will use the following notation: $\Pi^{n,r} = \Pi^n  \CD^r$. Given a decorated forest, we set $ (\Pi^{n,r} F)_{\ell} $ to be the term of order $ \tau^{\ell}  $ in $ (\Pi^{n,r} F)(\tau) $ in the sense that one has:
\begin{equs}
\mathcal{Q}_{\leq r +1}(\Pi^{n,r} F)(\tau) = \sum_{ \hat{n}_+(F)\leq \ell \leq r +1} (\Pi^{n,r} F)_\ell(\tau) \tau^{\ell}
\end{equs}
where each $(\Pi^{n,r}F)_{\ell}(\tau)$ is bounded, by denoting the bound by $(\Pi^{n,r}F)_{\ell}$, we have $\forall \tau, \ (\Pi^{n,r}F)_{\ell}(\tau)\le (\Pi^{r,n}F)_{\ell}$.

\begin{definition}\label{def:Llow}
Let $ n \in \N $, $ r \in \Z $. We recursively define $ \mathcal{L}^{r}_{\text{\tiny{low}}}(\cdot,n)$ as
\begin{equs}
\mathcal{L}^{r}_{\text{\tiny{low}}}(F,n) = 1, \quad r < 0.
\end{equs}
Else
\begin{equs}
\mathcal{L}^{r}_{\text{\tiny{low}}}(\one,n) & = 1, \quad \mathcal{L}^{r}_{\text{\tiny{low}}}(\CI_{o_1}( \lambda_{k}^{\ell}  F ),n)  = \mathcal{L}^{r-\ell}_{\text{\tiny{low}}}(  F,n ) 
\\
\mathcal{L}^{r}_{\text{\tiny{low}}}(F_1 \cdot  F_2,n) &  = 
 (\Pi F_2)_0 \,  \mathcal{L}^{r -\hat{n}_+(F_2)}_{\text{\tiny{low}}}(F_1,n ) \\ & + \sum_{\ell \leq r + 1-\hat{n}_+(F_2)} (\Pi^{n,r-\hat{n}_+(F_2)} F_1)_{\ell} \mathcal{L}^{r-\ell}_{\text{\tiny{low}}}( F_2,n) \\
\mathcal{L}^{r}_{\text{\tiny{low}}}(\CI_{o_2}( \lambda^{\ell}_{k}  F ),n) &  =   k^{\alpha} \mathcal{L}^{r-\ell-1}_{\text{\tiny{low}}}(  F,n ) +  R_{o_2,n}^{k,r}(\xi^{\ell}(\Pi^{n,r-\ell-1} F) (\xi)).
\end{equs}
\end{definition}

\begin{example}
		We illustrate  Definition~\ref{def:Llow} on the following paired trees stemming from the NLS equation:
		\begin{equs}\label{NLS_T2_bis}
			F_2 =  T_1 \cdot T_2, 
				\quad T_1 = \CI_{o_1}(\lambda_k) = \begin{tikzpicture}[scale=0.2,baseline=-5]
					\coordinate (root) at (0,1);
					\coordinate (tri) at (0,-1);
					\draw[kernels2] (tri) -- (root);
					\node[var] (rootnode) at (root) {\tiny{$ k $}};
					\node[not] (trinode) at (tri) {};
				\end{tikzpicture}, \quad T_2 = \begin{tikzpicture}[scale=0.2,baseline=-5]
					\coordinate (root) at (0,-1);
					\coordinate (t3) at (0,1);
					\coordinate (t4) at (0,3);
					\coordinate (t41) at (-2,5);
					\coordinate (t42) at (2,5);
					\coordinate (t43) at (0,7);
					\draw[kernels2] (t3) -- (root);
					\draw[symbols] (t3) -- (t4);
					\draw[kernels2,tinydots] (t4) -- (t41);
					\draw[kernels2] (t4) -- (t42);
					\draw[kernels2] (t4) -- (t43);
					\node[not] (rootnode) at (root) {};
					\node[not] (rootnode) at (t4) {};
					\node[not] (rootnode) at (t3) {};
					\node[var] (rootnode) at (t41) {\tiny{$ \bar{k} $}};
					\node[var] (rootnode) at (t42) {\tiny{$ k_{\tiny{2}} $}};
					\node[var] (rootnode) at (t43) {\tiny{$  \bar{k}_{\tiny{2}} $}};
				\end{tikzpicture}, 
			\end{equs}
			where $ \bar{k} = -k $. We have,
			\begin{equs}\label{L_low2}
				\CL_{\text{\tiny{low}}}^0(T_1\cdot {T_2},1) = (\Pi T_2)_{0} \CL_{\text{\tiny{low}}}^{-1}(T_1,1) &+ (\Pi^{1,-1}T_1)_{0} \CL_{\text{\tiny{low}}}^0({T_2},1),
			\end{equs}
			where $\Pi^{1,-1} = \Pi^{1}\CD^{-1} = \Pi^1$, $(\Pi^{1,-1}T_1)_{0} = e^{-i\tau k^2}$, $ \CL_{\text{\tiny{low}}}^{-1}(T_1,1) = 1$  and $(\Pi T_2)_{0} = 1$. Next, we have
			$$
			\CL_{\text{\tiny{low}}}^0({T_2},1) = \CL_{\text{\tiny{low}}}^0(G_2,1) + R^{k,0}_{o_2,1}(\Pi^{1,-1}G_2),
			$$
			with $T_2 = \CI_{o_1}\left(\lambda_k\CI_{o_2}(\lambda_k G_2)\right)$, and $G_2 = \CI_{\bar{o}_1}(\lambda_{-k})\CI_{o_1}(\lambda_{-k_2}) \CI_{o_1}(\lambda_{k_2})$. A quick computation allow to check that $ \CL_{\text{\tiny{low}}}^0(G_2,1) $ is bounded by 
			\begin{equs}
				\mathcal{L}^{0}_{\text{\tiny{low}}}(\CI_{\bar{o}_1}(\lambda_{-k}),1) + \mathcal{L}^{0}_{\text{\tiny{low}}}(\CI_{o_1}(\lambda_{-k_2}),1) + \mathcal{L}^{0}_{\text{\tiny{low}}}(\CI_{o_1}(\lambda_{k_2}),1)  \end{equs}
			and one has:
			\begin{equs}
					\mathcal{L}^{0}_{\text{\tiny{low}}}(\CI_{\bar{o}_1}(\lambda_{-k}),1) =  \mathcal{L}^{0}_{\text{\tiny{low}}}(\CI_{o_1}(\lambda_{-k_2}),1)  = \mathcal{L}^{0}_{\text{\tiny{low}}}(\CI_{o_1}(\lambda_{k_2}),1) = 
					\mathcal{L}^{0}_{\text{\tiny{low}}}(\one,1) = 1.
				\end{equs}
\end{example}

As in \cite[Thm. 3.17]{BS}, we get the following error for our approximation under the assumption:
\begin{equs} \label{assumption_1}
\mathcal{Q}_{\leq r+1} \Pi^{n,r} = \mathcal{Q}_{\leq r+1} \Pi^{n,r'}, \quad r \leq r'.
\end{equs}

\begin{theorem} \label{approxima_tree_paired}
For every forests $ F_1 $ and $ F_2 $, one has,
\begin{equs} \label{identity_main}
\left(\Pi (F_1 \cdot F_2) - \CQ_{\leq r+1} \Pi^{n,r}( F_1 \cdot F_2) \right)(\tau)  = \mathcal{O}\left( \tau^{r+2}  \mathcal{L}^{r}_{\text{\tiny{low}}}(F_1 \cdot F_2,n)  \right)
\end{equs}
where $\Pi $ is defined in \eqref{Pi},  $\Pi^n$ is given in \eqref{recursive_pi_r}  and $\Pi^{n,r} = \Pi^n  \CD^r$ satisfying \eqref{assumption_1}.
\end{theorem}
\begin{proof} The proof follows the same steps given in \cite[Thm. 3.17]{BS}. However, one has to be more precise for dealing with products as now frequencies $ k_1,...,k_m $ can appear several times. We proceed with the following decomposition:
\begin{equs}
 \Pi &(F_1 \cdot F_2)(\tau)  -  \mathcal{Q}_{\leq r+1} \Pi^{n,r} ( F_1 \cdot F_2 )(\tau)   
\\  = & \left( \Pi F_1 -  \CQ_{\leq r+1-\hat{n}_+(F_2)} \Pi^{n,r-\hat{n}_+(F_2)} F_1 \right)(\tau) (\Pi F_2)(\tau)
\\ & + \CQ_{\leq r+1-\hat{n}_+(F_2)}  \left(\Pi^{n,r-\hat{n}_+(F_2)} F_1 \right)(\tau) (\Pi F_2)(\tau)  -  \mathcal{Q}_{\leq r+1} \Pi^{n,r} ( F_1 \cdot F_2 )(\tau) 
\\ = & \left( \Pi F_1 -  \CQ_{\leq r+1-\hat{n}_+(F_2)}  \Pi^{n,r-\hat{n}_+(F_2)} F_1 \right)(\tau) (\Pi F_2)(\tau) \\ & + \sum_{\ell \leq r +1  - \hat{n}_+(F_2)} \tau^{\ell} (\Pi^{n,r} F_1)_{\ell}(\tau) \left( \Pi F_2 - \CQ_{\leq r +1-\ell} \Pi^{n,r-\ell} F_2\right)(\tau)
\end{equs}
where we have used the following identities
\begin{equs}
& \CQ_{\leq r +1  - \hat{n}_+(F_2)}\left(\Pi^{n,r-\hat{n}_+(F_2)} F_1 \right)(\tau)  = \sum_{\ell \leq r +1 - \hat{n}_+(F_2)} \tau^{\ell} (\Pi^{n,r-\hat{n}_+(F_2)} F_1)_{\ell}(\tau) \\
&\CQ_{\leq r+1} \left( \Pi^{n,r}( F_1 \cdot F_2) \right)(\tau)  =
 \sum_{\ell \leq r +1 - \hat{n}_+(F_2)} \tau^{\ell} (\Pi^{n,r} F_1)_{\ell}(\tau) \mathcal{Q}_{\leq r+1 -\ell}( \Pi^{n,r-\ell} F_2 )(\tau)
\end{equs}
and the fact that by Assumption~\eqref{assumption_1}, one has:
\begin{equs}
(\Pi^{n,r-\hat{n}_+(F_2)} F_1)_{\ell}(\tau) = (\Pi^{n,r} F_1)_{\ell}(\tau).
\end{equs}
Then, we apply  recursively the bounds on $ F_1 $ and $ F_2 $ to get:
\begin{equs}
\left( \Pi F_1 -  \CQ_{\leq r+1 - \hat{n}_+(F_2)} \Pi^{n,r-\hat{n}_+(F_2)} F_1 \right)(\tau) & = \mathcal{O}\left( \tau^{r-\hat{n}_+(F_2)+2}  \mathcal{L}^{r-\hat{n}_+(F_2)}_{\text{\tiny{low}}}(F_1,n) \right)
\\
\left( \Pi F_2 - \CQ_{\leq r+1 -\ell} \Pi^{n,r-\ell} F_2\right)(\tau) & = \mathcal{O}\left( \tau^{r-\ell+2}  \mathcal{L}^{r-\ell}_{\text{\tiny{low}}}(F_2,n) \right)
\end{equs}
which allows us to conclude. It remains to treat the case when $ F = \CI_{o_2}( \lambda^{\ell}_{k}  \bar{F} ) $. In that case, one has:
\begin{equs}
&\left( \Pi-\Pi^{n,r} \right) \left( F \right) (\tau)  = - i|\nabla|^{\alpha} (k) \int_{0}^{\tau}  \xi^{\ell} e^{ i\xi P_{o_2}(k)} (\Pi- \Pi^{n,r- 1-\ell})( \bar F)(\xi) d \xi 
\\ &   - i|\nabla|^{\alpha} (k) \int_{0}^{\tau} \xi^{\ell} e^{ i\xi P_{o_2}(k)}  (\Pi^{n,r-1-\ell}  \bar F)(\xi) d \xi -\CK^{k,r}_{o_2} ( \xi^{\ell} (\Pi^{n,r-1-\ell} \bar{F} )(\xi) )(\tau)
 \\ & = \int_{0}^{\tau} \CO \left( \xi^{r+1} k^{\alpha} \mathcal{L}^{r-1-\ell}_{\text{\tiny{low}}}(\bar F,n) \right)  d\xi +  \CO(\tau^{r+2} R^{k,r}_{o_2,n}\left(\xi^{\ell} (\Pi^{n,r-1-\ell} \bar{F} )(\xi) \right)   ) \\
& =   \CO \left( \tau^{r+2} \mathcal{L}^{r}_{\text{\tiny{low}}}(F,n) \right) .
\end{equs}
where the term $   R^{k,r}_{o_2,n}\left(\xi^{\ell} (\Pi^{n,r-1-\ell} \bar{F} )(\xi) \right)   $  is obtained by applying Lemma~\ref{Taylor_bound}.
\end{proof}

\begin{remark}  An immediate consequence of Theorem~\ref{approxima_tree_paired}, is that for a paired decorated forest $ F =T_1 \cdot T_2 $, one has
\begin{equs} \label{identity_main_bis}
\left(\bar{\Pi} T_1 \Pi T_2 - \CQ_{\leq r+1} \bar{\Pi}^{n,r} T_1 \Pi^{n,r} T_2 \right)(\tau)  = \mathcal{O}\left( \tau^{r+2}  \mathcal{L}^{r}_{\text{\tiny{low}}}(T_1 \cdot T_2,n)  \right)
\end{equs}
\end{remark}

\begin{remark}Assumption~\ref{assumption_1} is satisfied by the examples covered in this work as we use mainly full Taylor expansions for describing the numerical scheme. In full generality, Assumption~\ref{assumption_1} does, however, not hold true (see also Remark \ref{rem:sysK}) and one has to work with a two-parameter family in order to describe the local error. The identity \eqref{identity_main} becomes:
\begin{equs}
\left(\Pi (F_1 \cdot F_2) - \CQ_{\leq m} \Pi^{n,r}( F_1 \cdot F_2) \right)(\tau)  = \mathcal{O}\left( \tau^{r+2}  \mathcal{L}^{r,m}_{\text{\tiny{low}}}(F_1 \cdot F_2,n)  \right), \quad m \leq r+1
\end{equs}
where $ \mathcal{L}^{r,m}_{\text{\tiny{low}}}(F_1 \cdot F_2,n) $ can be defined in a recursive way as in Definition~\ref{def:Llow}. The main change occurs in the approximation of a decorated tree given by Lemma~\ref{Taylor_bound}. It is changed into
\begin{equs}
& - i|\nabla|^{\alpha} (k) \int_{0}^{\tau} \xi^{\ell} e^{ i\xi P_{o_2}(k)}  (\Pi^{n,r-1-\ell}  \bar F)(\xi) d \xi -\CQ_{\leq m}\CK^{k,r}_{o_2} ( \xi^{\ell} (\Pi^{n,r-1-\ell} \bar{F} )(\xi) )(\tau)
\\ & =  \CO(\tau^{m+1} R^{k,r,m}_{o_2,n}\left(\xi^{\ell} (\Pi^{n,r-1-\ell} \bar{F} )(\xi) \right)   )
\end{equs}
The new quantity $ R^{k,r,m}_{o_2,n} $ is the new building block for recursively defining  $  \mathcal{L}^{r,m}_{\text{\tiny{low}}} $.
\end{remark}
\section{Low regularity scheme}
\label{sec::4}
The mild solution of  \eqref{dis} is given by Duhamel's formula
\begin{equation}\label{duhLin_it}
u(\tau,v^{\eta}) = e^{ i \tau \mathcal{L}\left(\nabla\right)} v^{\eta}  - i\vert \nabla \vert^\alpha e^{ i \tau  \mathcal{L}\left(\nabla \right)}  \int_0^{\tau} e^{ -i\xi  \mathcal{L}\left(\nabla \right)} p(u(\xi),\bar u(\xi)) d\xi .
\end{equation}
In the sequel, we will focus on nonlinearities of type
\begin{equs}\label{poly}
p(u, \bar u) = u^N \bar u^M.
\end{equs}
It has been proven in \cite[Prop. 4.3]{BS} that the following tree series expansion is the $k$-th Fourier coefficient of a solution of $ \eqref{duhLin_it} $ up to order $ r +1 $:
\begin{equs}\label{genscheme}
U_{k}^{r}(\tau, v^{\eta}) = \sum_{T \in \hat{\CT}^{r,k}_{0}(R)} \frac{\Upsilon^{p}( T)(v^{\eta})}{S(T)} (\Pi   T )(\tau)
\end{equs}
where 
\begin{itemize}
\item For a decorated tree $ T_{\Labe} = (T,\Labe)$ with only edge decorations, we define the symmetry 
factor $S(T_{\Labe})$  inductively by  $S(\one)\,  { =} 1$, while if 
$T$ is of the form
\begin{equs}  
\prod_{i,j}  \mathcal{I}_{(\Labhom_{t_i},p_i)}\left( T_{i,j}\right)^{\beta_{i,j}}, \quad   
\end{equs}
with $T_{i,j} \neq T_{i,\ell}$ for $j \neq \ell$, then 
\begin{align}\label{S}
S(T)
\,  { :=} 
\Big(
\prod_{i,j}
S(T_{i,j})^{\beta_{i,j}}
\beta_{i,j}!
\Big)\;.
\end{align}
We extend this definition to any tree $ T_{\Labe}^{\Labn,\Labo} $ in $ \CT $ by setting:
\begin{equs}
S(T_{\Labe}^{\Labn,\Labo} )\,  { :=}  S(T_{\Labe} ).
\end{equs}
\item Then, we define the map $ \Upsilon^{p}(T)(v) $ for 
\begin{equs}
 T  = 
\CI_{(\Labhom_1,a)}\left( \lambda_k \CI_{(\Labhom_2,a)}( \lambda_k   \prod_{i=1}^n \CI_{(\Labhom_1,0)}( \lambda_{k_i} T_i) \prod_{j=1}^m \CI_{(\Labhom_1,1)}( \lambda_{\tilde k_j} \tilde T_j)  ) \right), \quad a \in \lbrace 0,1 \rbrace  
\end{equs}
by
\begin{equs}\label{upsi}
\Upsilon^{p}(T)(v)& \,  { :=}  \partial_v^{n} \partial_{\bar v}^{m} p_a(v,\bar v) \prod_{i=1}^n  \Upsilon^p( \CI_{(\Labhom_1,0)}\left( \lambda_{k_i}  T_i \right) )(v) \\ & \prod_{j=1}^m \Upsilon^p( \CI_{(\Labhom_1,1)}( \lambda_{\tilde k_j}\tilde T_j ) )(v)
\end{equs}
and 
\begin{equs}
\Upsilon^{p}(\CI_{(\Labhom_1,0)}( \lambda_{k})  )(v)  \,  { :=}   v_k, \quad 
\Upsilon^{p}(\CI_{(\Labhom_1,1)}( \lambda_{k})  )(v)  \,  { :=} \bar{v}_k.
\end{equs}
Above, we have used the notation:
\begin{equs}
p_0(v,\bar{v}) = p(v,\bar{v}), \quad p_{1}(v,\bar{v}) = \overline{p(v, \bar{v})} 
\end{equs}
In the sequel, we will use the following short hand notation:
\begin{equs}
\overline{\Upsilon^{p}(T)}(v) = \bar{\Upsilon}^{p}(T)(v).
\end{equs}
\item We set
\begin{equs}
\hat \CT_0(R)&  = \lbrace \CI_{(\Labhom_1,0)}( \lambda_k \CI_{(\Labhom_2,0)}( \lambda_k   \prod_{i=1}^N T_i \prod_{j=1}^M  \tilde T_j  ) ), \CI_{(\Labhom_1,0)}(\lambda_k) \,  \\ & \quad T_i \in \hat{\CT}_0(R), \, \tilde{T}_j \in \bar{\hat{\CT}}_0(R), \, k \in \Z^{d}  \rbrace \\
\bar{\hat{\CT}}_0(R)&  = \lbrace \CI_{(\Labhom_1,1)}( \lambda_k \CI_{(\Labhom_2,1)}( \lambda_k   \prod_{i=1}^N T_i \prod_{j=1}^M  \tilde T_j  ) ), \CI_{(\Labhom_1,1)}(\lambda_k) \,  \\ & \quad T_i \in \bar{\hat{\CT}}_0(R), \, \tilde{T}_j \in \bar{\hat{\CT}}_0(R), \, k \in \Z^{d}  \rbrace .
\end{equs}
For a fixed $ k \in \Z^d$, we denote the set $ \hat{\CT}^k_0(R) $ (resp. $ \bar{\hat{\CT}}^k_0(R) $) as the subset of $ \hat{\CT}_0(R) $ (resp.$ \bar{\hat{\CT}}_0(R) $) whose decorated trees have decorations on the node connected to the root given by $ k $. For $ r \in \Z $, $ r \geq -1 $, we set:
\begin{equs}
 \hat \CT_0^{r,k}(R) = \lbrace  T_{\Labe}^{\Labo} \in \hat \CT^{k}_0{ (R)}   \,  , n_+(T_{\Labe}^{\Labo}) \leq r +1  \rbrace. 
\end{equs}
In the previous space, we disregard iterated integrals which have more than $ r+1 $ integrals and will be of order $ \mathcal{O}(\tau^{r+2}) $. We denote by $ \CG^{r,k}_0(R) \subset \CG^{r}_k $ the paired decorated forests $ F = T_1 \cdot T_2 $ such that $ T_1, T_2 \in \hat{\CT}_0^{r,k}(R) $ and $ n_+(F) \leq r + 1 $.
\end{itemize}
We want to compute the second moment of the truncated sum \eqref{genscheme} and then provide a discretisation. Before doing so, let us recall Wick's theorem on the higher order moments of Gaussian random variables that is needed.
\begin{proposition} \label{Wick}
Let $ I $ be a finite set and $ (X_{i})_{i \in I} $ a collection of centred jointly Gaussian random variables. Then
\begin{equs}
\mathbb{E}\left(\prod_{i \in I} X_i \right) = \sum_{P \in \mathcal{P}(I)} \prod_{\lbrace i,j \rbrace \in P} \mathbb{E}(X_i X_j)
\end{equs}
where $ \mathcal{P}(I) $ are partitions of $ I $ with two elements of $I$ in each block of the partition.
\end{proposition}

\begin{proposition} \label{prop_moment_2} One has:
\begin{equs}
\CQ_{ \leq r+1} \mathbb{E}(|U_{k}^{r}(\tau, v^{\eta})|^2)  & =  \sum_{F = T_1 \cdot T_2 \in \CG^{r,k}_0(R)} m_F \frac{\bar{\Upsilon}^{p}(T_1)(v) \, \Upsilon^{p}(T_2)(v)}{S(T_1) S(T_2)} \\ & \left( \bar{\Pi} T_1 \right)(\tau)  \, \left( \Pi T_2 \right)(\tau).
\end{equs}
where the $ m_F $ belongs to $ \N $.
\end{proposition}
\begin{proof}
We first notice that:
\begin{equs} \label{comput_E}
\mathbb{E}(|U_{k}^{r}(\tau, v^{\eta})|^2) & = \sum_{T_1, T_2 \in \hat \CT_0^{r,k}(R)} \mathbb{E} \left(  \frac{\bar{\Upsilon}^{p}(T_1)(v^{\eta}) \, \Upsilon^{p}(T_2)( v^{\eta})}{S(T_1) S(T_2)}\right) \\ & \left( \bar{\Pi} T_1 \right)(\tau)  \, \left( \Pi T_2 \right)(\tau)
\end{equs}
Now, we use the fact that for each decorated tree $ T $, one has 
\begin{equs}
\Upsilon^{p}(T)(v^{\eta}) =\Upsilon^{p}(T)(v) \prod_{u \in L_T}  \eta^{u}_{k_{\Labo(u)}}
\end{equs}
where $ \eta^{u}_{k_{\Labo(u)}} $ is given by 
\begin{equs}
\eta^{u}_{k_{\Labo(u)}} = \overline{\eta_{k_{\Labo(u)}}}, \quad \text{if } \mathfrak{p}(e_u) = 1, \quad
 \eta^{u}_{k_{\Labo(u)}} = \eta_{k_{\Labo(u)}}, \quad \text{if } \mathfrak{p}(e_u) = 0,
\end{equs}
where $ \mathfrak{p}(e_u) $ is the second decoration on the edge $ e_u $ connecting the leaf $ u $ to the rest of the tree $ T $.
Then, using the Wick theorem, one gets:
\begin{equs}
& \mathbb{E} \left(  \bar{\Upsilon}^{p}(T_1)(v^{\eta}) \, \Upsilon^{p}(T_2)( v^{\eta})\right)  \\ & = \sum_{P \in \mathcal{P}(L_{T_1} \sqcup L_{T_2})} \prod_{\lbrace u,v \rbrace \in P} \mathbb{E}(\tilde{\eta}^{u}_{k_{\Labo(u)}} \tilde{\eta}^{v}_{k_{\Labo(v)}})
  \bar{\Upsilon}^{p}(T_1)( v) \, \Upsilon^{p}(T_2)( v).
\end{equs}
where
\begin{equs}
\tilde{\eta}^{v}_{k_{\Labo(v)}} = 
 \left\{ \begin{aligned}
  &   \overline{\eta}^{v}_{k_{\Labo(v)}}, \,
  \text{if } v \in L_{T_1}  , \\
  & \eta^{v}_{k_{\Labo(v)}}, \,
  \text{if } v \in L_{T_2}, 
  \\
  \end{aligned} \right. 
\end{equs}
We obtain 
\begin{equs}\label{assumpIC}
\mathbb{E}(\tilde{\eta}^{u}_{k_{\Labo(u)}} \tilde{\eta}^{v}_{k_{\Labo(v)}})  = \left\{ \begin{aligned}
  &   \delta_{(-1)^{\mathfrak{p}(e_u)}\Labo(u)  + (-1)^{\mathfrak{p}(e_v)}\Labo(v),0 }, \,
  \text{if } \lbrace u,v \rbrace \subset   L_{T_1} \text{ or } \lbrace u,v \rbrace \subset L_{T_2}  , \\
  & \delta_{(-1)^{\mathfrak{p}(e_u) +1}\Labo(u)  + (-1)^{\mathfrak{p}(e_v)}\Labo(v),0 }, \,
  \text{ otherwise}. 
  \\
  \end{aligned} \right.  
\end{equs}
that fixes the condition \eqref{pairFreq}. The map $ \delta_{\ell_1,\ell_2} $ is given by
\begin{equs}
\delta_{\ell_1,\ell_2} = 1 \quad \text{if } \ell_1 + \ell_2 = 0, \,\text{and } 0 \text{ otherwise}.
\end{equs}
The coefficients $ m_F $ reflects the number of pairings that can give the same $ F $ which allows us to conclude.
\end{proof}
From Proposition~\ref{prop_moment_2}, our general scheme is given by:
\begin{equs} \label{general_scheme_2}
V_{k}^{n,r}(\tau, v) & = \sum_{ F = T_1 \cdot T_2 \in \CG^{r,k}_0(R)} m_F \frac{\bar{\Upsilon}^{p}(T_1)(v) \, \Upsilon^{p}(T_2)(v)}{S(T_1) S(T_2)}  \\ & \mathcal{Q}_{\leq r+1} \left( \bar{\Pi}^{n,r} T_1 \Pi^{n,r} T_2 \right)(\tau) 
\end{equs}
The second moment of the $k$-th Fourier coefficient $ u_k $ of the solution of \eqref{duhLin_it} is given by: $
V_k = \mathbb{E}(|u_k|^2).$
\begin{theorem}[Local error]\label{thm:genloc} 
The numerical scheme \eqref{general_scheme_2}  with initial value $v = u(0)$ approximates the exact second moment  $V_{k}(\tau,v) $ up to a  local error of type
\begin{equs}
V_{k}^{n,r}(\tau,v) - V_{k}(\tau,v) & = \sum_{T_1 \cdot T_2 \in \CG^{r,k}_0(R)} \CO\left(\tau^{r+2} \CL^{r}_{\text{\tiny{low}}}(T_1 \cdot T_2,n) \bar{\Upsilon}^{p}( T_1)(v)\Upsilon^{p}( T_2)(v)  \right)
\\  & + \CO\left( \tau^{r+2} |\nabla|^{\alpha (r+2)} (k) \,  \mathbb{E}(\tilde p_k(u(t),v^\eta))\right)
\end{equs}
for some polynomial $\tilde p_k$ and $ 0 \leq t \leq \tau $ and where the operator $\CL^{r}_{\text{\tiny{low}}}(T_1 \cdot T_2,n)$, given in Definition \ref{def:Llow}, embeds the necessary regularity of the solution.
\end{theorem}
\begin{proof}
The second moment of the $ k $-th Fourier coefficient of the solution  up to order $ r $ is given by
\begin{equs}
V_{k}^{r}(\tau,v) =  \sum_{F = T_1 \cdot T_2 \in \CG^{r,k}_0(R)} m_F \frac{\bar{\Upsilon}^{p}(T_1)(v) \, \Upsilon^{p}(T_2)(v)}{S(T_1) S(T_2)}   \left( \bar{\Pi} T_1 \Pi T_2 \right)(\tau) 
\end{equs}
which satisfies
\begin{equs}\label{app1}
 V_k(\tau,v) - V_{k}^{r}(\tau,v)  = 
 \CO\left( \tau^{r+2} |\nabla|^{\alpha (r+2)} (k) \,  \mathbb{E}(\tilde p_k(u(t),v^\eta))\right)
\end{equs}
 for some polynomial $\tilde p_k$ and $ 0 \leq t \leq \tau $. Thanks to  Theorem~\ref{approxima_tree_paired}, we furthermore obtain that
\begin{equs}\label{app2}\begin{aligned}
V_{k}^{n,r} & (\tau,v) - V_{k}^{r}(\tau,v) 
\\ & = \sum_{F = T_1 \cdot T_2 \in \CG^{r,k}_{0}(R)} m_F \frac{\bar{\Upsilon}^{p}( T_1)\Upsilon^{p}( T_2)}{S(T_1)S(T_2)}(v)
\\ & \left( \bar{\Pi} T_1 \Pi T_2  -\CQ_{\leq r +1 }\bar{\Pi}^{n,r} T_1 \Pi^{n,r} T_2 \right) (\tau) \\
& = \sum_{T_1 \cdot T_2  \in \CG^{r,k}_{0}(R)}  \CO\left(\tau^{r+2} \CL^{r}_{\text{\tiny{low}}}(T_1 \cdot T_2 ,n)  \bar{\Upsilon}^{p}( T_1)(v) \Upsilon^{p}( T_2)(v)\right).
\end{aligned}
\end{equs}
Next we write
\begin{equs}
V_{k}^{n,r}(\tau,v) - V_{k}(\tau,v)  & = V_{k}^{n,r}(\tau,v) - V_k^r(\tau,v)  + V_k^r(\tau,v) - V_{k}(\tau,v)
\end{equs}
where by the definition of $\CL^{r}_{\text{\tiny{low}}}(T_1 \cdot T_2,n)$ we easily see that the approximation error~\eqref{app2} is in general dominant compared to \eqref{app1} (see KdV at order one for a counter-example in Section~\ref{sec:KDV}). 
\end{proof}

We note that generally the terms $ \CL^{r}_{\text{\tiny{low}}}(T_1 \cdot T_2,n)$ are the leading terms in the local error analysis.

\section{Applications}\label{sec:app}
In this section we detail our general numerical scheme \eqref{general_scheme_2} together with its error analysis (see Theorem  \ref{thm:genloc}) on two concrete examples: The cubic Schr\"odinger equation (see Section \ref{sec:NLS}) and the Korteweg--de Vries equation (see Section \ref{sec:KDV}).

\subsection{Nonlinear Schr\"odinger}\label{sec:NLS}
As a first example let us consider the  nonlinear Schr\"odinger (NLS) equation
\begin{equation}\label{NLS}
i  \partial_t u(t,x) +  \Delta u(t,x) =  \vert u(t,x)\vert^2 u(t,x)\quad (t,x) \in \R \times  \T^d 
\end{equation}
together with an initial condition of the form \eqref{ic}. 
The NLS equation \eqref{NLS} fits into the general form \eqref{dis} with
\begin{equation*}\label{kgrDo}
\begin{aligned}
 \mathcal{L}\left(\nabla \right)  = \Delta, \quad \alpha = 0,   \quad \text{and}\quad   p(u,\overline u) =   |u|^2u.
 \end{aligned}
\end{equation*} 
We set $\Lab = \{\mathfrak{t}_1, \mathfrak{t}_2\}$, $P_{\mathfrak{t}_1} = -X^2$ and $P_{\mathfrak{t}_2} = X^2$.  An edge decorated by $(\mathfrak{t}_1, 0)$ is denoted by $ \<thick> $, while an edge decorated by $(\mathfrak{t}_1, 1)$ is denoted by $ \<thick2> $. Similarly, an edge decorated by $(\mathfrak{t}_2, 0)$ is denoted by $\<thin>$, and by $\<thin2>$ an edge decorated by $(\mathfrak{t}_2, 1)$.

\subsubsection{First order scheme}\label{NLS1}
We start by expanding the solution as a power series expansion generated by its Duhamel's iterates (see \eqref{duhamel}). For a first order approximation we have,
\begin{equation}\label{expansion1}
u(\tau,v^{\eta}) = u^0(\tau,v^{\eta}) + u^1(\tau,v^{\eta}) + \CO(\tau^2),
\end{equation}
where $u^{m}(\tau,v^{\eta})$ denotes the $m$-th Duhamel iterate. Namely we have,
\begin{align*}
u^0(\tau, v^{\eta}) = e^{i \tau \Delta } v^{\eta}, \quad u^1(\tau,v^{\eta})= e^{i\tau\Delta}\int_0^\tau e^{-i\zeta \Delta} (e^{i\zeta \Delta}v^{\eta})^2(e^{-i\zeta \Delta}\overline{v^{\eta}}) d\zeta.
\end{align*}
Using \eqref{decoratedV1} one can expand the $k$-th Fourier coefficient of the above truncated solution \eqref{expansion1}, and of its conjugate, as a sum over all decorated ternary trees of size at most one. 
As such, using \eqref{series-trunc}, we can express a first order truncated approximation of $ \mathbb{E}(|u_k(\tau,v^{\eta})|^2)$ as a sum over paired couples of decorated trees. 
We now explain how these pairings are made, and construct the set of paired couples $\mathcal{G}^0_k$.

For a first order approximation we fix $r=0$, and wish to construct an approximation of $ \mathbb{E}\left(|u_k(\tau,v^{\eta})|^2\right)$ with a local error of second order. We have
\begin{equs}\label{o1Exp} \begin{aligned}
{\mathbb{E}}\left(|u_k(\tau,v^{\eta})|^2\right) &= {\mathbb{E}}\Bigg(\left(u^0_k(\tau,v^{\eta})+ u^1_k(\tau,v^{\eta})\right)\left(\overline{u^0_k(\tau,v^{\eta})+ u^1_k(\tau,v^{\eta})}\right)\Bigg)\\ 
&+ \mathcal{O}(\tau^2)\\\nonumber
&= {\mathbb{E}}\left(|u^0_k(\tau,v^{\eta})|^2\right) + 2Re\Bigg( {\mathbb{E}}\left(\overline{u^0_k(\tau,v^{\eta})}u^1_k(\tau,v^{\eta})\right) \Bigg) \\
&+  \mathcal{O}(\tau^2).
\end{aligned}
\end{equs}

Each of the above terms involves expectations of products of Gaussians. Hence, their computations revolve around applications of Wick's formula, which we have encoded using a decorated tree formalism at equation \eqref{pairFreq}.
We start by dealing with the first term $\mathop{\mathbb{E}}\left(u^0_k(\tau,v^{\eta})\overline{u^0_k(\tau,v^{\eta})} \right)$. In order for this expectation not to be a trivial one we require the frequencies to be equal. Namely, the couple of paired tree stemming from this pairing is 
\begin{equs}\label{NLS_T1}
 F_1 = T_1 \cdot T_1, \quad 
 T_1 = \CI_{(\mathfrak{t_1},0)}(\lambda_k) = \begin{tikzpicture}[scale=0.2,baseline=-5]
\coordinate (root) at (0,1);
\coordinate (tri) at (0,-1);
\draw[kernels2] (tri) -- (root);
\node[var] (rootnode) at (root) {\tiny{$ k $}};
\node[not] (trinode) at (tri) {};
\end{tikzpicture}.
\end{equs}

We now deal with the computation of the second term $\mathop{\mathbb{E}}\left(\overline{u^0_k(\tau,v^{\eta})}u^1_k(\tau,v^{\eta})\right)$. In this case, there are two possible pairings of the frequencies that satisfy equation \eqref{pairFreq}, and which yield a nontrivial computation of this second term. See the Example \ref{ex} for more details. 
The couple of paired trees stemming from these two pairings are:
\begin{equs}\label{NLS_T2}
 F_2 =  T_1 \cdot T_2, 
\quad T_2 = \begin{tikzpicture}[scale=0.2,baseline=-5]
\coordinate (root) at (0,-1);
\coordinate (t3) at (0,1);
\coordinate (t4) at (0,3);
\coordinate (t41) at (-2,5);
\coordinate (t42) at (2,5);
\coordinate (t43) at (0,7);
\draw[kernels2] (t3) -- (root);
\draw[symbols] (t3) -- (t4);
\draw[kernels2,tinydots] (t4) -- (t41);
\draw[kernels2] (t4) -- (t42);
\draw[kernels2] (t4) -- (t43);
\node[not] (rootnode) at (root) {};
\node[not] (rootnode) at (t4) {};
\node[not] (rootnode) at (t3) {};
\node[var] (rootnode) at (t41) {\tiny{$ \bar{k} $}};
\node[var] (rootnode) at (t42) {\tiny{$ k_{\tiny{2}} $}};
\node[var] (rootnode) at (t43) {\tiny{$  \bar{k}_{\tiny{2}} $}};
\end{tikzpicture}, 
\end{equs}
where $ \bar{k} = -k $ and  in symbolic notation we have 
\begin{equs}\label{NLS_T2sym}
T_2 = \CI_{{(\mathfrak{t}_1,0)}}\lambda_k\Big(\CI_{{(\mathfrak{t}_2,0)}}\left(\lambda_{k} \CI_{(\mathfrak{t}_1,1)}(\lambda_{-k})\CI_{(\mathfrak{t}_1,0)}(\lambda_{-k_2}) \CI_{{(\mathfrak{t}_1,0)}}(\lambda_{k_2})\right)\Big). \quad
\end{equs}
Finally the third pairing to consider is,
\begin{equs}\label{NLS_T3}
 F_3 =  T_1 \cdot T_3, 
\quad T_3 = \begin{tikzpicture}[scale=0.2,baseline=-5]
\coordinate (root) at (0,-1);
\coordinate (t3) at (0,1);
\coordinate (t4) at (0,3);
\coordinate (t41) at (-2,5);
\coordinate (t42) at (2,5);
\coordinate (t43) at (0,7);
\draw[kernels2] (t3) -- (root);
\draw[symbols] (t3) -- (t4);
\draw[kernels2,tinydots] (t4) -- (t41);
\draw[kernels2] (t4) -- (t42);
\draw[kernels2] (t4) -- (t43);
\node[not] (rootnode) at (root) {};
\node[not] (rootnode) at (t4) {};
\node[not] (rootnode) at (t3) {};
\node[var] (rootnode) at (t41) {\tiny{$ k_2 $}};
\node[var] (rootnode) at (t42) {\tiny{$ k $}};
\node[var] (rootnode) at (t43) {\tiny{$  k_{\tiny{2}} $}};
\end{tikzpicture},
\end{equs}
where in symbolic notation we have,
$$
T_3= \CI_{{(\mathfrak{t}_1,0)}}\lambda_k\Big(\CI_{{(\mathfrak{t}_2,0)}}\left(\lambda_{k} \CI_{(\mathfrak{t}_1,1)}(\lambda_{k_2})\CI_{(\mathfrak{t}_1,0)}(\lambda_{k_2}) \CI_{{(\mathfrak{t}_1,0)}}(\lambda_{k})\right) \Big).
$$
There is another possible pairing which is equivalent to $F_3$, and is given by,
\begin{equs}\label{tildeF_3}
 \tilde{F_3} =  T_1 \cdot T_3, 
\quad T_3 = \begin{tikzpicture}[scale=0.2,baseline=-5]
\coordinate (root) at (0,-1);
\coordinate (t3) at (0,1);
\coordinate (t4) at (0,3);
\coordinate (t41) at (-2,5);
\coordinate (t42) at (2,5);
\coordinate (t43) at (0,7);
\draw[kernels2] (t3) -- (root);
\draw[symbols] (t3) -- (t4);
\draw[kernels2,tinydots] (t4) -- (t41);
\draw[kernels2] (t4) -- (t42);
\draw[kernels2] (t4) -- (t43);
\node[not] (rootnode) at (root) {};
\node[not] (rootnode) at (t4) {};
\node[not] (rootnode) at (t3) {};
\node[var] (rootnode) at (t41) {\tiny{$ k_2 $}};
\node[var] (rootnode) at (t42) {\tiny{$ k_2 $}};
\node[var] (rootnode) at (t43) {\tiny{$  k $}};
\end{tikzpicture}.  
\end{equs}
The above pairing will yield the same quantity to approximate as the pairing of $F_3$, since the resonance structure obtained through both pairings are identical: the resonance structure is zero. Hence, we only consider the paired decorated forest $F_3$. We keep track of this by setting $m_{F_3}=2$, while $m_{F_1} = 1 = m_{F_2}$. 
In conclusion, the set of couples of paired decorated forests is given by 
$$\mathcal{G}^{0,k}_{0} = \{ T_1 \cdot T_1, \ T_1 \cdot T_2, \ T_1 \cdot T_3 \}$$
We are now ready to introduce the first order low-regularity approximation to the second order moments of $u_k(\tau,v^{\eta})$, solution of \eqref{NLS} with initial data \eqref{ic}.

\begin{corollary}\label{cor_schemeNLS1}
At first order our general low regularity scheme \eqref{schema_general} takes the form:
\begin{equs}\label{schemeNLS1_Fourier}
V_{k}^{1,0}(\tau,v) = v_k\overline{v_k} ,
\end{equs}
and is locally of order $\CO(\tau^2|\nabla|v)$. The above is the Fourier coefficient associated to the following scheme written in physical space: 
\begin{equs}\label{schemeNLS1}
u^{\ell+1} = u^{\ell} * \tilde{u}^{\ell} 
\end{equs}
where $ \tilde{u}^{\ell}(x) = \overline{u^{\ell}(-x)} $. 
\end{corollary}

\begin{remark}
In order to prove convergence of the schemes presented in this article, one can make an analysis in the Sobolev spaces $H^r, r\ge 0$, and combine a stability argument with an analysis of the local error terms to conclude a global error bound on the numerical scheme. 
This global error analysis is out of scope for this paper, but can nevertheless be made by following for example the steps of \cite{AB}. 
\end{remark}

\begin{proof}
We read from equation \eqref{schema_general} that the first-order scheme has the general form,
\begin{equs}\label{expansionNLS1}
V_{k}^{1,0}(\tau,v)&  = \sum_{F = T\cdot \tilde{T} \in \CG^{0,k}_0(R)} m_F\frac{\bar{\Upsilon}^{p}(T)(v) \, \Upsilon^{p}(\tilde{T})(v)}{S(T) S(\tilde{T})} \\ & \mathcal{Q}_{\le 1} \left((\bar{\Pi}^{1,0} T)(\Pi^{1,0}\tilde{T}) \right)(\tau),\qquad .
\end{equs}
%
%
Furthermore, we have that,
\begin{align}\label{upsilon1}
&\Upsilon^{p}(T_1)(v) = v_k, \quad \Upsilon^{p}(T_2)(v) = 2v_{k}v_{k_{2}}\overline{v_{k_{2}}}= \Upsilon^{p}(T_3)(v),\\\nonumber
& S(T_1) = 1 , \quad S(T_2) = 2 = S(T_3).
\end{align}
Hence, the first-order scheme \eqref{expansionNLS1} takes the form
\begin{equs}\label{expansionNLS1_1}
V_{k}^{1,0}(\tau,v) &= v_k\overline{v_k} \CQ_{\le 1} (\bar{\Pi}^{1,0}(T_1))(\Pi^{1,0}(T_1)) \\
&+ 2Re\left(v_k\overline{v_k}\sum_{k_2\in \Z^d} |v_{k_{2}}|^2 \CQ_{\le 1} \bigg(\bar{\Pi}^{1,0}(T_1)\big(\Pi^{1,0}(T_2) + 2\Pi^{1,0}(T_3) \big)\bigg)(\tau)\right).
\end{equs}
We are left to calculate the first-order approximations, encoded in the definition of the character $\Pi^{1,0}$, of each of our decorated trees appearing in the couples defined in $\CG^{0,k}_0(R)$.
From \eqref{NLS_T1} and by Definition \ref{recursive_pi_r}, given that $P_{(\mathfrak{t}_{1},0)}(k) = -k^2$ and $P_{(\mathfrak{t}_{1},1)}(k) =k^2$ we have,
\begin{equs}\label{compT1_NLS}
(\Pi^{1,0} T_1)(\tau) = e^{i\tau P_{(\mathfrak{t}_{1},0)}(k)}=e^{-i \tau k^2},
\quad
(\bar{\Pi}^{1,0} {T_1})(\tau) = e^{i\tau P_{(\mathfrak{t}_{1},1)}(k)} = e^{i \tau k^2}.\qquad
\end{equs}
Henceforth, for notational convenience we will denote $o_1 = (\mathfrak{t}_1, 0)$, $\overline{o_1} = (\mathfrak{t}_1, 1)$, and $o_2 = (\mathfrak{t}_2, 0)$. Then again by Definition \ref{recursive_pi_r} and \eqref{NLS_T2sym} it follows that,
\begin{equs}\label{NLS_compT2} \begin{aligned}
& \left( \Pi^{1,0} T_2 \right)(\tau)  \\& = e^{-i \tau k^2} \left( \Pi^{1,0} \CI_{o_2}(\lambda_{k} \CI_{\bar{o}_1}(\lambda_{-k})\CI_{o_1}(\lambda_{-k_2}) \CI_{o_1}(\lambda_{k_2})) \right)(\tau)\\ \nonumber
& = e^{-i \tau k^2}  \mathcal{K}_{o_2}^{k,0} \left((\Pi^{n,-1}\CI_{\bar{o}_1}(\lambda_{-k}))(\xi)(\Pi^{n,-1} \CI_{o_1}(\lambda_{-k_2})(\xi) \right. \\ & \left. (\Pi^{n,-1}\CI_{o_1}(\lambda_{k_2}))(\xi) \right)(\tau)\\\nonumber
&= e^{-i \tau k^2}  \mathcal{K}_{o_2}^{k,0}( e^{i\xi(k^2 - 2k_2^2)} ,1)(\tau)\\ \nonumber
& = -i\tau e^{-i \tau k^2},
\end{aligned}
\end{equs}
where we applied Definition \ref{CK} with $q=\ell=0$, $P(k_1,..., k_n) = k^2 - 2k_2^2$, and $P_{o_2}(k) = k^2$. We note that the resonance structure obtain via this pairing is given by
\begin{equs}\label{NonZeroRes}
P(k_1,..., k_n) + P_{o_2}(k) = 2(k^2 - k_2^2).
\end{equs}
Similarly we have,
\begin{equs}\label{NLS_compT3}
(\Pi^{1,0} T_3)(\tau) &= e^{-i \tau k^2}\left( \Pi^{1,0} \CI_{o_2}(\lambda_{k} \CI_{\bar{o}_1}(\lambda_{k_2})\CI_{o_1}(\lambda_{k_2}) \CI_{o_1}(\lambda_{k})) \right)(\tau)\\ 
&= e^{-i \tau k^2}  \mathcal{K}_{o_2}^{k,0}( e^{i\xi(k_2^2 - k_2^2 - k^2)} ,1)(\tau)\\
&= -i\tau e^{-i \tau k^2}, 
\end{equs}
where we notice that the resonance structure for the integral $\Pi T_3$ is zero: 
\begin{equs}\label{ZeroRes}
P(k_1,..., k_n) + P_{o_2}(k) = k^2 + (k_2^2 - k_2^2 - k^2) = 0.
\end{equs} 
Hence exact integration takes place, and no approximation is necessary when applying the approximation operator $\CK_{o_2}$ on the integral corresponding to $T_3$.
Collecting the above computations, and plugging it in the expansion \eqref{expansionNLS1_1} yields the low-regularity scheme \eqref{schemeNLS1_Fourier}.

\underline{\bf Local Error:} It remains to show the first order local error bound: $\CO(\tau^2|\nabla| v)$. By Theorem \ref{thm:genloc} we have
\begin{equs}
V^{n,0}_k(\tau,v) - V_k(\tau,v) = \sum_{T \cdot \tilde{T} \in \CG^{0,k}_0(R)}\CO(\tau^2\CL_{\text{\tiny{low}}}^0(T\cdot \tilde{T},1)\bar{\Upsilon}^p(T)(v)\Upsilon^p(\tilde{T})(v)).
\end{equs}
We now calculate $\CL_{\text{\tiny{low}}}^0(T\cdot \tilde{T},1)$, for every pair $T\cdot \tilde{T}\in \CG_{0}^{0,k}(R)$. First, we consider the pair $F_1 = T_1\cdot T_1$, where we have $n_{+}(T_1) = 0$ and hence where no time discretization is necessary. By Definition \ref{def:Llow} of $\CL_{\text{\tiny{low}}}^r$ we have,
\begin{equs}\label{L_low1}
\CL_{\text{\tiny{low}}}^0(T_1\cdot {T_1},1) = (\Pi T_1)_{0} \CL_{\text{\tiny{low}}}^0(T_1,1) &+ (\Pi^{1,0}T_1)_{0} \CL_{\text{\tiny{low}}}^0({T_1},1) \\
&+ (\Pi^{1,0}T_1)_{1}\CL_{\text{\tiny{low}}}^{-1}({T_1},1),
\end{equs}
where $  \CL_{\text{\tiny{low}}}^0(T_1,1) = \CL_{\text{\tiny{low}}}^{-1}(T_1,1) = 1$.
Furthermore, we have $(\Pi T_1)_{0} = 1$ and $ e^{-i\tau k^2} = (\Pi^{1,0}T_1)_{0}$, 
and finally $(\Pi^{1,0}T_1)_{1} = 0$ since $T_1$ corresponds to a term of only zero-th order. Hence, from \eqref{L_low1} it follows that $\CL_{\text{\tiny{low}}}^{0}(T_1 \cdot T_1,1)= 2e^{-i\tau k^2}$ and as seen previously $\Upsilon^p(T_1) = v_k$. 

Next we consider the pair $F_2 = T_1 \cdot T_2$, where $n_{+}(T_2) = 1$.
Again by Definition \ref{def:Llow} of $\CL_{\text{\tiny{low}}}^r$ we have,
\begin{equs}\label{L_low2}
\CL_{\text{\tiny{low}}}^0(T_1\cdot {T_2},1) = (\Pi T_2)_{0} \CL_{\text{\tiny{low}}}^{-1}(T_1,1) &+ (\Pi^{1,-1}T_1)_{0} \CL_{\text{\tiny{low}}}^0({T_2},1),
\end{equs}
where $\Pi^{1,-1} = \Pi^{1}\CD^{-1} = \Pi^1$, $(\Pi^{1}T_1)_{0} = e^{-i\tau k^2}$, and $(\Pi T_2)_{0} = 1$. Next, we have
$$
\CL_{\text{\tiny{low}}}^0({T_2},1) = \CL_{\text{\tiny{low}}}^0(G_2,1) + R^{k,0}_{o_2,1}(\Pi^{1}G_2),
$$
with $T_2 = \CI_{o_1}\left(\lambda_k\CI_{o_2}(\lambda_k G_2)\right)$, and $G_2 = \CI_{\bar{o}_1}(\lambda_{-k})\CI_{o_1}(\lambda_{-k_2}) \CI_{o_1}(\lambda_{k_2})$. Given that we are performing a full Taylor approximation \eqref{CK} of $\Pi T_2$, we have that the local error $R^{k,0}_{o_2,1}(\Pi^{1}G_2)$ is given by \eqref{R_fullTaylor}. Given the resonance factor \eqref{NonZeroRes} it follows that $R_{o_2,1}^{k,0}(\Pi^1G_2) = R_{o_2,1}^{k,0}(e^{i\xi(k^2 - 2k_2^2)}) = 2(k^2 - k_2^2)$.
Lastly, $\CL_{\text{\tiny{low}}}^0(G_2,1)$ consists of linear combinations involving the propagator $e^{\pm i\tau k^2}$, and not of a polynomial of the frequencies $(k_1, ..., k_n)$. This is due to the fact that $G_2$ only consists of terms of zero-th order, and hence these computations are similar to those already made for $T_1\cdot T_1$. Therefore, 
$\CL_{\text{\tiny{low}}}^0({T_2},1) = \CO(\sum_{k_2 \in \Z^d}(k^2-k_2^2))$, and by \eqref{L_low2},
\begin{equs}\label{localerr-NLS1}
\CL_{\text{\tiny{low}}}^0(T_1\cdot {T_2},1)\bar{\Upsilon}^p(T_1)\Upsilon^p(T_2) = \CO(\sum_{k_2 \in \Z^d} (k^2 - k_2^2)\overline{v_{k}}v_{k}|v_{k_{2}}|^2).
\end{equs}
In physical space this yields an error of $\CO(\tau^2 (\nabla v *\nabla \tilde{v}) \|v\|_{L^2}^2 + (v*\tilde{v})\|\nabla v\|_{L^2}^2)  $.

Finally, the computations for the pair $F_3 = T_1 \cdot T_3$ is a simpler case of the above analysis for the pair $F_2$. Indeed, with the pairing of $F_3$ the resonance factor \eqref{ZeroRes} is zero and hence, no local error is induced by this term: exact integration takes place. Therefore, the local error is given by \eqref{localerr-NLS1} in Fourier space, and in physical space asks for $\CO(\tau|\nabla| v)$, namely requires only one spacial derivative on the initial data thanks to the use of convolutions.
\end{proof}
\subsubsection{Second order schemes}
For a second order approximation we have,
\begin{align*}
\mathop{\mathbb{E}}\left(|u_k(\tau,v^{\eta})|^2\right) &= {\mathbb{E}}\left((u^0_k(\tau,v^{\eta})+ u^1_k(\tau,v^{\eta}) + u^2_k(\tau,v^{\eta}))\right.\\ & \left.(\overline{u^0_k(\tau,v^{\eta})+ u^1_k(\tau,v^{\eta}) + u^2_k(\tau,v^{\eta})})\right) + \CO(\tau^3)\\
&= {\mathbb{E}}\left(|u^0_k(\tau,v^{\eta})|^2\right) +  2 Re{\mathbb{E}}\left(\overline{u^0_k(\tau,v^{\eta})}u^1_k(\tau,v^{\eta})\right)
\\ & + 2Re{\mathbb{E}}\left(\overline{u^0_k(\tau,v^{\eta})}u^2_k(\tau,v^{\eta})\right)
+ {\mathbb{E}}\left(|u^1_k(\tau,v^{\eta})|^2\right)  + \CO(\tau^3),
\end{align*}
where 
\begin{equs}\label{u^2NLS}
\begin{aligned}
u^2(\tau,v^{\eta}) 
&= -\int_{0}^{\tau}e^{i(\tau-\xi)\Delta} \Big( \int_{0}^{\xi}e^{i(\xi-\xi_1)\Delta}  \left((e^{i\xi_1 \Delta} v^{\eta})^2(e^{-i\xi_1\Delta}\overline{v^{\eta}})\right)d\xi_1  \\
& (2(e^{i\xi \Delta} v^{\eta})(e^{-i\xi \Delta}\overline{v^{\eta}})) \Big) d\xi\\
& + \int_{0}^{\tau}e^{i(\tau-\xi)\Delta} \left(\Big( \int_{0}^{\xi} 
e^{-i(\xi-\xi_1)\Delta} (e^{-i\xi_1 \Delta} \overline{v^{\eta}})^2(e^{i\xi_1\Delta}v^{\eta}) d\xi_1 \Big) (e^{i\xi \Delta} v^{\eta})^2 \right) d\xi\\
&:= u^{2,1}(\tau,v^{\eta}) + u^{2,2}(\tau,v^{\eta}),
\end{aligned} 
\end{equs}
and where $u^{2,1}(\tau,v^{\eta})$ and $u^{2,2}(\tau,v^{\eta})$ are encoded by the following two decorated trees respectively:
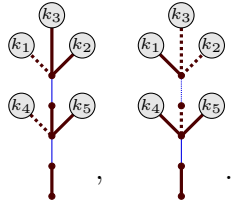
\begin{equs}
\begin{tikzpicture}[scale=0.2,baseline=-5]
\coordinate (root) at (0,2);
\coordinate (tri) at (0,0);
\coordinate (tri2) at (0,-2);
\coordinate (t1) at (-2,4);
\coordinate (t2) at (2,4);
\coordinate (t3) at (0,4);
\coordinate (t4) at (0,6);
\coordinate (t41) at (-2,8);
\coordinate (t42) at (2,8);
\coordinate (t43) at (0,10);
\draw[kernels2] (tri2) -- (tri);
\draw[kernels2,tinydots] (t1) -- (root);
\draw[kernels2] (t2) -- (root);
\draw[kernels2] (t3) -- (root);
\draw[symbols] (root) -- (tri);
\draw[symbols] (t3) -- (t4);
\draw[kernels2,tinydots] (t4) -- (t41);
\draw[kernels2] (t4) -- (t42);
\draw[kernels2] (t4) -- (t43);
\node[not] (rootnode) at (tri2) {};
\node[not] (rootnode) at (root) {};
\node[not] (rootnode) at (t4) {};
\node[not] (rootnode) at (t3) {};
\node[not] (trinode) at (tri) {};
\node[var] (rootnode) at (t1) {\tiny{$ k_{\tiny{4}} $}};
\node[var] (rootnode) at (t41) {\tiny{$ k_{\tiny{1}} $}};
\node[var] (rootnode) at (t42) {\tiny{$ k_{\tiny{2}} $}};
\node[var] (rootnode) at (t43) {\tiny{$ k_{\tiny{3}} $}};
\node[var] (trinode) at (t2) {\tiny{$ k_5 $}};
\end{tikzpicture}, \quad
\begin{tikzpicture}[scale=0.2,baseline=-5]
\coordinate (root) at (0,2);
\coordinate (tri) at (0,0);
\coordinate (tri2) at (0,-2);
\coordinate (t1) at (-2,4);
\coordinate (t2) at (2,4);
\coordinate (t3) at (0,4);
\coordinate (t4) at (0,6);
\coordinate (t41) at (-2,8);
\coordinate (t42) at (2,8);
\coordinate (t43) at (0,10);
\draw[kernels2] (tri2) -- (tri);
\draw[kernels2] (t1) -- (root);
\draw[kernels2] (t2) -- (root);
\draw[kernels2,tinydots] (t3) -- (root);
\draw[symbols] (root) -- (tri);
\draw[symbols,tinydots] (t3) -- (t4);
\draw[kernels2] (t4) -- (t41);
\draw[kernels2,tinydots] (t4) -- (t42);
\draw[kernels2,tinydots] (t4) -- (t43);
\node[not] (rootnode) at (tri2) {};
\node[not] (rootnode) at (root) {};
\node[not] (rootnode) at (t4) {};
\node[not] (rootnode) at (t3) {};
\node[not] (trinode) at (tri) {};
\node[var] (rootnode) at (t1) {\tiny{$ k_{\tiny{4}} $}};
\node[var] (rootnode) at (t41) {\tiny{$ k_{\tiny{1}} $}};
\node[var] (rootnode) at (t42) {\tiny{$ k_{\tiny{2}} $}};
\node[var] (rootnode) at (t43) {\tiny{$ k_{\tiny{3}} $}};
\node[var] (trinode) at (t2) {\tiny{$ k_5 $}};
\end{tikzpicture}.
\end{equs}

Using Wick's formula we can once again obtain the possible pairings which come at play when wanting to compute each of the above terms. This determines the set of decorated pair of trees $\CG_0^{1,k}(R)$, of size at most two, which we detail below. 

From the previous Section \ref{NLS1}, we have that the pair of decorated trees which encodes the first term $\mathbb{E}(|u^0_k(\tau,v^{\eta})|^2)$ is given by $F_1 =T_1 \cdot T_1$, defined in \eqref{NLS_T1}. Similarly, we have that the second term $  {\mathbb{E}}[\overline{u^0_k(\tau,v^{\eta})}u^1_k(\tau,v^{\eta})]$ is encoded by the pairings $F_2=T_1 \cdot T_2 $ and $F_3= T_1 \cdot T_3$, defined in \eqref{NLS_T2} and \eqref{NLS_T3} respectively.
Next, we consider the possible pairings we need to take into account for the computation of ${\mathbb{E}}[\overline{u^0_k(\tau,v^{\eta})}u^2_k(\tau,v^{\eta})]$, where $u^2$ is the sum of two second order terms given in \eqref{u^2NLS}. We start by considering the different pairings between the frequencies of $\overline{u^0}$ and $u^{2,1}$, those for $u^{2,2}$ will then follow. We can classify the different pairings into five groups, which we call the (five) {\it principal} pairings. This classification is based upon two factors: 
\begin{itemize}
\item[(a)]
first, whether the frequency pairing is made internally within the same layer/integral in the tree, 
\item[(b)] secondly, whether the pairing is made between frequencies of the same or opposite sign. 
\end{itemize}

We start by looking at the pairings which consist of at least one pairing of two frequencies made internally within the same integral. There are then two possibilities to consider, which depend on the sign of the frequency pairing.
Namely, we have that a pairing between two frequencies of opposite sign will result in a null resonance factor, as was the case in \eqref{ZeroRes} for $F_3$, and a pairing between two frequencies of the same sign will results in a resonance structure of the form \eqref{NonZeroRes}, namely as a difference of the squares of the frequencies. Hence, there are solely four cases of pairings which we need to consider:
\begin{itemize} 
\item
First, the case when the pairings are all of opposite signs: $\{k_1 = k_2, k_4 = k_5, k = k_3 \} \cup \{k_1 = k_3,k_2 = k, k_4 = k_5 \} \cup \{k_1 = k_2, k_3 = k_4, k = k_5 \} \cup \{k_1 = k_3, k_2 = k_4, k = k_5 \}$.
Without loss of generality we take the first pairing in the union above as the representative of this class, which we include in the set $\CG^{1,k}_0(R)$:
\begin{equs}\label{NLS_T4}
F_4 =  T_1 \cdot T_4, \quad T_1 = \begin{tikzpicture}[scale=0.2,baseline=-5]
\coordinate (root) at (0,1);
\coordinate (tri) at (0,-1);
\draw[kernels2] (tri) -- (root);
\node[var] (rootnode) at (root) {\tiny{$ k $}};
\node[not] (trinode) at (tri) {};
\end{tikzpicture}, \quad T_4 = \begin{tikzpicture}[scale=0.2,baseline=-5]
\coordinate (root) at (0,2);
\coordinate (tri) at (0,0);
\coordinate (tri2) at (0,-2);
\coordinate (t1) at (-2,4);
\coordinate (t2) at (2,4);
\coordinate (t3) at (0,4);
\coordinate (t4) at (0,6);
\coordinate (t41) at (-2,8);
\coordinate (t42) at (2,8);
\coordinate (t43) at (0,10);
\draw[kernels2] (tri2) -- (tri);
\draw[kernels2,tinydots] (t1) -- (root);
\draw[kernels2] (t2) -- (root);
\draw[kernels2] (t3) -- (root);
\draw[symbols] (root) -- (tri);
\draw[symbols] (t3) -- (t4);
\draw[kernels2,tinydots] (t4) -- (t41);
\draw[kernels2] (t4) -- (t42);
\draw[kernels2] (t4) -- (t43);
\node[not] (rootnode) at (tri2) {};
\node[not] (rootnode) at (root) {};
\node[not] (rootnode) at (t4) {};
\node[not] (rootnode) at (t3) {};
\node[not] (trinode) at (tri) {};
\node[var] (rootnode) at (t1) {\tiny{$ k_{\tiny{4}} $}};
\node[var] (rootnode) at (t41) {\tiny{$ k_{\tiny{2}} $}};
\node[var] (rootnode) at (t42) {\tiny{$ k $}};
\node[var] (rootnode) at (t43) {\tiny{$ k_{\tiny{2}} $}};
\node[var] (trinode) at (t2) {\tiny{$ k_4 $}};
\end{tikzpicture}.
\end{equs}
\item Secondly, the case where frequencies of opposite signs are paired in the inner-integral (or tree), whereas frequencies of the same signs are paired in the outer-integral (or tree):
$\{k_1 = k_2, k_3 = -k_5, k = -k_4 \}, \cup \{ k_1 = k_3, k_2 = -k_5, k = -k_4\}$.
We once again take the first pairing in the union above as the representative of this class, which we include in the set $  \CG^{1,k}_0(R) $:
\begin{equs}
F_5 =  T_1 \cdot T_5, \quad T_1 = \begin{tikzpicture}[scale=0.2,baseline=-5]
\coordinate (root) at (0,1);
\coordinate (tri) at (0,-1);
\draw[kernels2] (tri) -- (root);
\node[var] (rootnode) at (root) {\tiny{$ k $}};
\node[not] (trinode) at (tri) {};
\end{tikzpicture}, \quad T_5 = \begin{tikzpicture}[scale=0.2,baseline=-5]
\coordinate (root) at (0,2);
\coordinate (tri) at (0,0);
\coordinate (tri2) at (0,-2);
\coordinate (t1) at (-2,4);
\coordinate (t2) at (2,4);
\coordinate (t3) at (0,4);
\coordinate (t4) at (0,6);
\coordinate (t41) at (-2,8);
\coordinate (t42) at (2,8);
\coordinate (t43) at (0,10);
\draw[kernels2] (tri2) -- (tri);
\draw[kernels2,tinydots] (t1) -- (root);
\draw[kernels2] (t2) -- (root);
\draw[kernels2] (t3) -- (root);
\draw[symbols] (root) -- (tri);
\draw[symbols] (t3) -- (t4);
\draw[kernels2,tinydots] (t4) -- (t41);
\draw[kernels2] (t4) -- (t42);
\draw[kernels2] (t4) -- (t43);
\node[not] (rootnode) at (tri2) {};
\node[not] (rootnode) at (root) {};
\node[not] (rootnode) at (t4) {};
\node[not] (rootnode) at (t3) {};
\node[not] (trinode) at (tri) {};
\node[var] (rootnode) at (t1) {\tiny{$ \bar{k} $}};
\node[var] (rootnode) at (t41) {\tiny{$ k_{\tiny{2}} $}};
\node[var] (rootnode) at (t42) {\tiny{$ k_{\tiny{4}} $}};
\node[var] (rootnode) at (t43) {\tiny{$ k_{\tiny{2}} $}};
\node[var] (trinode) at (t2) {\tiny{$ \bar{k}_{\tiny{4}} $}};
\end{tikzpicture}.
\end{equs}


\item Thirdly, the case where frequencies of the same signs are paired within the inner-integral (or tree), whereas frequencies of opposite signs are paired in the outer-integral (or tree):
$\{k_2 = -k_3, k = -k_1, k_4 = k_5 \} \cup \{k_2 = -k_3, k_4 = -k_1, k = k_5 \}.$
We take the first pairing in the union above as the representative of this class:
\begin{equs}
F_6 = T_1 \cdot T_6, \quad T_1 = \begin{tikzpicture}[scale=0.2,baseline=-5]
\coordinate (root) at (0,1);
\coordinate (tri) at (0,-1);
\draw[kernels2] (tri) -- (root);
\node[var] (rootnode) at (root) {\tiny{$ k $}};
\node[not] (trinode) at (tri) {};
\end{tikzpicture}, \quad T_6 = \begin{tikzpicture}[scale=0.2,baseline=-5]
\coordinate (root) at (0,2);
\coordinate (tri) at (0,0);
\coordinate (tri2) at (0,-2);
\coordinate (t1) at (-2,4);
\coordinate (t2) at (2,4);
\coordinate (t3) at (0,4);
\coordinate (t4) at (0,6);
\coordinate (t41) at (-2,8);
\coordinate (t42) at (2,8);
\coordinate (t43) at (0,10);
\draw[kernels2] (tri2) -- (tri);
\draw[kernels2,tinydots] (t1) -- (root);
\draw[kernels2] (t2) -- (root);
\draw[kernels2] (t3) -- (root);
\draw[symbols] (root) -- (tri);
\draw[symbols] (t3) -- (t4);
\draw[kernels2,tinydots] (t4) -- (t41);
\draw[kernels2] (t4) -- (t42);
\draw[kernels2] (t4) -- (t43);
\node[not] (rootnode) at (tri2) {};
\node[not] (rootnode) at (root) {};
\node[not] (rootnode) at (t4) {};
\node[not] (rootnode) at (t3) {};
\node[not] (trinode) at (tri) {};
\node[var] (rootnode) at (t1) {\tiny{$k_{\tiny{4}} $}};
\node[var] (rootnode) at (t41) {\tiny{$ \bar{k }$}};
\node[var] (rootnode) at (t42) {\tiny{$ \bar{k}_{\tiny{2}} $}};
\node[var] (rootnode) at (t43) {\tiny{$ k_{\tiny{2}} $}};
\node[var] (trinode) at (t2) {\tiny{$ k_{\tiny{4}} $}};
\end{tikzpicture}.
\end{equs}
\item Fourthly, we are left with the case where the internal pairings are of the same sign.  
The only pairing which corresponds to this case is given by: $\{k_2 = -k_3, k_1 = k_5, k=-k_4 \}$ .
\begin{equs}
F_7 =  T_1 \cdot T_7, \quad T_1 = \begin{tikzpicture}[scale=0.2,baseline=-5]
\coordinate (root) at (0,1);
\coordinate (tri) at (0,-1);
\draw[kernels2] (tri) -- (root);
\node[var] (rootnode) at (root) {\tiny{$ k $}};
\node[not] (trinode) at (tri) {};
\end{tikzpicture}, \quad T_7 = \begin{tikzpicture}[scale=0.2,baseline=-5]
\coordinate (root) at (0,2);
\coordinate (tri) at (0,0);
\coordinate (tri2) at (0,-2);
\coordinate (t1) at (-2,4);
\coordinate (t2) at (2,4);
\coordinate (t3) at (0,4);
\coordinate (t4) at (0,6);
\coordinate (t41) at (-2,8);
\coordinate (t42) at (2,8);
\coordinate (t43) at (0,10);
\draw[kernels2] (tri2) -- (tri);
\draw[kernels2,tinydots] (t1) -- (root);
\draw[kernels2] (t2) -- (root);
\draw[kernels2] (t3) -- (root);
\draw[symbols] (root) -- (tri);
\draw[symbols] (t3) -- (t4);
\draw[kernels2,tinydots] (t4) -- (t41);
\draw[kernels2] (t4) -- (t42);
\draw[kernels2] (t4) -- (t43);
\node[not] (rootnode) at (tri2) {};
\node[not] (rootnode) at (root) {};
\node[not] (rootnode) at (t4) {};
\node[not] (rootnode) at (t3) {};
\node[not] (trinode) at (tri) {};
\node[var] (rootnode) at (t1) {\tiny{$\bar{k} $}};
\node[var] (rootnode) at (t41) {\tiny{$ k_{\tiny{5}}$}};
\node[var] (rootnode) at (t42) {\tiny{$ \bar{k}_{\tiny{2}} $}};
\node[var] (rootnode) at (t43) {\tiny{$ k_{\tiny{2}} $}};
\node[var] (trinode) at (t2) {\tiny{$ k_{\tiny{5}} $}};
\end{tikzpicture}.
\end{equs}
\end{itemize}
Lastly, going back to point (a.) in the above we are left to consider pairings which are all external, meaning that there are no pairings of two frequencies made internally within the same integral. Hence, no simplification will be made at the level of the resonance structure of each of the integrals. This defines the last class to be considered. There are six possibilities of such pairings: $\{ k = -k_1, k_2 = k_4, k_3 = - k_5 \} \cup \{k = -k_1, k_2 = -k_5, k_3 = k_4 \} \cup \{k = k_2, k_1 = -k_4, k_3 = -k_5 \} \cup \{k = k_2, k_1 = k_5, k_3 = k_4 \} \cup \{k=k_3, k_2 = -k_5, k_1 = -k_4 \} \cup \{k=k_3, k_2 = k_4, k_1 = k_5 \}$. We take the first pairing in the union above as the representative of this class:
\begin{equs}
F_8 =  T_1 \cdot T_8, \quad T_1 = \begin{tikzpicture}[scale=0.2,baseline=-5]
\coordinate (root) at (0,1);
\coordinate (tri) at (0,-1);
\draw[kernels2] (tri) -- (root);
\node[var] (rootnode) at (root) {\tiny{$ k $}};
\node[not] (trinode) at (tri) {};
\end{tikzpicture}, \quad T_8 = \begin{tikzpicture}[scale=0.2,baseline=-5]
\coordinate (root) at (0,2);
\coordinate (tri) at (0,0);
\coordinate (tri2) at (0,-2);
\coordinate (t1) at (-2,4);
\coordinate (t2) at (2,4);
\coordinate (t3) at (0,4);
\coordinate (t4) at (0,6);
\coordinate (t41) at (-2,8);
\coordinate (t42) at (2,8);
\coordinate (t43) at (0,10);
\draw[kernels2] (tri2) -- (tri);
\draw[kernels2,tinydots] (t1) -- (root);
\draw[kernels2] (t2) -- (root);
\draw[kernels2] (t3) -- (root);
\draw[symbols] (root) -- (tri);
\draw[symbols] (t3) -- (t4);
\draw[kernels2,tinydots] (t4) -- (t41);
\draw[kernels2] (t4) -- (t42);
\draw[kernels2] (t4) -- (t43);
\node[not] (rootnode) at (tri2) {};
\node[not] (rootnode) at (root) {};
\node[not] (rootnode) at (t4) {};
\node[not] (rootnode) at (t3) {};
\node[not] (trinode) at (tri) {};
\node[var] (rootnode) at (t1) {\tiny{$k_{\tiny{2}}$}};
\node[var] (rootnode) at (t41) {\tiny{$ \bar{k}$}};
\node[var] (rootnode) at (t42) {\tiny{$ \bar{k}_{\tiny{4}} $}};
\node[var] (rootnode) at (t43) {\tiny{$ k_{\tiny{2}} $}};
\node[var] (trinode) at (t2) {\tiny{$ k_{\tiny{4}} $}};
\end{tikzpicture}.
\end{equs}

Analogously, one can make the analysis for the pairings of $\overline{u^0}$ and $u^{2,2}$ to obtain that five representatives are given by: $F_i =  T_1 \cdot T_i,$ for $i\in \{9,...,13 \} $, where
\begin{equs}
T_9 & = \begin{tikzpicture}[scale=0.2,baseline=-5]
\coordinate (root) at (0,2);
\coordinate (tri) at (0,0);
\coordinate (tri2) at (0,-2);
\coordinate (t1) at (-2,4);
\coordinate (t2) at (2,4);
\coordinate (t3) at (0,4);
\coordinate (t4) at (0,6);
\coordinate (t41) at (-2,8);
\coordinate (t42) at (2,8);
\coordinate (t43) at (0,10);
\draw[kernels2] (tri2) -- (tri);
\draw[kernels2] (t1) -- (root);
\draw[kernels2] (t2) -- (root);
\draw[kernels2,tinydots] (t3) -- (root);
\draw[symbols] (root) -- (tri);
\draw[symbols,tinydots] (t3) -- (t4);
\draw[kernels2] (t4) -- (t41);
\draw[kernels2,tinydots] (t4) -- (t42);
\draw[kernels2,tinydots] (t4) -- (t43);
\node[not] (rootnode) at (tri2) {};
\node[not] (rootnode) at (root) {};
\node[not] (rootnode) at (t4) {};
\node[not] (rootnode) at (t3) {};
\node[not] (trinode) at (tri) {};
\node[var] (rootnode) at (t1) {\tiny{$ k $}};
\node[var] (rootnode) at (t41) {\tiny{$ k_{\tiny{2}} $}};
\node[var] (rootnode) at (t42) {\tiny{$ k_{\tiny{4}} $}};
\node[var] (rootnode) at (t43) {\tiny{$ k_{\tiny{2}} $}};
\node[var] (trinode) at (t2) {\tiny{$ k_4 $}};
\end{tikzpicture}, \quad T_{10}=
\begin{tikzpicture}[scale=0.2,baseline=-5]
\coordinate (root) at (0,2);
\coordinate (tri) at (0,0);
\coordinate (tri2) at (0,-2);
\coordinate (t1) at (-2,4);
\coordinate (t2) at (2,4);
\coordinate (t3) at (0,4);
\coordinate (t4) at (0,6);
\coordinate (t41) at (-2,8);
\coordinate (t42) at (2,8);
\coordinate (t43) at (0,10);
\draw[kernels2] (tri2) -- (tri);
\draw[kernels2] (t1) -- (root);
\draw[kernels2] (t2) -- (root);
\draw[kernels2,tinydots] (t3) -- (root);
\draw[symbols] (root) -- (tri);
\draw[symbols,tinydots] (t3) -- (t4);
\draw[kernels2] (t4) -- (t41);
\draw[kernels2,tinydots] (t4) -- (t42);
\draw[kernels2,tinydots] (t4) -- (t43);
\node[not] (rootnode) at (tri2) {};
\node[not] (rootnode) at (root) {};
\node[not] (rootnode) at (t4) {};
\node[not] (rootnode) at (t3) {};
\node[not] (trinode) at (tri) {};
\node[var] (rootnode) at (t1) {\tiny{$ k_{\tiny{4}} $}};
\node[var] (rootnode) at (t41) {\tiny{$ k_{\tiny{2}} $}};
\node[var] (rootnode) at (t42) {\tiny{$ \bar{k}  $}};
\node[var] (rootnode) at (t43) {\tiny{$ k_{\tiny{2}} $}};
\node[var] (trinode) at (t2) {\tiny{$ \bar{k}_4 $}};
\end{tikzpicture},\quad T_{11} = 
\begin{tikzpicture}[scale=0.2,baseline=-5]
\coordinate (root) at (0,2);
\coordinate (tri) at (0,0);
\coordinate (tri2) at (0,-2);
\coordinate (t1) at (-2,4);
\coordinate (t2) at (2,4);
\coordinate (t3) at (0,4);
\coordinate (t4) at (0,6);
\coordinate (t41) at (-2,8);
\coordinate (t42) at (2,8);
\coordinate (t43) at (0,10);
\draw[kernels2] (tri2) -- (tri);
\draw[kernels2] (t1) -- (root);
\draw[kernels2] (t2) -- (root);
\draw[kernels2,tinydots] (t3) -- (root);
\draw[symbols] (root) -- (tri);
\draw[symbols,tinydots] (t3) -- (t4);
\draw[kernels2] (t4) -- (t41);
\draw[kernels2,tinydots] (t4) -- (t42);
\draw[kernels2,tinydots] (t4) -- (t43);
\node[not] (rootnode) at (tri2) {};
\node[not] (rootnode) at (root) {};
\node[not] (rootnode) at (t4) {};
\node[not] (rootnode) at (t3) {};
\node[not] (trinode) at (tri) {};
\node[var] (rootnode) at (t1) {\tiny{$ k_{\tiny{4}} $}};
\node[var] (rootnode) at (t41) {\tiny{$ \bar{k}_{\tiny{4}} $}};
\node[var] (rootnode) at (t42) {\tiny{$ \bar{k}_{\tiny{2}}  $}};
\node[var] (rootnode) at (t43) {\tiny{$ k_{\tiny{2}} $}};
\node[var] (trinode) at (t2) {\tiny{$k $}};
\end{tikzpicture}, \\ 
\quad T_{12} & = 
\begin{tikzpicture}[scale=0.2,baseline=-5]
\coordinate (root) at (0,2);
\coordinate (tri) at (0,0);
\coordinate (tri2) at (0,-2);
\coordinate (t1) at (-2,4);
\coordinate (t2) at (2,4);
\coordinate (t3) at (0,4);
\coordinate (t4) at (0,6);
\coordinate (t41) at (-2,8);
\coordinate (t42) at (2,8);
\coordinate (t43) at (0,10);
\draw[kernels2] (tri2) -- (tri);
\draw[kernels2] (t1) -- (root);
\draw[kernels2] (t2) -- (root);
\draw[kernels2,tinydots] (t3) -- (root);
\draw[symbols] (root) -- (tri);
\draw[symbols,tinydots] (t3) -- (t4);
\draw[kernels2] (t4) -- (t41);
\draw[kernels2,tinydots] (t4) -- (t42);
\draw[kernels2,tinydots] (t4) -- (t43);
\node[not] (rootnode) at (tri2) {};
\node[not] (rootnode) at (root) {};
\node[not] (rootnode) at (t4) {};
\node[not] (rootnode) at (t3) {};
\node[not] (trinode) at (tri) {};
\node[var] (rootnode) at (t1) {\tiny{$ k_{\tiny{4}} $}};
\node[var] (rootnode) at (t41) {\tiny{$ k $}};
\node[var] (rootnode) at (t42) {\tiny{$ \bar{k}_{\tiny{2}}  $}};
\node[var] (rootnode) at (t43) {\tiny{$ k_{\tiny{2}} $}};
\node[var] (trinode) at (t2) {\tiny{$\bar{k}_{\tiny{4}} $}};
\end{tikzpicture},
\quad T_{13} = 
\begin{tikzpicture}[scale=0.2,baseline=-5]
\coordinate (root) at (0,2);
\coordinate (tri) at (0,0);
\coordinate (tri2) at (0,-2);
\coordinate (t1) at (-2,4);
\coordinate (t2) at (2,4);
\coordinate (t3) at (0,4);
\coordinate (t4) at (0,6);
\coordinate (t41) at (-2,8);
\coordinate (t42) at (2,8);
\coordinate (t43) at (0,10);
\draw[kernels2] (tri2) -- (tri);
\draw[kernels2] (t1) -- (root);
\draw[kernels2] (t2) -- (root);
\draw[kernels2,tinydots] (t3) -- (root);
\draw[symbols] (root) -- (tri);
\draw[symbols,tinydots] (t3) -- (t4);
\draw[kernels2] (t4) -- (t41);
\draw[kernels2,tinydots] (t4) -- (t42);
\draw[kernels2,tinydots] (t4) -- (t43);
\node[not] (rootnode) at (tri2) {};
\node[not] (rootnode) at (root) {};
\node[not] (rootnode) at (t4) {};
\node[not] (rootnode) at (t3) {};
\node[not] (trinode) at (tri) {};
\node[var] (rootnode) at (t1) {\tiny{$ k_{\tiny{2}} $}};
\node[var] (rootnode) at (t41) {\tiny{$ k $}};
\node[var] (rootnode) at (t42) {\tiny{$ k_{\tiny{4}}  $}};
\node[var] (rootnode) at (t43) {\tiny{$ k_{\tiny{2}} $}};
\node[var] (trinode) at (t2) {\tiny{$k_{\tiny{4}} $}};
\end{tikzpicture}.
\end{equs}

Finally, also using the same case analysis as made previously we can construct the possible pairings between $\overline{u^1}$ and $u^1$ in order to compute the last term ${\mathbb{E}}(|u^1_k(\tau,v^{\eta})|^2)$. We have the same case figures as seen previously: six possible external pairings, four different possible pairings which are internal and of opposite signs, by symmetry four possibilities of pairings which are internal and one of opposite and the other of the same signs, and finally one possible internal pairing of the same signs. A representative of each of the above four class is given by: $F_i =  \tilde{T_i} \cdot T_i,$ for $i\in \{14,...,17 \} $, where
\begin{equs}
\tilde{T}_{14} &= \begin{tikzpicture}[scale=0.2,baseline=-5]
\coordinate (root) at (0,-1);
\coordinate (t3) at (0,1);
\coordinate (t4) at (0,3);
\coordinate (t41) at (-2,5);
\coordinate (t42) at (2,5);
\coordinate (t43) at (0,7);
\draw[kernels2] (t3) -- (root);
\draw[symbols] (t3) -- (t4);
\draw[kernels2,tinydots] (t4) -- (t41);
\draw[kernels2] (t4) -- (t42);
\draw[kernels2] (t4) -- (t43);
\node[not] (rootnode) at (root) {};
\node[not] (rootnode) at (t4) {};
\node[not] (rootnode) at (t3) {};
\node[var] (rootnode) at (t41) {\tiny{$ k_1 $}};
\node[var] (rootnode) at (t42) {\tiny{$ k $}};
\node[var] (rootnode) at (t43) {\tiny{$  k_{1} $}};
\end{tikzpicture}, \ \  T_{14} =  
 \begin{tikzpicture}[scale=0.2,baseline=-5]
\coordinate (root) at (0,-1);
\coordinate (t3) at (0,1);
\coordinate (t4) at (0,3);
\coordinate (t41) at (-2,5);
\coordinate (t42) at (2,5);
\coordinate (t43) at (0,7);
\draw[kernels2] (t3) -- (root);
\draw[symbols] (t3) -- (t4);
\draw[kernels2,tinydots] (t4) -- (t41);
\draw[kernels2] (t4) -- (t42);
\draw[kernels2] (t4) -- (t43);
\node[not] (rootnode) at (root) {};
\node[not] (rootnode) at (t4) {};
\node[not] (rootnode) at (t3) {};
\node[var] (rootnode) at (t41) {\tiny{$k_2 $}};
\node[var] (rootnode) at (t42) {\tiny{$ k $}};
\node[var] (rootnode) at (t43) {\tiny{$  k_{2} $}};
\end{tikzpicture} = T_3,\quad
\tilde{T}_{15} =  \begin{tikzpicture}[scale=0.2,baseline=-5]
\coordinate (root) at (0,-1);
\coordinate (t3) at (0,1);
\coordinate (t4) at (0,3);
\coordinate (t41) at (-2,5);
\coordinate (t42) at (2,5);
\coordinate (t43) at (0,7);
\draw[kernels2] (t3) -- (root);
\draw[symbols] (t3) -- (t4);
\draw[kernels2,tinydots] (t4) -- (t41);
\draw[kernels2] (t4) -- (t42);
\draw[kernels2] (t4) -- (t43);
\node[not] (rootnode) at (root) {};
\node[not] (rootnode) at (t4) {};
\node[not] (rootnode) at (t3) {};
\node[var] (rootnode) at (t41) {\tiny{$ k_1 $}};
\node[var] (rootnode) at (t42) {\tiny{$ \bar{k}  $}};
\node[var] (rootnode) at (t43) {\tiny{$ k_1  $}};
\end{tikzpicture} , \
 T_{15} = 
\begin{tikzpicture}[scale=0.2,baseline=-5]
\coordinate (root) at (0,-1);
\coordinate (t3) at (0,1);
\coordinate (t4) at (0,3);
\coordinate (t41) at (-2,5);
\coordinate (t42) at (2,5);
\coordinate (t43) at (0,7);
\draw[kernels2] (t3) -- (root);
\draw[symbols] (t3) -- (t4);
\draw[kernels2,tinydots] (t4) -- (t41);
\draw[kernels2] (t4) -- (t42);
\draw[kernels2] (t4) -- (t43);
\node[not] (rootnode) at (root) {};
\node[not] (rootnode) at (t4) {};
\node[not] (rootnode) at (t3) {};
\node[var] (rootnode) at (t41) {\tiny{$k $}};
\node[var] (rootnode) at (t42) {\tiny{$ \bar{k}_2  $}};
\node[var] (rootnode) at (t43) {\tiny{$ k_2  $}};
\end{tikzpicture} = T_2, \\
& \tilde{T}_{16} =  \begin{tikzpicture}[scale=0.2,baseline=-5]
\coordinate (root) at (0,-1);
\coordinate (t3) at (0,1);
\coordinate (t4) at (0,3);
\coordinate (t41) at (-2,5);
\coordinate (t42) at (2,5);
\coordinate (t43) at (0,7);
\draw[kernels2] (t3) -- (root);
\draw[symbols] (t3) -- (t4);
\draw[kernels2,tinydots] (t4) -- (t41);
\draw[kernels2] (t4) -- (t42);
\draw[kernels2] (t4) -- (t43);
\node[not] (rootnode) at (root) {};
\node[not] (rootnode) at (t4) {};
\node[not] (rootnode) at (t3) {};
\node[var] (rootnode) at (t41) {\tiny{$ k $}};
\node[var] (rootnode) at (t42) {\tiny{$ k_1 $}};
\node[var] (rootnode) at (t43) {\tiny{$  \bar{k}_{1} $}};
\end{tikzpicture} , \  T_{16}  =
 \begin{tikzpicture}[scale=0.2,baseline=-5]
\coordinate (root) at (0,-1);
\coordinate (t3) at (0,1);
\coordinate (t4) at (0,3);
\coordinate (t41) at (-2,5);
\coordinate (t42) at (2,5);
\coordinate (t43) at (0,7);
\draw[kernels2] (t3) -- (root);
\draw[symbols] (t3) -- (t4);
\draw[kernels2,tinydots] (t4) -- (t41);
\draw[kernels2] (t4) -- (t42);
\draw[kernels2] (t4) -- (t43);
\node[not] (rootnode) at (root) {};
\node[not] (rootnode) at (t4) {};
\node[not] (rootnode) at (t3) {};
\node[var] (rootnode) at (t41) {\tiny{$k $}};
\node[var] (rootnode) at (t42) {\tiny{$ \bar{k}_2 $}};
\node[var] (rootnode) at (t43) {\tiny{$  k_{2} $}};
\end{tikzpicture} = T_{15}, 
\quad
\tilde{T}_{17} =  \begin{tikzpicture}[scale=0.2,baseline=-5]
\coordinate (root) at (0,-1);
\coordinate (t3) at (0,1);
\coordinate (t4) at (0,3);
\coordinate (t41) at (-2,5);
\coordinate (t42) at (2,5);
\coordinate (t43) at (0,7);
\draw[kernels2] (t3) -- (root);
\draw[symbols] (t3) -- (t4);
\draw[kernels2,tinydots] (t4) -- (t41);
\draw[kernels2] (t4) -- (t42);
\draw[kernels2] (t4) -- (t43);
\node[not] (rootnode) at (root) {};
\node[not] (rootnode) at (t4) {};
\node[not] (rootnode) at (t3) {};
\node[var] (rootnode) at (t41) {\tiny{$ k_1 $}};
\node[var] (rootnode) at (t42) {\tiny{$ k_3 $}};
\node[var] (rootnode) at (t43) {\tiny{$  k_{2} $}};
\end{tikzpicture} =  T_{17}.
%
\end{equs}
We denote by $m_{F_i}, \ i\in [4,17]$ the multiplicative constant representing the number of elements in the class who's representative is $F_i$. Namely, it follows from the above analysis that 
\begin{equs}\label{m_F} \begin{aligned}
&m_{F_i} = 6, \ \text{for} \ i \in \{8,13, 17 \}, \\
&m_{F_i} = 2, \ \text{for} \  i \in \{5, 10,6,11 \}, \\
&m_{F_i} = 1, \ \text{for} \ i \in \{7,12,16 \}, \\
&m_{F_{i}} = 4, \ \text{for} \ i \in \{4,9,14,15 \}.
\end{aligned}
\end{equs}

In conclusion, we have that the set of representatives of approximated paired forests is given by: $\CG^{1,k}_0(R) = \{F_i \}_{i= \{1,...,17\}}$.
We now present a second order low-regularity approximation to the second order moments of $u_k(\tau,v^{\eta})$, solution of \eqref{NLS} with initial data \eqref{ic}.
\begin{corollary}\label{cor_schemeNLS2}
At second order our general low regularity scheme \eqref{schema_general} takes the form:
\begin{equs}\label{schemeNLS2_Fourier}
\begin{aligned}
V_{k}^{2,1}(\tau,v) &= v_k\overline{v_k}
+ 2\tau^2 v_k \overline{v_k}\left( |k|^2 \sum_{k_2 \in \Z^d}|v_{k_2}|^2 - \sum_{k_2 \in \Z^d}|k_2|^2 |v_{k_2}|^2\right) \\
&-6\tau^2\Big( 6 v_k \overline{v_k} \big(\sum_{k_2 \in \Z^d} |v_{k_2}|^2\big)^2 - \sum\limits_{\substack{k_1,k_2,k_3 \in \Z^d\\ -k_1 + k_2 + k_3 = k}} |v_{k_1}|^2|v_{k_2}|^2|v_{k_3}|^2 \Big). \qquad
\end{aligned}
\end{equs}
The scheme \eqref{schemeNLS2_Fourier} is locally of order $\CO(\tau^3|\nabla|^2v)$. The above is the Fourier coefficient associated to the following scheme written in physical space: 
\begin{equs}\label{schemeNLS2} \begin{aligned}
u^{\ell+1} &= u^\ell * \tilde{u}^\ell 
 + 2\tau^2\left((\nabla u^\ell * \nabla \tilde{u}^\ell)\|u^\ell\|_{L^2}^2\right. \\
&\left. \quad - (u^\ell * \tilde{u}^\ell) \|\nabla u^\ell\|_{L^2}^2\right)- 6\tau^2\left( 6(u^\ell*\tilde{u}^\ell)\|u^\ell\|_{L^2}^4 - (u^\ell*\tilde{u}^\ell)^3\right),
\end{aligned}
\end{equs}
where $ \tilde{u}^{\ell}(x) = \overline{u^{\ell}(-x)} $. 
\end{corollary}

\begin{remark}[Stabilisation technique]\label{rem:stabNLS}
We note that in contrast to the first-order scheme \eqref{schemeNLS1} the  second-order scheme \eqref{schemeNLS2} involves spatial derivatives on the numerical solution such as for instance in the term $\nabla u^{\ell} * \nabla \tilde{u}^{\ell}$. These derivatives cause instability in the discretisation. For practical implementations, one needs to stabilize the scheme in order for it to converge.
Different approaches can be taken to stabilize the above scheme, which do not require a Courant-Friedrichs-Lewy (CFL) type condition on the step sizes. Here we propose an approach based on the a posteriori inclusion of well chosen filter functions. We refer to the works of \cite{BBS, AB} for more details regarding these questions of stability and of the suitable choice of filter functions used to yield stable schemes with optimal local error. We also refer to \cite{HLW} for a general introduction to filter function in the case of oscillatory ordinary differential equations.

In this work we introduce the following filter function 
$$
\mathrm{sinc}^2\left(\tau^{\frac{1}{2}} \left \vert \nabla\right \vert  \right) 
$$

which in in Fourier space takes the form,
\begin{equation}\label{psi_filter}
\Psi(i\tau |k|^2) =  \mathrm{sinc}^2\left(\tau^{\frac{1}{2}}|k|  \right) = 
 \frac{1}{(i\tau^{\frac{1}{2}}|k| )^2}\left(e^{i\tau^{\frac{1}{2}}|k|/2}-e^{-i\tau^{\frac{1}{2}}|k|/2}\right)^2.
\end{equation}
In order to stabilise our scheme \eqref{schemeNLS2}  we pre-multiply both critical terms in the second line of \eqref{schemeNLS1_Fourier} with the filter function \eqref{psi_filter} in  the corresponding frequencies $k$ and $k_2$ respectively. This yields the following stabilised version of~\eqref{schemeNLS2_Fourier} 
\begin{equs}\label{schemeNLS2_FourierStab}
\begin{aligned}
& V_{k}^{2,1}(\tau,v) = v_k\overline{v_k} \\
&+ 2\tau^2 v_k \overline{v_k} \left(\Psi(i\tau |k|^2)  |k|^2 \sum_{k_2 \in \Z^d}|v_{k_2}|^2 - \sum_{k_2 \in \Z^d} \Psi(i\tau |k_2|^2)  |k_2|^2 |v_{k_2}|^2\right) + ...
\end{aligned} 
\end{equs}
We have that 
$$\Psi(i\tau |k|^2) = 1 + \mathcal{O}(\tau |k|^2)$$ and hence the stabilised scheme \eqref{schemeNLS2_FourierStab} preserves the low regularity error structure of $ \mathcal{O}(\tau |\nabla|^2 u)$.
This is essential for the local error analysis  of the scheme. Furthermore, thanks to the observation that
$$
 \left \vert \Psi(i\tau |k |^2)  \tau \right\vert k \vert^2 \vert \leq 1\quad \text{for all $k \in \Z^d$}
$$
$\Psi$ renders a stabilised version of the scheme \eqref{schemeNLS2}  in physical space given by:
\begin{equs}
u^{\ell+1} &= u^\ell * \tilde{u}^\ell \\
&-{2\tau}\left(((e^{i \tau^{\frac{1}{2}}|\nabla|/2}-e^{-i\tau^{\frac{1}{2}}|\nabla|/2}) u^\ell * (e^{i\tau^{\frac{1}{2}}|\nabla|/2}-e^{-i\tau^{\frac{1}{2}}|\nabla|/2}) \tilde{u}^\ell)\|u^\ell\|_{L^2}^2\right. \\
&\left. \quad - (u^\ell * \tilde{u}^{\ell}) \|(e^{i\tau^{\frac{1}{2}}|\nabla|/2}-e^{-i \tau^{\frac{1}{2}}|\nabla|/2}) u^\ell\|_{L^2}^2\right)\\
&- 6 \tau^2\left( 6 (u^\ell*\tilde{u}^\ell)\|u^\ell\|_{L^2}^4 - (u^\ell*\tilde{u}^\ell)^3\right).
\end{equs}
\end{remark}

\begin{remark}[Practical implementation]
In order to allow for a practical implementation of our schemes we want to be able to express the discretization both  in physical and Fourier space. This allows us to use the Fast Fourier Transform (FFT) whose computational cost is of order $O(|K|^d log(|K|^d))$, where $K$ denotes the highest frequency in the discretization and $d$ is the dimension. Namely, we compute the action of the filter functions on the solution  in frequency space, while computing the product of functions in physical space.
\end{remark}

\begin{remark} In \eqref{schemeNLS2}, the scheme $ V_k^{n,r} $ is written in physical space. In fact, one can wonder if it is possible to get a general statement such as \cite[Prop. 3.18]{BS} which shows that low regualrity schemes can always be rewritten in physical space. It is not clear how to prove such a statement in full generality when exact integrations are performed instead of a full Taylor expansion. Indeed, the core of the proof  \cite[Prop. 3.18]{BS} relies on \cite[Assumption 1]{BS} which is no longer true in our case: Frenquencies on the leaves are no longer disjoint.
	We will have to be more cautious in the scheme and probably Taylor expand a bit more as exact integrations will not allow to move back to physical space. It is a challenging open question. In the case of the present work, we are mostly doing Taylor expansions. They guarantee that we are able to go back to physical space.
	\end{remark}

\begin{proof}
We are considering a second order scheme and hence, we enter the case where $r=1$. Furthermore, we ask for two derivatives on the initial data, and hence we take the regularity parameter $n$ to take on the value of $2$.
We read from equation \eqref{schema_general} that the second-order scheme has the general form,
\begin{equs}\label{expansionNLS2} \begin{aligned}
V_{k}^{2,1}(\tau,v) = \sum_{F = T_1\cdot T_2 \in \CG^{1,k}_{0}(R)} m_F &\frac{\bar{\Upsilon}^{p}(T_1)(v) \, \Upsilon^{p}(T_2)(v)}{S(T) S(T_2)}  \\
&\mathcal{Q}_{\le 2} \left((\bar{\Pi}^{2,1}(T_1))(\Pi^{2,1}(T_2)) \right)(\tau),
\end{aligned}
\end{equs}
where $\{m_F\}_{F\in \CG^{1,k}_0(R)}$ are given in \eqref{m_F}, and where in the same spirit of \eqref{upsilon1} we have that for all $i\in \{4,...,13 \}$, 
\begin{align*}
\Upsilon^{p}(T_4)(v) &= 2 \Upsilon^{p}(\CI_{{(\mathfrak{t}_1,1)}}(\lambda_{k_4}))(v)\Upsilon^p( T_3)(v) \Upsilon^p(\CI_{(\mathfrak{t}_1,0)}(\lambda_{k_4}) )(v) \\
&= 2\overline{v_{k_4}}(2\overline{v_{k_2}}v_{k_2}v_k)v_{k_4}\\
&= 4 |v_{k_4}|^2 |v_{k_{2}}|^2v_k\\
&=\Upsilon^p(T_i)(v),
\end{align*}
and for $i\in \{4,..., 8 \}, \  S(T_i) = 2$ while for $i \in \{9, ..., 13 \}, \ S(T_i) = 4$.
Furthermore, we have that for $i\in\{14,...,17 \}$,
\begin{align*}
\bar{\Upsilon}^p(\tilde{T_i})(v)\Upsilon^p(T_i)(v)  = 4 |v_{k_1}|^2|v_{k_2}|^2|v_{k_3}|^2, \quad S(T_i)^2 = 4.
\end{align*}
Hence, the second order scheme \eqref{expansionNLS2} takes the form
\begin{equs}\label{expansionNLS1_2}
\begin{aligned}
& V_{k}^{2,1}(\tau,v) =  v_k\overline{v_k} 
\\ & + 2Re\left(  v_k\overline{v_k}\sum_{k_2\in \Z^d} |v_{k_{2}}|^2 \CQ_{\le 2} \bigg(\bar{\Pi}^{2,1}({T_1})\big(\Pi^{2,1}(T_2) + 2\Pi^{2,1}(T_3) \big)\bigg)(\tau)\right) \\ 
& + 2Re\left( v_k\overline{v_k} \sum_{k_2, k_4 \in \Z^d} |v_{k_2}|^2|v_{k_4}|^2 \right.\\
 &\left.\bigg(\sum_{4\le i \le 13}m_{F_i}\CQ_{\le 2}\left( \bar{\Pi}^{2,1}({T_1})\Pi^{2,1}(T_i)\right)\bigg)(\tau)\right)\\
& + v_k \overline{v_k} \sum_{k_1,k_2 \in \Z^d} |v_{k_1}|^2|v_{k_2}|^2 \bigg(\sum_{14\le i \le 16}m_{F_i}\CQ_{\le 2}\left( \bar{\Pi}^{2,1}(\tilde{T_i})\Pi^{2,1}(T_i)\right)\bigg)(\tau) \\
& + \sum\limits_{\substack{k_1,k_2,k_3 \in \Z^d\\ -k_1 + k_2 + k_3 = k}} |v_{k_1}|^2|v_{k_2}|^2|v_{k_3}|^2 \bigg(m_{F_{17}}\CQ_{\le 2}\left( \bar{\Pi}^{2,1}({T_{17}})\Pi^{2,1}(T_{17})\right)\bigg)(\tau). 
\end{aligned}
\end{equs}
\\
1. Computation of $\CQ_{\le 2}\left(\bar{\Pi}^{2,1}({T_1})\big(\Pi^{2,1}(T_2) + 2\Pi^{2,1}(T_3)\right)$:
We apply Definition \ref{recursive_pi_r} to each of the three decorated trees.
\begin{itemize}
\item Computation of $\bar{\Pi}^{2,1}(T_1)$: Given that \eqref{NLS_T1} is a zero-th order term, as in \eqref{compT1_NLS} we have,
$$
( \bar{\Pi}^{2,1}{T_1})(\tau) = e^{i\tau k^2}.
$$
\item Computation of $\Pi^{2,1}(T_2)$: It follows from the definition \eqref{NLS_T2} and \eqref{NLS_T2sym} of $T_2$ that,
\begin{equs}\label{comp_T2_o2} \begin{aligned}
(\Pi^{2,1} T_2)(\tau) &= e^{-i \tau k^2}  \mathcal{K}_{o_2}^{k,1} \left((\Pi^{2,0}\CI_{\bar{o}_1}(\lambda_{-k}))(\xi) \right. \\& \left. (\Pi^{2,0} \CI_{o_1}(\lambda_{-k_2})(\xi) (\Pi^{2,0}\CI_{o_1}(\lambda_{k_2}))(\xi) \right)(\tau)\\
&= e^{-i \tau k^2}  \mathcal{K}_{o_2}^{k,1}( e^{i\xi(k^2 - 2k_2^2)} ,2)(\tau)\\ 
& = -i\tau e^{-i \tau k^2} + \tau^2 e^{-i \tau k^2}(k^2 - k_2^2),
\end{aligned}
\end{equs}
where we applied Definition \ref{CK} with $\ell=1$. Namely, up to one additional order then made in \eqref{NLS_compT2} for the first order analysis.
\item Computation of $\Pi^{2,1}(T_3)$: As for the calculations made in \eqref{NLS_compT3}, given that the resonance structure \eqref{ZeroRes} of the integrand in $\Pi(T_3)$ is zero it follows that,
\begin{equs}\label{comp_T3_o2}
(\Pi^{2,1} T_3)(\tau) = -i\tau e^{-i \tau k^2}.
\end{equs}
\end{itemize}
Hence, having only at most second order terms in the above calculations, $\CQ_{\le 2}$ does not play a role and we have:
\begin{equs}
\CQ_{\le 2}\left(\bar{\Pi}^{2,1}(T_1)\big(\Pi^{2,1}(T_2) + 2\Pi^{2,1}(T_3)\right)(\tau) & = \bar{\Pi}^{2,1}(T_1)\big(\Pi^{2,1}(T_2) + 2\Pi^{2,1}(T_3))(\tau) \\
&= -3i\tau + \tau^2(k^2 - k_2^2).
\end{equs}
Given that we take the real part of the above approximation, only the second term in the above will contribute to the scheme.\\
2. Computation of $\CQ_{\le 2}\left( \bar{\Pi}^{2,1}({T_1})\Pi^{2,1}(T_i)\right)$, for $i\in \{4,...,13\}$:
\begin{itemize}
\item Computation of $\left( \bar{\Pi}^{2,1}({T_1})\Pi^{2,1}(T_4)\right)$: We recall that $T_4$ is given by \eqref{NLS_T4} and in symbolic notation we have,
\begin{equs}
T_4 = \CI_{o_1}\Big(\lambda_{k}\Big( \CI_{o_2}\left(\lambda_k(\CI_{\bar{o_1}}(\lambda_{k_4})T_3\CI_{o_1}(\lambda_{k_4}))\right) \Big) \Big),
\end{equs}
where $T_3$ is given in \eqref{NLS_T3}.
Hence, it follows that
\begin{equs}
(\Pi^{2,1} T_4)(\tau) &= e^{-i \tau k^2}\mathcal{K}_{o_2}^{k,1}\left( (\Pi^{2,0} \CI_{\bar{o}_1}(\lambda_{k_4}))(\xi)(\Pi^{2,0}T_3)(\xi)(\Pi^{2,0}\CI_{o_1}(\lambda_{k_4}))(\xi) \right)(\tau)\\ 
&= e^{-i \tau k^2}\mathcal{K}_{o_2}^{k,1}\left( -i\xi e^{i\xi (k_{4}^2 - k^2 - k_{4}^2)} \right)(\tau)\\
&= -ie^{-i \tau k^2}  \mathcal{K}_{o_2}^{k,1}( \xi e^{-i\xi k^2} ,2)(\tau)\\
&= -\frac{\tau^2}{2} e^{-i \tau k^2}, 
\end{equs}
where in the second line we used that $(\Pi^{2,0}T_3)(\xi) = -i\xi e^{-i\xi k^2}$, which comes from \eqref{NLS_compT3}, obtained during the first order analysis. We note that when calculating $\CK$ in the last line, the resonance structure is as expected equal to zero.
\item Computation of $\Pi^{2,1}(T_i)$, for $i \in \{5,...,13\}$: These computations follow exactly the same line as in the above case. We have,
$( \Pi^{2,1} T_i)(\tau)= -\frac{\tau^2}{2} e^{-i \tau k^2}$.

The only difference lies in the expression of the resonance factor which will either be of the form \eqref{NonZeroRes}, be equal to zero (as was the case in the above), or finally in the case of an external pairing ($\Pi(T_i)$, $i \in \{ 8,13\}$) of the form \eqref{resDecomp}. Nevertheless, given that we can make the full Taylor approximation \eqref{CK} when approximating the each integral, the approximation of $\Pi^{2,1}(T_i)$, $i \in \{4,...,13\}$ are the same, the difference lies in the local error produced.
\end{itemize}
Therefore, $\CQ_{\le 2}\left( \bar{\Pi}^{2,1}({T_1})\Pi^{2,1}(T_i)\right)(\tau) = -\frac{\tau ^2}{2}$.\\
3. Computation of $\CQ_{\le 2}\left( \bar{\Pi}^{2,1}({T_i})\Pi^{2,1}(T_i)\right)(\tau)$, for $i\in \{14,...,17\}$: 
\begin{itemize}
\item Computation of $\bar{\Pi}^{2,1}({T_{14}})\Pi^{2,1}(T_{14})$. Given that $T_{14}$ has the same null resonance structure as $T_3$, it follows as in \eqref{comp_T3_o2} that $(\Pi^{2,1} T_{14})(\tau) = -i\tau e^{-i \tau k^2}$ and hence,
$$(\bar{\Pi}^{2,1} {T_{14}} \Pi^{2,1} T_{14})(\tau) = \tau^2.$$
For the remaining computations the truncation operator $\CQ_{\le 2}$ will be essential to incorporate the terms of correct order into the scheme. 
\item Computation of $\CQ_{r\le 2}\bar{\Pi}^{2,1}(\tilde{T}_{15})\Pi^{2,1}(T_{15})$, $\CQ_{r\le 2}\bar{\Pi}^{2,1}({\tilde{T}_{16}})\Pi^{2,1}(T_{16})$, \\
 and $\CQ_{r\le 2}\bar{\Pi}^{2,1}(\tilde{T}_{17})\Pi^{2,1}(T_{17})$. First, we have that $\Pi(T_{15}) $ and $\Pi(T_{16})$ have the same resonance structure as $\Pi T_2$, and hence following the calculations in \eqref{comp_T2_o2} it follows that, 
\begin{equs}
(\Pi T_{15})(\tau) = (\Pi T_{16})(\tau) = -i\tau e^{-i \tau k^2} + \tau^2 e^{-i \tau k^2}(k^2 - k_2^2).
\end{equs}
 Furthermore, we have that $\bar{\Pi} (\tilde{T}_{15})$ has a null resonance structure and hence it follows that $(\bar{\Pi} \tilde{T}_{15})(\tau) = i\tau e^{i\tau k^2}$ and,
\begin{equs}
\CQ_{r\le 2} (\bar{\Pi}^{2,1} \tilde{T}_{15} \Pi^{2,1} T_{15})(\tau) &= \CQ_{r\le 2} \left( \tau^2 + i\tau^3(k^2 - k_2^2) \right)\\
&= \tau^2.
\end{equs}
Similarly, we have
\begin{equs}
\CQ_{r\le 2}( \bar{\Pi}^{2,1} \tilde{T}_{16}{\Pi}^{2,1} T_{16})(\tau) = \tau^2 = \CQ_{r\le 2}( \bar{\Pi}^{2,1} \tilde{T}_{17}\Pi^{2,1} T_{17})(\tau),
\end{equs}
 where the third and fourth order terms are truncated by the operator $\CQ_{\le 2}$.
\end{itemize}
Plugging the results obtained in the above computations into \eqref{expansionNLS1_2} yields the second order low regularity scheme \eqref{schemeNLS2_Fourier}, which is given in physical space by \eqref{schemeNLS2}.

{\bf Local Error:} It remains to show the claimed third order local error bound: $\CO(\tau^3|\nabla|^2 v)$. We follow along the same lines made for the first order local error analysis in Section \ref{NLS1}. By Theorem \ref{thm:genloc} we have
\begin{equs}
V^{2,1}_k(\tau,v) - V_k(\tau,v) = \sum_{T_1 \cdot T_2 \in \CG^{1,k}_0(R)}\CO(\tau^3\CL_{\text{\tiny{low}}}^1(T_1\cdot {T_2},2)\bar{\Upsilon}^p(T_1)(v)\Upsilon^p({T_2})(v)).
\end{equs}
We are left to calculate $\CL_{\text{\tiny{low}}}^1(T_1\cdot {T_2},2)$, for $T_1\cdot {T_2}\in \CG^{1,k}_0(R)$. We are interested in the integrals $\Pi T$, $\Pi \tilde{T}$ which have non-zero resonance structure, since oscillatory integrals who's resonance structure is zero are integrated exactly,
and hence do not contribute to the local error analysis of the scheme.
\begin{itemize}
\item Computation of $\CL_{\text{\tiny{low}}}^1({T_1}\cdot T_2,2)$, where $\Pi T_2$ has the non-zero resonance structure \eqref{NonZeroRes}. Following the analysis made in \eqref{L_low2} for this pairing we have,
\begin{equs}
R_{o_2,1}^{k,1}(\Pi^{n=2}G_2) = R_{o_2,2}^{k,1}(e^{i\xi(k^2 - 2k_1^2)}) = \left( 2(k^2 - k_1^2)\right)^2,
\end{equs}
where we applied definition \eqref{R_fullTaylor} up to an additional order. Therefore, it follows from the previously made steps \eqref{L_low2} that
\begin{equs}\label{localerr-NLS1}
\CL_{\text{\tiny{low}}}^1({T_1}\cdot {T_2},2)\bar{\Upsilon}^p({T_1})\Upsilon^p(T_2) = \CO(\sum_{k_1 \in \Z^d} (k^2 - k_1^2)^2\overline{v_{k}}v_{k}|v_{k_{1}}|^2).
\end{equs}
In physical space this yields an error of the form
\begin{equs}\label{order2worst}
\CO\Big(\tau^3 &\left((|\nabla|^2 v *|\nabla|^2 \tilde{v}) \|v\|_{L^2}^2 + \Big.\right.\\
&\left.\Big. (v*\tilde{v})\||\nabla|^2 v\|_{L^2}^2 + (|\nabla| v *|\nabla| \tilde{v}) \||\nabla|v\|_{L^2}^2 \right) \Big). 
\end{equs}
\end{itemize}
The computations for the remaining pairings $F\in \CG^{1,k}_0(R)$ can be made in an analogous fashion and can be shown to produce an error term of the form $O(\tau^3 |\nabla| v)$. In consequence, their approximation requires less regularity on the initial data than that for the above pairing ${T_1} \cdot T_2$. We conclude from the above expression \eqref{order2worst} that the local error in physical space is of the form $\CO (\tau^3 |\nabla|^2v)$, which requires two derivatives on the initial data. 
\end{proof}

\begin{remark}
Recently new low-regularity integrators (\cite{AB-sym}) and resonance-based discretisations (\cite{FMS, MS}) have been introduced, which  preserve better the underlying structure of the solution over long-times, and which exhibits improved error constants at low-regularity. Encapsulating the idea behind the construction of these low-regularity structure-preserving schemes into the general framework presented in this article would be of interest in the future. This could yield better insight on the long-time behaviour of the solution on a numerical level.
\end{remark}

\subsection{KdV}\label{sec:KDV}
Let us next consider the Korteweg--de Vries (KdV) equation
\begin{equs}\label{kdv}
\partial_t u + \partial_x^{3} u = \frac12 \partial_x u^2,\quad (t,x) \in \R \times  \T
\end{equs}
with a random initial value $u(0) =v$ of the form \eqref{ic}. 
The KdV equation \eqref{kdv} fits into the general framework \eqref{dis} with
\begin{equation*}\label{kgrDo}
\begin{aligned}
 \mathcal{L}\left(\nabla \right)  = i \partial_x^3, \quad \alpha = 1,   \quad \text{and}\quad   p(u,\overline u) =  p(u) =  i \frac12 u^2.
 \end{aligned}
\end{equation*} 
One has $ \CL = \lbrace \Labhom_1, \Labhom_2 \rbrace $, $ P_{\Labhom_1} = - \lambda^3 $ and $ P_{\Labhom_2} =  \lambda^3 $ . Then, we denoted by $ \<thick> $ an edge decorated by $ (\Labhom_1,0) $ and by $\<thin>$ an edge decorated by $ (\Labhom_2,0) $.
The general framework~\eqref{general_scheme_2} derived in Section \ref{sec::4} builds the foundation of the  first- and second-order schemes presented below for the KdV equation~\eqref{kdv}.

\begin{corollary}
\label{corKdV} For the KdV equation \eqref{kdv} the general scheme~ \eqref{genscheme} takes at first order the form
\begin{equation}
\begin{aligned}\label{schemeKdV1}
u^{\ell+1} &=  u^{\ell} * \tilde{u}^{\ell}
\end{aligned}
\end{equation}
with a local error  of order $\mathcal{O}\Big(
\tau^2 \partial_x u
\Big)$
and at  second-order
\begin{equation}
\begin{aligned}\label{schemeKdV}
u^{\ell+1} &=  u^{\ell} * \tilde{u}^\ell + \tau^2\partial_x^2 ( u^\ell * \tilde{u}^\ell)^2
\\&  + \frac{\tau^2}{2} \partial_x \left( u^{\ell} * \tilde{u}^{\ell} \right)  (\sum_{{k_1} \in \Z^d} k_1 |u^\ell_{k_1}|^2)   - \frac{\tau^2}{2} \left( \partial_x u^{\ell} * \partial_x \tilde{u}^{\ell} \right)  ((u^{\ell}_0)^2 + (\overline{u^{\ell}_0})^2) 
\end{aligned}
\end{equation}
with a local error  of order $\mathcal{O}\Big(
\tau^3 \partial_x^2 u
\Big)$, and where $ \tilde{u}^{\ell}(x) = \overline{u^{\ell}(-x)} $. 
\end{corollary}

\begin{remark}[Stability] In view of a practical implementation of the second order scheme \eqref{schemeKdV} for KdV, we  address the need to stabilize the three terms in the above scheme involving spatial derivatives on the numerical solution.
As in the Remark \ref{rem:stabNLS}, for the stabilisation of the second order scheme for NLS, we introduce three filter functions which we premultiply in front of each of the last three critical term appearing in the scheme \eqref{schemeKdV}. We introduce the following three filter functions in Fourier space,
\begin{align*}
&\Psi_1(k) = \varphi_1(-i\tau |k|^2) = \frac{1 - e^{-i\tau |k|^2}}{i\tau k^2}, \\
&\Psi_2(k,k_1) = 
 \frac{1}{(i\tau^{\frac{1}{2}}|k| )(i\tau^{\frac{1}{2}}|k_1| )}\left(e^{i\tau^{\frac{1}{2}}|k|/2}-e^{-i\tau^{\frac{1}{2}}|k|/2}\right)\left(e^{i\tau^{\frac{1}{2}}|k_1|/2}-e^{-i\tau^{\frac{1}{2}}|k_1|/2}\right) , \\
&\Psi_3(k) = \text{sinc}^2(\tau^{\frac{1}{2}}|k|), 
\end{align*}
where $\Psi_3(k)$ is the same filter function used for the stabilization of the second order scheme for NLS (see equation \eqref{psi_filter}). 
We have that
\begin{align*}
& \Psi_j(k) =1 + \CO(\tau |k|^2), \quad |\Psi_j (k) \tau |k|^2|\le 1, \ j \in \{1,3 \} , \ k \in \Z, \\
& \Psi_2(k, k_1) = 1+ \CO(\tau (|k|^2 + |k_1|^2)), \quad |\Psi_2 (k,k_1) \tau |k| |k_1| |\le 1, \  k,k_1 \in \Z.
\end{align*}
Namely, these three filter functions preserve the low regularity error structure of $\CO(\tau \partial_x^2 u)$, and provides the following stabilized version of the scheme \eqref{schemeKdV} which is given in physical space by,
\begin{align*}
u^{\ell+1} &=  u^{\ell} * \tilde{u}^\ell -i\tau (e^{i\tau \partial_x^2} - 1) ( u^\ell * \tilde{u}^\ell)^2 \\
& - \frac{\tau}{2} \Big((e^{\tau^{\frac{1}{2}}\partial_x/2}-e^{-\tau^{\frac{1}{2}}\partial_x/2}) ( u^{\ell} * \tilde{u}^{\ell})\Big)  \Big(\sum_{{k_1} \in \Z^d}(e^{\tau^{\frac{1}{2}}|k_1|/2}-e^{-\tau^{\frac{1}{2}}|k_1|/2}) |u^\ell_{k_1}|^2\Big)   \\
&+ \frac{\tau}{2} \left( (e^{\tau^{\frac{1}{2}}\partial_x/2}-e^{-\tau^{\frac{1}{2}}\partial_x/2}) u^{\ell} * (e^{\tau^{\frac{1}{2}}\partial_x/2}-e^{-\tau^{\frac{1}{2}}\partial_x/2}) \tilde{u}^{\ell} \right)  ((u^{\ell}_0)^2 + (\overline{u^{\ell}_0})^2). 
\end{align*}
\end{remark}

\begin{proof}
We proceed as for the Nonlinear Schrödinger equation. The construction of the schemes is again based on the general framework~\eqref{genscheme}. We consider for $r = 0,1$ the expansion,
\begin{equs}\label{kdvU}
V_{k}^{n,r}(\tau, v) & = \sum_{F = T_1 \cdot T_2 \in \CG^{r,k}_0(R)} m_F \frac{\bar{\Upsilon}^{p}(T_1)(v) \, \Upsilon^{p}(T_2)(v)}{S(T_1) S(T_2)} \\ & \mathcal{Q}_{\leq r +1} \left( \bar{\Pi}^{n,r} T_1 \Pi^{n,r} T_2 \right)(\tau) .
\end{equs} 
For the first-order scheme ($r=0$) we need only to consider the single paired decorated forest:
\begin{equs}\label{pair_KDVo1}
\CG^{0,k}_0(R) = \lbrace  T_0 \cdot T_0 \rbrace, \quad 
T_0 = \begin{tikzpicture}[scale=0.2,baseline=-5]
\coordinate (root) at (0,1);
\coordinate (tri) at (0,-1);
\draw[kernels2] (tri) -- (root);
\node[var] (rootnode) at (root) {\tiny{$ k $}};
\node[not] (trinode) at (tri) {};
\end{tikzpicture}.
 \end{equs}
Indeed, we recall the first order expansion \eqref{o1Exp}, where for the KdV equation we have,
\begin{align*}
u^0(\tau,v^{\eta}) = e^{- \tau \partial_x^3 } v^{\eta}, \quad u^1(\tau,v^{\eta})= e^{-\tau\partial_x^3}\int_0^{\tau} e^{\zeta \partial_x^3}\partial_x (e^{-\zeta \partial_x^3}v^{\eta})^2 d\zeta.
\end{align*}
Given that for the KdV equation the term $  {\mathbb{E}}[\overline{u^0_k(\tau,v^{\eta})}u^1_k(\tau,v^{\eta})]$ consists of a product of three Gaussians, we have that $  {\mathbb{E}}[\overline{u^0_k(\tau,v^{\eta})}u^1_k(\tau,v^{\eta})]$ is equal to zero. 
Hence, the only term to compute in the expansion \eqref{o1Exp} is
 $$  {\mathbb{E}}[\overline{u^0_k(\tau,v^{\eta})}u^0_k(\tau,v^{\eta})],$$ and the pairing to consider is \eqref{pair_KDVo1}.

A straightforward calculation yields,
\begin{equs}
S(T_0) & = 1, \quad 
\Upsilon^{p}(T_0)(v)  = v_{k}, \quad m_{T_0 \cdot T_0} = 1,
\end{equs}
and
\begin{equs}
(\Pi^{n,0} T_0)(\tau) = e^{-i \tau k^{3}}.
\end{equs}
In the end, the first order scheme takes on the simple form:
\begin{equs}
V_{k}^{n,0}(\tau, v) = |v_k|^2. 
\end{equs}
One can notice that the only approximation error made comes from the first-order truncation of the tree series \eqref{general_scheme_2}. Namely, by noting that $\alpha = 1$, it follows from \eqref{app1} that $(V_k^{n,0} - V_k)(\tau,v) = \CO(\tau^2 k^2 v_k \bar{v}_k)$. Therefore, in physical space the local error is of order $\CO(\tau^2(\partial_x v * \partial_x \tilde{v}) )$, which solely requires one additional derivative on the initial data.

For a second order scheme, one has to compute more terms. The set of paired decorated forests is given by:
\begin{equs}
\CG^{1,k}_0(R) & = \lbrace  T_0 \cdot T_0,  T_1 \cdot T_1,  T_0 \cdot T_2, T_0 \cdot T_3, \, \, k_i \in \Z^{d} \rbrace, 
\\  T_1 &  =   
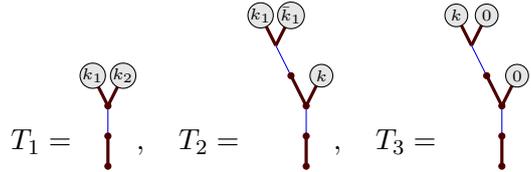
\begin{tikzpicture}[scale=0.2,baseline=-5]
\coordinate (root) at (0,2);
\coordinate (tri) at (0,0);
\coordinate (tri1) at (0,-2);
\coordinate (t1) at (-1,4);
\coordinate (t2) at (1,4);
\draw[kernels2] (t1) -- (root);
\draw[kernels2] (t2) -- (root);
\draw[kernels2] (tri) -- (tri1);
\draw[symbols] (root) -- (tri);
\node[not] (rootnode) at (root) {};
\node[not] (trinode) at (tri) {};
\node[not] (trinode) at (tri1) {};
\node[var] (rootnode) at (t1) {\tiny{$ k_1 $}};
\node[var] (trinode) at (t2) {\tiny{$ k_2 $}};
\end{tikzpicture}, \quad T_2 = \begin{tikzpicture}[scale=0.2,baseline=-5]
\coordinate (root) at (0,2);
\coordinate (tri1) at (0,-2);
\coordinate (tri) at (0,0);
\coordinate (t1) at (-1,4);
\coordinate (t11) at (-2,6);
\coordinate (t12) at (-3,8);
\coordinate (t13) at (-1,8);
\coordinate (t2) at (1,4);
\draw[kernels2] (t11) -- (t13);
\draw[kernels2] (t11) -- (t12);
\draw[kernels2] (t1) -- (root);
\draw[kernels2] (tri) -- (tri1);
\draw[symbols] (t1) -- (t11);
\draw[kernels2] (t2) -- (root);
\draw[symbols] (root) -- (tri);
\node[not] (rootnode) at (root) {};
\node[not] (trinode) at (tri) {};
\node[not] (trinode) at (tri1) {};
\node[not] (trinode) at (t1) {};
\node[var] (rootnode) at (t12) {\tiny{$ k_{\tiny{1}} $}};
\node[var] (rootnode) at (t13) {\tiny{$ \bar{k}_{\tiny{1}} $}};
\node[var] (trinode) at (t2) {\tiny{$ k $}};
\end{tikzpicture}, \quad T_3  = \begin{tikzpicture}[scale=0.2,baseline=-5]
\coordinate (root) at (0,2);
\coordinate (tri1) at (0,-2);
\coordinate (tri) at (0,0);
\coordinate (t1) at (-1,4);
\coordinate (t11) at (-2,6);
\coordinate (t12) at (-3,8);
\coordinate (t13) at (-1,8);
\coordinate (t2) at (1,4);
\draw[kernels2] (t11) -- (t13);
\draw[kernels2] (t11) -- (t12);
\draw[kernels2] (t1) -- (root);
\draw[kernels2] (tri) -- (tri1);
\draw[symbols] (t1) -- (t11);
\draw[kernels2] (t2) -- (root);
\draw[symbols] (root) -- (tri);
\node[not] (rootnode) at (root) {};
\node[not] (trinode) at (tri) {};
\node[not] (trinode) at (tri1) {};
\node[not] (trinode) at (t1) {};
\node[var] (rootnode) at (t12) {\tiny{$ k $}};
\node[var] (rootnode) at (t13) {\tiny{$ 0 $}};
\node[var] (trinode) at (t2) {\tiny{$ 0 $}};
\end{tikzpicture}.
\end{equs} 
A straightforward computation gives:
\begin{equs}
 S(T_1)& = 2, \quad S(T_2) = S(T_3)= 2, \quad  
\Upsilon^{p}(T_1)(v)  =  v_{k_1} v_{k_2},  \\ \Upsilon^{p}(T_2)(v)&  =  |v_{k_1}|^2 v_{k}, \quad \Upsilon^{p}(T_3)(v)  =  v^2_{0} v_{k}
\\ m_{T_1 \cdot T_1} & = 2, \quad m_{T_0 \cdot T_2} = 1, \quad m_{T_0 \cdot T_3} = 2.
\end{equs}
Further, due to cancellations in the resonance, the integrals encoded by the trees $T_2$ and $T_3$ have a null resonance structure, and hence it follows that
\begin{equs}
(\Pi^{n,1} T_0)(\tau) & = e^{-i \tau k^{3}},  \quad 
(\Pi^{n,1} T_2)(\tau)  = - \frac{\tau^2}{2} k_1 k  e^{-i \tau k^{3}}, \\ (\Pi^{n,1} T_3)(\tau) & = - \frac{\tau^2}{2} k^2 e^{-i \tau k^{3}}.
\end{equs}
 For the computation of $ (\Pi^{n,1} T_1)(\tau) $, we proceed by a full Taylor expansion which yields
\begin{equs}
 \left( \Pi^{n,1} T_1  \right)(\tau)
 & =  e^{-i \tau k^3} \mathcal{K}_{(\Labhom_2,1)}^{k,1} \left( e^{i \xi (-k_1^3 - k_2^3)} ,n\right)(\tau)
 \\ & = e^{-i \tau k^3} \left( i \tau k  + \frac{3}{2} \tau^2 k_1 k_2 k^2 \right),
\end{equs}
where we used the fact that the resonance structure is given by,
\begin{equs}
 P_{(\Labhom_2,0)}(k) - k_1^3-k_2^3  =  k^3 - k_1^3 - k_2^3 = 3 k_1 k_2 (k_1+k_2) = 3k_1k_2k.
\end{equs}
Then, by taking the product and making the truncation, we obtain:
\begin{equs}
\CQ_{\leq 2} \left( \bar{\Pi}^{n,1} T_1  \right)(\tau) \left( \Pi^{n,1} T_1  \right)(\tau) & = \CQ_{\leq 2} \left((-i \tau k  +\frac{3}{2} \tau^2 k_1 k_2 k^2) (i \tau k  +\frac{3}{2} \tau^2 k_1 k_2 k^2) \right)
\\ & =  \tau^2 k^2.
\end{equs}
The local error introduced is of order $ \mathcal{O}(\tau^{3} k^4) $ which corresponds to two derivatives due to the convolution structure. In the end, the scheme in Fourier space is given by:
\begin{equs}
V_{k}^{n,1}(\tau, v)&  = |v_k|^2 + \tau^2 k^2 \sum_{k = k_1 + k_2}  |v_{k_1}|^2 |v_{k_2}|^2   
\\ & - \sum_{k_1 \in \Z^d} \frac{\tau^2}{2} k_1 k |v_{k_1}|^2 |v_k|^2 -  \frac{\tau^2}{2} k^2 |v_k|^2 (v_0^2 + \overline{v_0}^2).
\end{equs}
By writing the above scheme in physical space we recover the scheme \eqref{schemeKdV}, which induces a local error of $\CO(\tau^3 \partial_x^2 v)$.
\end{proof}

\appendix
\endappendix

\end{document}